\documentclass{amsart}

\usepackage[text={152mm,240mm},centering]{geometry}
\usepackage[mathscr]{eucal}
\usepackage{color}
\numberwithin{equation}{section}
\usepackage[all]{xy}
\usepackage{amsfonts}   
\usepackage{amsmath}
\usepackage{amsthm}
\usepackage{amssymb}
\usepackage{latexsym}

\newtheorem{theorem}{Theorem}[section]
\newtheorem{lemma}{Lemma}[section]
\newtheorem{proposition}{Proposition}[section]
\newtheorem{corollary}{Corollary}[section]
\newtheorem{conjecture}{Conjecture}[section]

\theoremstyle{definition}
\newtheorem{definition}{Definition}[section]
\newtheorem*{remark}{Remark}
\newtheorem*{notation}{Notation}
\newtheorem{example}{Example}[section]

\newcommand{\la}{\label}
\def\c{\mathbb{C}}
\def\e{\mathbf{e}}
\def\f{\mathbf{f}}
\def\g{\mathbf{g}}
\def\v{\mathbf{v}}
\def\u{\mathbf{u}}
\def\h{\mathbf{h}}
\def\O{\mathcal{O}}
\def\N{\mathbb{N}}
\def\z{\mathbb{Z}}
\def\CC{\mathcal{C}}
\def\SS{\mathcal{S}}

\newcommand{\Rep}{{\mathtt{Rep}}}
\newcommand{\Hom}{{\mathrm{Hom}}}
\newcommand{\End}{{\mathrm{End}}}
\newcommand{\GL}{{\mathtt{GL}}}
\newcommand{\SL}{{\mathrm{SL}}}
\newcommand{\Mod}{{\tt{Mod}}}

\newcommand{\Pic}{{\mathrm{Pic}_{\c}}}
\newcommand{\Aut}{{\mathrm{Aut}_{\c}}}

\newcommand{\Autg}{{\mathrm{Aut}_{\Gamma}}}

\newcommand{\Tr}{{\mathrm{Tr}}}
\newcommand{\Ker}{{\mathrm{Ker}}}

\newcommand{\onto}{\,\,\twoheadrightarrow\,\,}

\begin{document}

\begin{abstract}
Let $O_\tau(\Gamma)$ be a family of algebras \textit{quantizing} the coordinate ring 
of $\c^2/\Gamma$, where $\Gamma$ is a finite subgroup of $\SL_2(\c)$, and 
let $G_\Gamma$ be the automorphism group of $O_\tau$.
We study the natural action of $G_\Gamma$ on the space
of right ideals of $O_\tau$ (equivalently, finitely generated rank 
$1$ projective $O_\tau$-modules). It is known that the later 
can be identified with disjoint  union of algebraic (quiver) varieties,
and this identification is $G_\Gamma$-equivariant.
In the present paper, when $\Gamma \cong \z_2$, we show that the $G_{\Gamma}$-action
on each quiver variety is transitive.
We also show that the natural embedding of
$G_\Gamma$ into $\Pic(O_\tau)$, 
the Picard group of $O_\tau$, is an isomorphism.
These results are used to prove that there are countably many
non-isomorphic algebras Morita equivalent to $O_\tau$,
and explicit presentation of these algebras are given.
Since algebras $O_\tau(\z_2)$
are isomorphic to primitive factors of $U(sl_2)$,
we obtain a complete description of algebras Morita equivalent
to primitive factors. A structure of the group $G_{\Gamma}$, where $\Gamma$ is an arbitrary cyclic group,  
is also investigated. Our results generalize earlier results obtained for the (first) Weyl algebra $A_1$ 
in \cite{BW1, BW2} and \cite{S}.
 \end{abstract}
\title{ AUTOMORPHISMS AND IDEALS OF NONCOMMUTATIVE DEFORMATIONS OF $\c^2/\z_2$}
\author[X. Chen]{Xiaojun Chen}

\address{X. Chen and F. Eshmatov : Department of Mathematics, Sichuan University, Chengdu 610064 P. R. China}
\email{xjchen@scu.edu.cn, olimjon55@hotmail.com}

\author[A. Eshmatov]{Alimjon Eshmatov}
\address{A. Eshmatov: Department of Mathematics, Cornell University, Ithaca, NY 14850, USA}
\email{aeshmat@math.cornell.edu}

\author[F. Eshmatov]{Farkhod Eshmatov}
\author[V. Futorny]{Vyacheslav Futorny}
\address{V. Futorny: Instituto de Matem\'atica e Estat\'istica, Universidade de Sa\~o Paulo, Caixa Postal 66281, Sa\~o Paulo, CEP 05315-970, Brasil}
\email{futorny@ime.usp.br}


\maketitle

\section{Introduction}
Let $\c \langle x, y \rangle $ be the free associative algebra on two generators.
Then a finite subgroup $\Gamma \subset \SL_2(\c)$ acts naturally on $\c \langle x, y \rangle $,
and we can form the crossed product  $R:=\c \langle x, y \rangle \ast \Gamma$.
For each $\tau \in Z(\c\Gamma)$, the center of the group  algebra $\c\Gamma$, let
$$  \mathcal{S}_\tau (\Gamma) \, := \, R / (xy-yx -\tau)\,  , 
\quad  O_\tau(\Gamma) = e \mathcal{S}_\tau e \, , $$ 
where $e$ is the symmetrizing idempotent $\sum_{g\in\Gamma} g/|\Gamma|$ in 
$\c\Gamma \subset \mathcal{S}_{\tau}$.  These algebras were introduced in \cite{CBH}.
One can easily verify that the associated
graded algebra of $O_\tau$ with respect to its natural filtration is $\c[x,y]^{\Gamma}$. 
Hence algebras $O_{\tau}$ can be viewed as a {\it quantization} of the coordinate ring of
the
classical Kleinian singularity $\mathbb C^2/\!\!/\Gamma$.
When $\Gamma \cong \z_2$ the algebra $O_\tau$ is isomorphic to
a primitive factor of the enveloping algebra of the Lie algebra $sl_2(\c)$. On the other
hand when $\tau=1$ then $O_\tau$ is isomorphic to the algebra of $\Gamma$-invariant 
elements of $A_1$.

Let $\mathcal{R}^\tau_{\Gamma}$ be the set of isomorphism classes of right ideals of $O_\tau$.
Since, for generic values of $\tau$, $O_\tau$ is a simple hereditary domain,
the set $\mathcal{R}^\tau_{\Gamma}$ can be identified with that of isomorphism classes of
finitely generated rank one projective $O_\tau$-modules. 
In \cite{BGK}, a bijective map $\Omega$ was constructed 
from $\mathcal{R}^\tau_{\Gamma}$ to the disjoint union of quiver varieties 
$\mathfrak{M}^\tau(Q_\Gamma)$, where $Q_\Gamma$ is the quiver
associated to $\Gamma$ under the McKay correspondence. 
Later in \cite{E}, the third author gave another more explicit
construction of $\Omega$. Moreover, he showed that the group 
$G_{\Gamma}:=\Aut(O_\tau)$ acts naturally on each  
$\mathfrak{M}^\tau(Q_\Gamma)$ and proved that $\Omega$ is a 
$G_\Gamma$-equivariant bijection. 

In the case when $\Gamma \cong \z_2$,
the above  $G_{\Gamma}$-equivariant bijective map is given by
\begin{equation}
\label{decompquiver}
\Omega : \,\mathfrak{M}^\tau(Q_\Gamma)\ = \
\bigsqcup_{\substack{ \epsilon=0,1\\ (m,n) \in L_{\epsilon}} } \mathfrak{M}^\tau_{\Gamma}(m,n; \epsilon) 
 \, \longrightarrow \,  \mathcal{R}^\tau_{\Gamma}\, ,
\end{equation}
where $L_{\epsilon}$ is a subset $(\z_{\ge 0})^2$ of dimension vectors 
of quiver representations $Q_{\Gamma}$ and $\epsilon=0$ or $1$ are indices
 corresponding to the trivial and the sign representations of $\z_2$.
Since the group $G_{\Gamma}$ acts on each $\mathfrak{M}^\tau_\Gamma(m,n; \epsilon)$,
we have the corresponding decomposition of $\mathcal{R}^\tau_{\Gamma}$:
$$  \mathcal{R}^\tau_{\Gamma} = \bigsqcup_{\substack{ \epsilon=0,1\\ (m,n) \in L_{\epsilon}} }
 \mathcal{R}^\tau_{\Gamma}(m,n; \epsilon)\, , $$
where $ \mathcal{R}^\tau_{\Gamma}(m,n; \epsilon):=\Omega(\mathfrak{M}^\tau_{\Gamma}(m,n; \epsilon))$.

The initial motivation for \cite{BGK} and \cite{E} came from  a series of papers \cite{BW1,BW2} by  
 Y. Berest and G. Wilson. They showed that the space $\mathcal{R}$
 of isomorphism classes of right ideals in  $A_1$ are parametrized by the disjoint union
  of the {\it Calogero-Moser } algebraic varieties
 $\mathcal{C}= \sqcup_{n\ge 0} \mathcal{C}_n$:
$$  \mathcal{C}_n:= \{ (X,Y) \in \mathtt{Mat}_{n}(\c) \times \mathtt{Mat}_{n}(\c) 
\, : \, \mathtt{rk}([X,Y] +I_n)=1\} / \mathtt{PGL}_n(\c)\, ,
$$
 where $\mathtt{PGL}_n(\c)$ acts on $(X,Y)$ by simultaneous conjugation.
 The space $\mathcal{C}_n$ is a basic example of quiver varieties:
it is the quiver variety corresponding to $n$-dimensional representations of
the  quiver with one vertex and one loop.
 They defined an algebraic action of the group $G:=\Aut(A_1)$
 on each $\mathcal{C}_n$ and showed that the bijective correspondence between $\mathcal{R}$
 and $\mathcal{C}$ is $G$-equivariant. Moreover, they proved that this action
 is transitive, that is, each $\mathcal{C}_n$ is a $G$-orbit.

In \cite{BL}, motivated by 
the transitivity of the $G$-action on $\mathcal{C}_n$,
R. Bocklandt and L. Le Bruyn proposed the following conjecture. 
For any quiver $Q$, they defined the algebra $A_{Q}:=\c[N_Q] \otimes \c\bar{Q}$, where
$N_Q$ is the \textit{necklace} Lie algebra of $Q$ and $\c\bar{Q}$ is the path  algebra
of the double of $Q$. Then they claimed that the group $\mathrm{Aut}_{\omega}(A_Q)$, the automorphism 
group of $A_Q$ preserving the symplectic element $\omega \in \c\bar{Q}$, 
acts transitively on quiver varieties of $\alpha$-dimensional representations of $Q$ 
for an appropriate choice of the dimension vector $\alpha$.
Since $G_{\Gamma}$ is a subgroup of $\mathrm{Aut}_{\omega}(A_{Q_{\Gamma}})$,
our first result is a proof the Bocklandt-Le Bruyn conjecture for $\z_2$ 
cyclic quiver varieties:
\begin{theorem}
\label{intrth1}
Let $\Gamma \cong \z_2$. Then $G_{\Gamma}$ acts transitively on each 
$\mathfrak{M}^\tau_{\Gamma}(m,n; \epsilon)$.
\end{theorem}
The proof of transitivity in Theorem~\ref{intrth1} has some advantages over that of Berest-Wilson.
Most notably, it is purely geometric and does not use any fact from integrable systems.
In fact, we will first give an alternative proof for transitivity of the $G$ action on $\mathcal{C}_n$.

Let us briefly outline our approach.  We will show that  $\mathfrak{M}^\tau_{\Gamma}(m,n; \epsilon)$
is $G_{\Gamma}$-{\it flexible}, i.e., the (co)tangent space at each point of  this variety
 is spanned by the (co)tangent vectors to the orbits of $G_{\Gamma}$.
 To this end, we construct a family of functions $\{f_n\}_{n \geq 0}$ such that $\{d f_n\}_{n \geq 0}$
 span the cotangent space at each point. More precisely, we consider flows of
 two one-parameter subgroups of $G_{\Gamma}$ on $\mathfrak{M}^\tau_{\Gamma}(m,n; \epsilon)$.
Then, since the action of $G_{\Gamma}$ is symplectic, we show that  these are Hamiltonian flows of a family 
functions generating the algebra of functions $\O ( \mathfrak{M}^\tau_{\Gamma}(m,n; \epsilon))$, which
implies the $G_{\Gamma}$-flexibility. Finally, using  the fact that  
 $\mathfrak{M}^\tau_{\Gamma}(m,n; \epsilon)$ is connected, we obtain that the later is  a single orbit. 
 
 In general, we expect:
 \begin{conjecture}
 For $\Gamma=\mathbb{Z}_m (m \ge 3)$, 
 $G_{\Gamma}$ acts transitively on each 
$\mathfrak{M}^\tau_{\Gamma}$.
 \end{conjecture}
 
The transitivity of the $G$-action on $\mathcal{C}_n$  in combination with Stafford's theorem 
$G \cong \Pic(A_1)$ (see \cite[Corollary E]{S})  gives the following remarkable result. 
 Let $D$ be a domain, which is Morita equivalent to $A_1$. 
 Then all such algebras $D$ are classified, up to algebra isomorphism, by a single integer $n\ge 0$; 
 the corresponding isomorphism classes are represented by the endomorphism rings
 $D_n:=\End_{A_1}(P_n)$ of the ideal $P_n:=x^{n+1}A_1+(xy+n)A_1$.
 Note here $A_1$ appears as the first member in the family $\{D_n\}_{n \ge 0}$. 
 
Our next goal in this paper is to describe  algebras Morita equivalent to $O_\tau(\z_2)$.
For that one needs to describe the
orbits of $\Pic(O_{\tau})$ group on the space $\mathcal{R}^\tau_{\Gamma}$.
Since $G_{\Gamma}$ is only a subgroup of $\Pic(O_{\tau})$, the orbits of the later
contain $G_{\Gamma}$-orbits. By Theorem~\ref{intrth1}, the $G_{\Gamma}$-orbits on
$\mathcal{R}^\tau_{\Gamma}$ are exactly 
$ \mathcal{R}^\tau_{\Gamma}(m,n; \epsilon)$
and for any $P, Q \in  \mathcal{R}^\tau_{\Gamma}(m,n; \epsilon)$, we have
an isomorphism of algebras
\begin{equation}
\la{isoideal}
 \End_{O_\tau}(P) \cong \End_{O_\tau}(Q) \, .
 \end{equation}

Note that $\mathcal{R}^\tau_{\Gamma}(0,0; 0)$ consists of a single point, the 
isomorphism class of the cyclic ideal $O_\tau$,
and clearly its endomorphism algebra is isomorphic to $O_\tau$. So it is natural 
to ask if endomorphism rings of two ideals chosen from distinct $G_{\Gamma}$-orbits 
are non-isomorphic. This is equivalent to the question whether $G_\Gamma$
and $\Pic(O_\tau)$-orbits are the same.
To answer this question, we recall that
the ideal $P_n$ of $A_1$ defined above is an ideal corresponding to a special point in $\mathcal{C}_n$.
More explicitly, it corresponds to a $\c^*$-fixed point in  $\mathcal{C}_n$ under a subgroup
action of $\c^*\subset G$.
Similarly, $G_{\Gamma}$ has a subgroup isomorphic to $\c^*$, and it is direct to show that there are only finitely many
$\c^*$-fixed points on $\mathfrak{M}^\tau_{\Gamma}(m,n; \epsilon)$. For each triple $(m,n; \epsilon)$, 
we construct explicitly a $\c^*$-fixed point of  $\mathfrak{M}^\tau_{\Gamma}(m,n; \epsilon)$.
Let $P^{(\epsilon)}_{m,n} \in \mathcal{R}^\tau_{\Gamma}(m,n; \epsilon)$ be the image of this fixed
point under $\Omega$. Then we shall prove:

\begin{theorem}
\label{intrth2}
 $\End_{O_\tau}(P^{(\epsilon)}_{m,n}) \ncong O_\tau$ unless $(m,n; \epsilon)=(0,0; 0)$.
\end{theorem}

The proof of Theorem~\ref{intrth2} is similar to that of Smith 
\cite{Sm1}, where he showed that
$\End_{A_1}(P_1) \ncong A_1$. We will give an explicit presentation of the ideals
$P^{(\epsilon)}_{m,n}$ in terms of generators of $O_\tau$ and an explicit
description of $\End_{O_\tau}(P^{(\epsilon)}_{m,n})$ as a subring of the field of fractions
of $O_\tau$. Using this presentation, we show the nonexistence of an isomorphism.

It follows immediately from Theorem~\ref{intrth2} that \eqref{isoideal} holds 
if and only if $P, Q \in  \mathcal{R}^\tau_{\Gamma}(m,n; \epsilon)$, and
hence we have proved:

\begin{corollary}
\la{intcor0}
Let $D$ be a domain Morita equivalent to $O_\tau$. Then there exists a unique triple 
$(m,n; \epsilon) \in L_{\epsilon} \times \z_2$  such that $ D\,  \cong\,  \End_{O_\tau}(P^{(\epsilon)}_{m,n})$. 
\end{corollary}

Morita equivalence of primitive factors of $U(sl_2)$ was studied by Hodges in \cite{H1}.
Explicitly, he classified them up to  Morita equivalence using the Hattori-Stallings trace.
Our result generalizes that by giving a description of all algebras equivalent to
 $O_\tau$ not only among primitive factors.

By Corollary~\ref{intcor0},  the $G_{\Gamma}$-orbits and $ \Pic(O_\tau)$-orbits on 
$ \mathcal{R}^\tau_{\Gamma}$
coincide, and therefore, we can prove the following generalization 
of Stafford's result for $A_1$ (see \cite[Corollary E]{S}):

\begin{theorem}
\label{intrth2.5}
The group monomorphism $G_{\Gamma} \to \Pic(O_\tau)$ is an isomorphism.
\end{theorem}
The second part of  \cite[Corollary E]{S} states that if $D$ is Morita equivalent to $A_1$ but $D \ncong A_1$
then $\Pic(D)/\Aut(D)$ is an infinite coset space. This means the automorphism group is an invariant
distinguishing $A_1$ from $D$. In fact, Stafford suggested a conjecture that this might be 
the case for a large class of algebras similar to $A_1$ (see \cite[p.625]{S}).
However, the algebra $O_\tau$ provides a counterexample for this. We will describe
 triples $(n,m; \epsilon) \neq (0,0;0)$ such that 
$\Aut(\End(P^{(\epsilon)}_{m,n})) \cong \Pic(\End(P^{(\epsilon)}_{m,n}))$, which,
by Morita equivalence,
will imply $\Aut(\End(P^{(\epsilon)}_{m,n})) \cong \Aut(O_\tau)$.

In the proof of Theorem~\ref{intrth1}, we have used generators of 
the group $G_{\Gamma}$ discovered by Bavula and Jordan in \cite{BJ}. For $\Gamma \cong \z_m$, 
they proved that  $G_{\Gamma}$ is generated by three abelian subgroups $\Theta_{\lambda}$,
$\Psi_{n,\lambda}$ and $\Phi_{n,\lambda}$, where $\lambda \in \c^*$ and 
$n \in \z_{\ge0}$ (see \eqref{gens1}-\eqref{gens3} for definition).  
For $\Gamma \cong \z_2$, these generators were originally found in Fleury \cite{F}, who also showed
that  $G_{\Gamma}$ can be presented as the coproduct $\SL_2(\c) \ast_{U}
\langle \Theta_{\lambda}, \Phi_{n,\lambda} \rangle$, where $U:=\SL_2(\c)
\cap \langle \Theta_{\lambda}, \Phi_{n,\lambda} \rangle$.
We recall that $\Aut(A_1)$ admits similar coproduct structure (see \cite{A}). 

So our goal here is twofold. First, we would like to extend a coproduct structure for $G_{\Gamma}$
when $\Gamma \cong \z_m (m \ge 3)$. Second, we would like clarify the relationships between groups 
$G_\Gamma=\Aut(O_\tau)$ and $\mathrm{Aut}_{\Gamma}(S_\tau)$. 

\begin{theorem}
\label{intrth3}
We have
\begin{enumerate}
\item[1.]
For $\Gamma \cong \z_m (m \ge 3)$, 
\begin{equation*}
\la{amal}
 G_{\Gamma} \cong A \ast_{U} B \, ,  
 \end{equation*}
where $A:=\langle \Theta_{\lambda}, \Psi_{n,\lambda} \rangle$, 
$B:=\langle \Theta_{\lambda}, \Phi_{n,\lambda} \rangle$ and $U:=A \cap B$.
\item[2.] For $\Gamma \cong \z_m (m \ge 2)$, the automorphisms  $\Theta_{\lambda}$,
$\Psi_{n,\lambda}$ and $\Phi_{n,\lambda}$ of $O_\tau$ can be lifted to $\Gamma$-equivariant
automorphisms of $S_\tau$ and
$$\mathrm{Aut}_{\Gamma}(\mathcal{S}_\tau) \cong \, H \rtimes  G_\Gamma\, , $$
where $H$ is an abelian subgroup defined in Lemma~\ref{subH}.
\end{enumerate}
\end{theorem}

We would like to finish the introduction by mentioning one interesting consequence
of Theorem~\ref{intrth1} in the theory of integrable systems.  
In the seminal paper \cite{W}, G. Wilson has shown that rational solutions
of \textit {KP hierarchy} can be parametrized by the Calogero-Moser
spaces $\mathcal{C}= \sqcup_{n\ge 0} \mathcal{C}_n$. 
Thus we have an action of the group $G$ on spaces of rational solutions, and
the $G$-transitivity on each of $\mathcal{C}_n$ 
implies that choosing one solution from $\mathcal{C}_n$
one can obtain any other solution 
by applying an appropriate element of $G$.
Recently, in \cite{CS}, a generalization of the KP hierarchy
associated to the cyclic quiver has been introduced.
By extending the result of \cite{W}, the authors have shown that
rational solutions are parametrized by the corresponding quiver 
varieties. Hence Theorem~\ref{intrth1} implies that
any rational solution from $\mathfrak{M}^\tau_{\Gamma}(m,n; \epsilon)$
can be obtained from any given solution 
by applying an appropriate element of $G_{\Gamma}$ for $\Gamma \cong \z_2$.

The paper is organized as follows. In Section~\ref{sec2}, we introduce the notations,
review some basic facts about algebras $\mathcal{S}_\tau$ and $O_\tau$ as well as
their automorphism groups, and give a proof of Theorem~\ref{intrth3}.
In Section~\ref{sec3} we explain the construction of the bijective map $\Omega$.
The main result is Theorem~\ref{thgeqivbij}, which shows how
to construct an ideal for a given point of the quiver variety.
The theorem also provides an important invariant of the corresponding ideal, 
$\kappa \in Q(O_\tau)$, where $ Q(O_\tau)$ is the field of fractions of $O_\tau$.
This invariant distinguishes ideals up to isomorphism.
In Section~\ref{sec4}, using $\kappa$, we
provide a natural generating set for the coordinate ring of the
quiver variety $\mathfrak{M}^{\tau}_{\z_2}(m,n; \epsilon)$. 
In Section~\ref{sec5}, we give a proof of Theorem~\ref{intrth1}.
In Section~\ref{secfixpoints}, we construct $\c^{\ast}$-fixed points 
of $\mathfrak{M}^{\tau}_{\z_2}(m,n; \epsilon)$. In Sections~\ref{sec7}-\ref{sec8}, we
compute $\kappa$'s for $\c^{\ast}$-fixed points,
their corresponding ideals and endomorphism rings.
Finally, in Section~\ref{sec9}, we give proofs of Theorems~\ref{intrth2}
and \ref{intrth2.5}. It also discusses a counterexample
to Stafford's conjecture.

\section{Automorphism groups of $\mathcal{S}_\tau$ and $O_\tau$} 
\label{sec2}

\subsection{Generalities}

In this section we collect various facts about algebras $\mathcal{S}_\tau(\Gamma)$ and $O_\tau(\Gamma)$
that we will need. Some of these facts will be used without comment in the body of the paper.
All algebras will be considered over the field $\c$.

As mentioned above, we will be dealing with cyclic $\Gamma$'s. 
Thus let $m \ge 1$ and assume $\Gamma \cong \z_m$.
Fix an embedding $\Gamma  $ into $\SL_2(\c)$  so that a generator $g$ of $\Gamma$
acts on $x$ and $y$ by $\epsilon$ and $\varepsilon^{-1}$ respectively, where $\varepsilon$ is 
a primitive  $m$-th root of unity. Then $\mathcal{S}_{\tau}$ is the quotient of  $\c \langle x,y \rangle
 \ast \Gamma$ subject to the
following relations
\begin{equation}
\la{rel1}
 g^{i} \cdot  x \,  =  \, \varepsilon^{i}\, x \cdot g^{i} \,  , \quad  g^{i} \cdot  y\, =  \, \varepsilon^{-i} \,y \cdot g^{i}
 \quad \mbox{ for  }  \, i=1,\cdots,m-1 \, ,
 \end{equation}
\begin{equation}
 x \cdot y - y \cdot x \, = \, \tau  \, . 
 \end{equation}
One can replace \eqref{rel1} by another set of relations. Indeed, let $e_0, \cdots, e_{m-1}$ be the complete 
set of orthogonal primitive idempotents  of $\c\Gamma$ given by
\begin{equation}
\la{idemp}
e_{i} \, :=\,  \frac{1}{m} \, \sum^{m-1}_{k=0} \varepsilon^{-i k} \, 
g^k \, ,
\end{equation}
where $g$ is a generator of $\Gamma$. Then relation \eqref{rel1} is equivalent to
\begin{equation}
\la{relidemp}
e_i \cdot x \, =  \, x\cdot e_{i+1} \, ,  \quad  e_i \cdot y \, =  \, y\cdot e_{i-1} \, \mbox{ for } \,  i \, ( \mathtt{mod}  \, m) \, .
\end{equation}

The homological and ring-theoretical properties of $O_\tau$ depend on the
values of the parameter $\tau$. By the McKay correspondence we can associate
to the group $\Gamma$ an extended  Dynkin diagram. The group algebra $\c\Gamma$
is then identified with the dual of the space spanned by the simple roots of the corresponding
root system. Following \cite{CBH},  we say that $\tau$ is \textit{regular} if it does
not belong to any root hyperplane in $\c\Gamma$.
In the case when $\tau$ is regular, the algebras $\mathcal{S}_{\tau}$ and $O_\tau$ are
 Morita equivalent, the equivalence between the categories of right modules is given by 
 (see \cite[Theorem~0.4]{CBH}) 
 \begin{equation}
 \la{equiv}
 F\, : \, \Mod(\mathcal{S}_{\tau} ) \, \to \, \Mod(O_\tau) \, , \,
  M \,  \mapsto\,  M \otimes_{\mathcal{S}_{\tau}} \mathcal{S}_{\tau}e\, .
 \end{equation}

Now let $v \in \c[h]$. Then $A(v)$ is the $\c$-algebra generated by $a, b$ and $h$
subject to the relations
\begin{equation}
\la{relAv}
a \cdot h = (h-1) \cdot a \, , \, b\cdot h = (h+1) \cdot b \, , \, b \cdot a = v(h) \, , \, a \cdot b = v(h-1) \, .
\end{equation}
 These algebras have been studied in \cite{Ba}, \cite{H2} and \cite{Sm2}.
 One can easily establish the following relation between algebras $O_\tau$
 and $A(v)$:

\begin{proposition}
\la{prop1}
Let $\Gamma \cong \z_m$ and let $\tau:=\tau_0 e_0+\cdots+ \tau_{m-1} e_{m-1} \in \c\Gamma $ be such that
 $\tau_0+\cdots+ \tau_{m-1} \neq 0$. 
Then  the map $\phi : O_{\tau} \to A(v)$ defined by
$$e x^{m}\,  \mapsto \, b \, , \,  e y^m \mapsto a \, , \, e yx \, \mapsto \, (\tau_0+\cdots+\tau_{m-1}) h$$
induces an algebra isomorphism, where
$$v(h):=(\tau_0+\tau_2+ \cdots
+ \tau_{m-1})^m \prod_{i=0}^{m-1} \bigg( h+\frac{\tau_0+\cdots+\tau_i}{\tau_0+\cdots+\tau_{m-1}}\bigg) \, .$$ 
\end{proposition}

\begin{corollary}
$A_1^{\mathbb Z_{m}} \cong A(v)$ for $v(x) \, =\,  \prod_{i=1}^{m} \, (x+\frac{i}{m}) $.
\end{corollary}

Now, if  $ \deg(v)=2$ then $A(v)$ is isomorphic
to a primitive factor of the universal enveloping algebra of $sl_2(\c)$.  
Indeed, let $U=U(sl_2)$ be the universal enveloping algebra of $sl_2(\c)$, and
 let $F, E, H$ be its standard generators.
Let $\Omega=4FE+H^2+2H$ be the Casimir element. Then primitive factors are of the form 
$U_\alpha=U/(\Omega-\alpha)$ ($\alpha \in \c$). 
If $v(x)=\alpha/4-x-x^2$, then the map
\begin{equation}
\la{eqphicong}
  \phi :  U_\alpha \to A(v) \, , \quad E \, \mapsto \, a, \,  F\, \mapsto \, b , \, H\,  \mapsto \, 2h 
 \end{equation} 
defines an isomorphism of algebras (see e.g. \cite[Example $4.7$]{H2}).
 
 The following result shows when $A(v_1)$ and $A(v_2)$ are isomorphic:
\begin{theorem} \cite[Theorem~$3.28$.]{BJ} 
\la{thBJ1}
Let $v_1,v_2 \in \c[h]$. Then $A(v_1) \cong A(v_2)$ if and only if $v_2(h) =\eta v_1(h+\beta)$
for $\beta \in \c$ and $\eta \in \c^*$.
\end{theorem}
Now we get
\begin{corollary}
If $\Gamma\cong \z_2$ and $\tau=\tau_0e_0+\tau_1e_1$, then $O_\tau \cong U_{\alpha}$ for $\alpha = 1- (\frac{\tau_1}{\tau_0+\tau_1})^2$.
\end{corollary}
\begin{proof}
By  \eqref{eqphicong} $U_\alpha \cong A( v_1)$, where $v_1=-(h+\frac{1+\sqrt{1-\alpha}}{2})(h+\frac{1-\sqrt{1-\alpha}}{2})$. By 
Proposition~\ref{prop1} $O_\tau \cong A(v_2)$, where $v_2=(\tau_0+\tau_1)^2(h+1)(h+\frac{\tau_1}{\tau_0+\tau_1})$. Now, we can use
Theorem~\ref{thBJ1} with $ \beta=\frac{1-\sqrt{1-\alpha}}{2}$ and $\eta=-(\tau_0+\tau_1)^2$ to prove our statement.
\end{proof}

\subsection{Proof of Theorem~\ref{intrth3}.}

Let $\Aut(\SS_\tau)$  be the group of $\c$-algebra automorphisms 
of $\SS_\tau$ and let $\Autg (\SS_\tau)$ be the subgroup
of $\Gamma$-invariant automorphisms. Let $\Autg_{,\omega}(R )$ 
be the group of $\Gamma$-invariant automorphisms of $R$ preserving $\omega:=xy-yx$.
Then there is a natural group epimorphism  
\begin{equation}
\la{varphi}
\varphi \, : \, \Autg_{,\omega}(R) \, \to \, \Autg (\SS_\tau).
\end{equation}
We also have the following group homomorphism
\begin{equation}
\mu \, :\, \Autg (\SS_\tau) \to \Aut(O_\tau)
\, , \quad \mu(\sigma)(ebe)\,  = \, e \sigma (b) e \, .
\end{equation}
Composing $\mu $ with the isomorphism in Proposition~\ref{prop1} we get a group homomorphism
\begin{equation}
\rho \,  : \,  \Autg (\SS_\tau) \to \Aut(A(v)) \, .
\end{equation}
We now present generators of the group $\Aut(A(v))$ discovered in \cite{BJ}.
Let $\sigma$ be a $\c$-linear automorphism of $\c[h]$ such that $\sigma(h)=h-1$.
For $k \in \mathbb{N}$, let $\triangle_{k}$ be the linear map $\sigma^k-1:\c[h] \to \c[h]$.
It is easy to check that $\triangle_k$ defines a  $\sigma^n$-derivation, i.e.,
$$ \triangle_k (fg)\, = \, \triangle_k(f) \, g + \sigma^k(f)\, \triangle_k(g),$$
for all $f,g \in\c[h]$.
 Let $m:=\deg(v)$, $n \ge 0$ be an integer and $\lambda \in \c$. Then the following families of automorphisms of $A(v)$:
 \begin{eqnarray}
 && \Theta_{\lambda}\,: \, a \mapsto \lambda^m \, a\, , \, h \mapsto h \, , \, b \mapsto \lambda^{-m}\, b \,,\\
 && \Psi_{k,\lambda} \, : \, a \mapsto a \, , \, h \mapsto  h - k \, \lambda \, a^k \, , \,
 b \mapsto b+ \sum_{i=1}^{m} \frac{ \lambda^i}{i!} \, \triangle^{i}_{k} (v)\, a^{ik-1}\,,\\
 && \Phi_{k,\lambda}  \, : \, a \mapsto a + \sum_{i=1}^{m} \frac{(- \lambda)^i}{i!} \, 
 \triangle^{i}_{k} (v)\, b^{ik-1}\, , \, h \mapsto  h + k \, \lambda \, b^k \, , \,
 b \mapsto b
 \end{eqnarray}
were introduced in  \cite{BJ}. Moreover, they proved  (see \textit{ loc. cit.} 
Theorem~$3.29$):

 \begin{theorem}
 \la{BavJor}
 $\Aut(A(v))$ is generated by $ \Theta_{\lambda}$, $\Psi_{k,\lambda} $ and $\Phi_{k,\lambda}$.
 \end{theorem}

  Let $k \in \N$ and $\lambda \in \c^*/\mathbb{Z}_m$. Then the following maps define 
  automorphisms both in  $\Autg (\SS_{\tau})$ and in $\Autg_{,\omega}(R)$:
 \begin{eqnarray}
 \label{gens1}
 && \theta_{\lambda} \, : \, x \mapsto \lambda^{-1} \, x \, , \,  y \mapsto \lambda\, y  \, ,  \quad g \mapsto g; \\
 \label{gens2}
 && \psi_{k,\lambda} \, : \, x \mapsto x - k\, \lambda \, (\tau_0+\cdots+\tau_m) \, y^{km-1} \, , 
 \quad y\mapsto y \, ,  \quad g \mapsto g ;\\
 \label{gens3}
 && \phi_{k,\lambda}  \, : \, x\mapsto x  \, , \quad y \mapsto y + n \, \lambda \, (\tau_0+\cdots+\tau_m) \, x^{km-1}  \, ,  \quad g \mapsto g.
 \end{eqnarray}
Therefore, from now  on, we will use the same notations for these automorphisms whether  
we consider them as elements of $\Autg (\SS_{\tau})$ or $\Autg_{,\omega}(R )$.
We can show
\begin{proposition}
\la{prop2}  
\begin{enumerate}
\item[(i)]
$\rho(\theta_{\lambda})\, =\, \Theta_{\lambda}\, , \,  \rho(  \psi_{k,\lambda})\,  = \,  \Psi_{k,\lambda}\, , \,
 \rho( \phi_{k,\lambda})\, = \, \Phi_{k,\lambda}$ in $\mathrm{Aut}_{\Gamma}(\SS_\tau)$  .
 \item [(ii)]
$\rho \circ \varphi (\theta_{\lambda})\, =\, \Theta_{\lambda}\, , \,  \rho \circ \varphi (  \psi_{k,\lambda})\,  = \, 
 \Psi_{k,\lambda}\, , \, \rho \circ \varphi ( \phi_{k,\lambda})\, = \, \Phi_{k,\lambda} \, .$
 \end{enumerate}
\end{proposition}

\begin{proof}
(i)
It is straightforward to see that $\rho(\theta_{\lambda})\, =\, \Theta_{\lambda}$. 
We will only prove $\rho(  \psi_{k,\lambda})\,  = \,  \Psi_{k,\lambda}$ 
since  the proof for  $\rho( \phi_{k,\lambda})\, = \, \Phi_{k,\lambda}$ is completely analogous. 

By identification $\phi : O_{\tau} \rightarrow A(v)$,  we obtain that $\psi_{k,\lambda}$ induces the automorphism
$$
(a, \, h, \, b) \, \mapsto \, (a  , \,   h - k \, \lambda \, a^k , \,
 e (x+\lambda (\tau_0 + \cdots +\tau_{m-1}) \, y^{mk-1})^m).
$$
To prove $\rho(  \psi_{k,\lambda})\,  = \,  \Psi_{k,\lambda}$ it suffices to show 
\begin{equation}
\label{psieq1}
 e (x+\lambda (\tau_0 + \cdots +\tau_{m-1}) \, y^{mk-1})^m =  
 b+ \sum_{i=1}^{m} \frac{ \lambda^i}{i!} \, \triangle^{i}_{k} (v)\, a^{ik-1} \, .
 \end{equation}
Set
$$
w_i:= ey^i (x+\lambda (\tau_0 + \cdots +\tau_{m-1}) \, y^{mk-1})^i \, .
$$
Note that $w_m$ is equal to the LHS of  \eqref{psieq1} multiplied by $ey^m=a$ from the left. 
Then using relation $xy-yx= \tau$, we obtain
\begin{eqnarray*}
\label{expwm}
w_i &= &ey^{i-1} (yx+\lambda (\tau_0 + \cdots +\tau_{m-1}) \, y^{mk}) 
(x+\lambda \tau_0 + \cdots +\tau_{m-1}) \, y^{mk-1})^{i-1} \\[0.15cm]
&=& ey^{i-1} (xy - \tau +\lambda (\tau_0 + \cdots +\tau_{m-1}) \, y^{nm}) 
(x+\lambda (\tau_0 + \cdots +\tau_{m-1}) \, y^{mk-1})^{i-1} \\[0.15cm]
&=& ey^{i-1} (x+\lambda (\tau_0 + \cdots +\tau_{m-1}) \, y^{mk-1})\,  y\, 
(x+\lambda (\tau_0 + \cdots +\tau_{m-1}) \, y^{mk-1})^{i-1} \\[0.15cm]
&&-   ey^{i-1} \tau (x+\lambda (\tau_0 + \cdots +\tau_{m-1}) \, y^{mk-1})^{i-1} \\[0.15cm]
&=&ey^{i-1} (x+\lambda (\tau_0 + \cdots +\tau_{m-1}) \, y^{mk-1})\, y 
\,(x+\lambda (\tau_0 + \cdots +\tau_{m-1}) \, y^{mk-1})^{i-1}\\[0.15cm]
&&- \tau_{m-i+1} ey^{i-1}  (x+\lambda (\tau_0 + \cdots +\tau_{m-1}) \, y^{mk-1})^{i-1}\\[0.15cm]
&=&ey^{i-1} (x+\lambda (\tau_0 + \cdots +\tau_{m-1}) \, y^{mk-1}) y
 (x+\lambda (\tau_0 + \cdots +\tau_{m-1}) \, y^{mk-1})^{i-1} \\[0.15cm]
&&-  \tau_{m-i+1} w_{i-1} .
\end{eqnarray*}
Repeating this process by moving further to the left $y$ term, we get
$$  w_i= w_{i-1} ( w_1- \tau_{m-i+1}- \cdots- \tau_{m-1}) \, .$$
 Hence 
$$
w_m =  w_1 (w_1 - \tau_{m-1})  (w_1 - \tau_{m-2}- \tau_{m-1})\cdots (w_1 -(\tau_{1} + \cdots + \tau_{m-1})  )  
$$
or  equivalently
$$
w_m =  (\tau_{0} + \cdots + \tau_{m-1})^m \, \prod_{i=0}^{m-1}  \Bigg(h+ \lambda a^k- \frac{\tau_1+\cdots+\tau_{i}}
{\tau_0 + \cdots +\tau_{m-1}}\Bigg).$$
Now if applying to $\Psi_{k, \lambda}^{-1}$, we have
$$
\Psi_{k, \lambda}^{-1} (w_m) =  (\tau_{0} + \cdots + \tau_{m-1})^m \,  \prod_{i=0}^{m-1}  \Bigg(h 
- \frac{\tau_1+\cdots+\tau_{i}}{\tau_0 + \cdots +\tau_{m-1}}\Bigg), $$
which is equal to $v(h-1) = a b$, and therefore
$$
w_m = \Psi_{k, \lambda} (a b) = a  \Big( b+ \sum_{i=1}^{n} \frac{ \lambda^i}{i!} \, \triangle^{i}_{m} (v)\, a^{im-1} \Big)\, .
$$
Since $O_\tau$ is a domain, dividing by $a$ the last expression, we get  \eqref{psieq1}.

(ii)
The proof of this part is similar to part (i).
\end{proof}

Now, let $G$ be the group generated by $\theta_{\lambda}, \psi_{m,\lambda}$ and $\phi_{m,\lambda} $.
Again $G$ can be viewed both as a subgroup of $\Autg (\SS_{\tau})$ and of $\Autg_{,\omega}(R )$.
Then:
\begin{corollary}
The homomorphisms $\rho$ and $\rho \circ \varphi$ map $G$ surjectively.
\end{corollary}

\begin{proposition}
\la{prop3}
Let $c_0,\cdots,c_{m-1} \in \c$. Then the assignment
$$
 x \, \mapsto \, x_1:=(c_0\, e_0\, + \, \cdots\, + \, c_{m-1} \ e_{m-1})\cdot x\, ,  \, 
 y \, \mapsto \, y_1:=(d_0\,e_0\,  + \cdots\, +\, d_{m-1}\ e_{m-1} ) 
 \cdot y\, , \, e_i \, \mapsto e_i 
 $$
defines an element  in $\Autg_{,\omega}(R )$ and $\Autg(\SS_\tau)$ if and only if $c_i\in \c^*$
and $d_i =1/c_{i+1}$ for $i\, (\mathtt{mod}\, m )$.
\end{proposition}

\begin{proof}
We will only prove that the above map defines an automorphism in $\Autg_{,\omega}(R )$, since for $ \SS_\tau$
the arguments 
are similar.  First, the above map is bijective if and only if $c_i \neq 0$, since the inverse is given by
$$
 x \, \mapsto \, x_1:=(d_{m-1} \, e_0\, + \, \cdots\, + \, d_{0} \ e_{m-1})\cdot x\, ,  \, 
 y \, \mapsto \, y_1:=(c_1\,e_0\,  + \cdots\, +\, c_{0}\ e_{m-1} ) 
 \cdot y\, , \, e_i \, \mapsto e_i \,.
 $$
We  need to show that $e_i\cdot x_1 = x_1 e_{i+1}$,    $e_i\cdot y_1 = y_1 e_{i-1}$ and $x_1\cdot y_1- y_1\cdot x_1=
x\cdot y- y\cdot x$. The first two relations are straightforward. The left hand side of the third relation is
\begin{eqnarray}
&&(c_0\, e_0\, + \, \cdots\, + \, c_{m-1} \ e_{m-1})  (d_0\,e_1\,  + \cdots\, + d_{m-2}\ e_{m-1}+  \, d_{m-1}\ e_0 )\cdot x \cdot y  \nonumber \\
& & -(d_0\,e_0\,  + \cdots\, +\, d_{m-1}\ e_{m-1} ) (c_0\, e_{m-1} \, + \, c_1\,e_0+\cdots\, + \, c_{m-1} \ e_{m-2}) \cdot y \cdot x   \nonumber \\
&=&  (c_0\,d_{m-1} \, e_0 + \, c_1 \, d_0\,e_1 + \cdots+\, c_{m-1}\,d_{m-2}\,e_{m-1} ) \cdot x \cdot y \,           \nonumber \\
&& - (c_1\,d_0 \, e_0 + \, c_2 \, d_1\,e_1 + \cdots+\, c_{m-1}\,d_{m-2}\,e_{m-2}\, +  \, c_0\, d_{m-1} \, e_{m-1}) \cdot y \cdot x           \nonumber\\
&=& xy-yx, \nonumber
\end{eqnarray}
which proves our claim.
\end{proof}

The group of automorphisms  defined in Proposition~\ref{prop3} is isomorphic to $(\c^{*})^m$.
Let $H$ be the subgroup of $(\c^{*})^m$ consisting of $m$-tuples $(c_0,c_1,\cdots,c_{m-1})$
such that $c_0c_1\cdots c_{m-1}=1$.

\begin{lemma}
\label{subH}
 $H$ is the kernel of $\rho$ and $\rho \circ \varphi$. In particular,
 $H$ is a normal subgroup of $\Autg_{,\omega}(R )$ and $\Autg(\SS_\tau)$.
\end{lemma}

\begin{proof}
Let $\eta \in \Autg(\SS_{\tau})$ and let $x_1:=\eta(x)$, $y_1:=\eta(y)$.
If $\eta \in \Ker(\rho)$ then
 $$ex_1^m\, =\, ex^m\, , \quad  ey_1^m \,= \,ey^m\, , \quad ey_1x_1\, = \, eyx \, .$$
 Let $ x_1= \sum^{m-1}_{i=0} \, e_i f_i(x,y)\,$. Then since  $e_ix_1=x_1e_{i+1}$, 
we have $ e_i f_i(x,y) e_{i+1} = e_i f_i(x,y)$
and hence $ ex_1^m \, =\, e\, f_0 \, f_2\, \cdots \, f_{m-1} = e x^m$.
The later implies   $e_i\, f_i =c_i e_i x$ for some $c_i \in \c^*$.
Thus $x_1=\sum_{i=1}^{n}\, c_i\, e_i x$ and $ex_1^m=c_0c_1\cdots c_{m-1} ex^m$.
Similarly, we can show that $y_1=\sum_{i=1}^{n}\, d_i\, e_i y$ and
$d_0\cdots d_{m-1}=1$. By Proposition~\eqref{prop3}, we can conclude that $\eta \in H$.
\end{proof}
Thus, combining Proposition~\ref{prop2}  and Lemma~\ref{subH} we have proved:

\begin{theorem} 
\label{semdir}
Let $G:=\Aut(O_{\tau}(\z_m))$. Then 
 $$\Autg(\SS_\tau(\z_m)) \, \cong \, H \rtimes G   \,  \cong \, \Autg_{,\omega}(R).$$
\end{theorem}

Let $G_1: = \, \langle \Theta_\lambda , \Psi_{k,\mu} \rangle $, 
$G_2: =  \, \langle \Theta_\lambda , \Phi_{k,\mu} \rangle$ and $G_3: =  \, \langle \Theta_\lambda \rangle$.
Then:
\begin{proposition}
\label{Th4pr}
 We have $G\,  \cong \, G_1\,  \ast_{G_3} \, G_2$.
\end{proposition}

\begin{proof}
By Theorem \ref{semdir}, the group $G$ can be embedded in $\Autg_{,\omega}(R)$ 
such that image of $G_1$ and $G_2$
are generated by $\langle \theta_\lambda , \psi_{k,\mu} \rangle_{k \geq 1}$ 
and $\langle \theta_\lambda , \phi_{k,\mu} \rangle_{k\geq 1}$ respectively.
Recall $\Autg_{,\omega}(R)$ can be presented as amalgamated product 
$$
A \,  \ast_{U} \, B,
$$
where $A$ is the group of symplectic affine transformations, $B$ is the group of 
triangular transformations and $U=A \cap B$. Any element in $g \in G$ can be written
$$
g= g_1 \, \cdots \, g_k,
$$
where $g_i$ is either in $G_1$ or $G_2$ but not in $G_3$. Using the
amalgamated product structure of $\Autg_{,\omega}(R)$ we can rewrite it as
$$
\theta_{\lambda} h_1 \, \cdots \, h_k,
$$
where $h_i$ is alternating in $\psi_{q(y)}$ and $\phi_{p(x)}$. This is a reduced word in $\Autg_{,\omega}(R)$ 
hence $g$ cannot be equal 1 unless $\lambda=1$
and $h_i=1$ for all $i$. 
\end{proof}

\section{Ideal classes of $O_\tau$ and quiver varieties}
\la{sec3}

In this section we recall some results proved by one of the authors \cite{E} and by \cite{BGK}, related to
a description of ideal classes of the algebra  $O_\tau$. More explicitly, we will show that there is a
$G$-equivariant bijection between $O_\tau$-ideals and the disjoint union of certain
quiver varieties $\mathfrak{M}_\tau$. %

\subsection{Nakajima's quiver varieties }\label{SNv}

%
Let $Q=(I,H)\,$ be a finite quiver (without loops) with the vertex set $I$ and the arrow set $H$, 
and let $\bar{Q}= (I, \bar H)$ be its \textit{ double} quiver, obtained 
by adjoining a reverse arrow $a^*$ for each arrow $a \in H$. Let $\mathbb{V}=(V_{i})_{i\in I}\,$ and
$\mathbb{W}=(W_{i})_{i\in I}\,$ be a pair of collections of vector
spaces. We consider the vector space of linear maps
\begin{equation}
\label{MVW}
\mathcal{M}(\mathbb{V}, \mathbb{W}) = \mathbf{E}(\mathbb{V})
\oplus \mathbf{L}(\mathbb{W}, \mathbb{V}) \oplus \mathbf{L}(\mathbb{V}, \mathbb{W}),
\end{equation}
where
$$ \mathbf{E}(\mathbb V):= \bigoplus_{a \in \bar{H}}\Hom(V_{t(a)}, V_{h(a)})
\,, \quad \mathbf{L}( \mathbb W, \mathbb V):=
\bigoplus_{i \in I} \Hom(W_{i}, V_{i}) \, .$$
Note that the space $\mathbf{E}(\mathbb V)$ can be identified with
$\Rep(\bar Q,\mathbf{v})$, the space of $\bar Q$-representations of
dimension vector $\,\mathbf{v} = (\dim\, V_{0}, \cdots, \dim\,
V_{n}) \in \mathbb {Z}^{I}\,$. There is a natural action of the
group $\, G(\mathbb V):= \prod \GL(V_{i})\,$ on $\mathcal{M}(\mathbb V, \mathbb W)$ given by
$$ (B_{a},v_{i},w_{i}) \, \mapsto \,
(g_{h(a)}\,B_{a}\,g^{-1}_{t(a)}, \, g_{i}\, v_{i}, w_{i}\,\,  g^{-1}_{i}) \, .$$
For each $\,\tau=(\tau_{i})_{i \in I} \in \mathbb C^{I}\,$, we
define a subvariety $\,\tilde {\mathfrak M}_{\tau}(\mathbb V, \mathbb W)\subseteq
\mathcal{M}(\mathbb V, \mathbb W)\,$ satisfying the following
\begin{itemize}
\item {\it Moment map equation:}

\begin{equation}\label{mme2}
  \mathop{\sum_{a\in Q}}_{h(a)=i} B_{a} B_{a^{\ast}} -
\mathop{\sum_{a\in Q}}_{t(a)=i} B_{a^{\ast}} B_{a} + v_{i} w_{i} = \tau_{i}
\,
\textrm{Id}_{V_{i}} \, , \quad i \in I
\end{equation}
\item{\it Stability condition :}
\begin{eqnarray}\label{stability}
  &&  \mbox{ If } \mathbb V^{\prime} \subset \mathbb V \, \mbox{ is a }
\bar{Q}-\mbox{submodule
    such that } B_{a}(V^{\prime}_{t(a)}) \subset V^{\prime}_{h(a)} \ ,
\nonumber \\*[1ex]
  && B_{a^{\ast}}(V^{\prime}_{h(a)}) \subset V^{\prime}_{t(a)} \,
   \mbox{ and } \, v_{i}(W_{i}) \subset V^{\prime}_{i} \,, \mbox{ then } \,
   \mathbb V^{\prime}= \mathbb V \, .
\end{eqnarray}
\end{itemize}
The action of $G(\mathbb V)$ on $\tilde {\mathfrak M}^{\tau}(\mathbb V, \mathbb W)\,$ is
free, due to the stability condition. The
\textit{Nakajima variety} associated to the triple
$\,(\mathbb V, \mathbb W,\tau)\,$ is defined as follows
\begin{equation}\label{nqv}
\mathfrak M^\tau_{Q}(\mathbb V, \mathbb W):=  \tilde
{\mathfrak M}^{\tau}(\mathbb V, \mathbb W) /\!\!/ G(\mathbb V) \, ,
\end{equation}
where $/\!\!/ $ is the GIT quotient.

\begin{remark}
The quiver variety $\mathfrak M^\tau_{Q}(\mathbb V, \mathbb W)$, associated to a quiver $Q$, 
can be identified with a representation variety of
 \textit{ the deformed preprojective algebra}  $\Pi^{\tau}(Q)$ (see \cite{CBH}).
\end{remark}

Let $\Gamma \subset \SL_{2}(\mathbb C)\,$ be a finite
subgroup. Then recall that one can associate to $\Gamma$ 
a quiver $Q_{\Gamma}=(I_{\Gamma}, H_{\Gamma})$, whose underlying graph is an extended 
Dynkin diagram (see \cite{McK}). Now let $\{\mathcal U_i\}_{i \in I_{\Gamma}}$
be a complete set of irreducible representations of $\Gamma$.
Then for a pair of $\Gamma$-modules $(V,W)$, we define collections of vector spaces $\mathbb V:=(V_i)$
and $\mathbb W:=(W_i)$, where $V \cong \oplus_{i} V_i  \otimes \,\mathcal U_i$  and 
$W \cong \oplus_{i} W_i \otimes \, \mathcal U_i$ are decompositions into irreducible modules.
Finally, let $\tau \in Z(\c\Gamma)$. Then $\tau=(\tau_i)_{i\in I_{\Gamma}}$ via the identification
$ Z(\c\Gamma) \cong \c^{I_{\Gamma}}$.
Thus, given such a triple $\, (V,W,\tau)$, we can associate to it
the quiver variety
$$ \mathfrak M^{\tau}_{\Gamma}( V, W):=\mathfrak M^{\tau}_{Q_\Gamma}(\mathbb V, \mathbb W)\, .$$

Since we will be concerned with the case when $\Gamma$ is cyclic and $W$ is one-dimensional
$\Gamma$-module, we can write quiver varieties for this case more explicitly.
Indeed, let $\Gamma \cong \z_m$ and  $W \cong \mathcal U_k $ for some $0 \le k \le m-1$. 
Then $Q_\Gamma$ has type $\tilde A_{m-1}$: it consists of $m$  vertices
$\{0, 1,\cdots, m-1\}$ and $m$ arrows $a_0, a_1,\cdots, a_{m-1}$, forming a cycle.
Denoting $B_{a_i}$ and $B_{a^*_i}$ by $X_i$ and $Y_i$ respectively, we get
\begin{eqnarray}
\mathfrak M^{\tau}_{\z_m}( V, \mathcal U_k)=&\Bigl\{&\Bigl(X_0, X_1, \cdots, X_{m-1}; \,Y_0,Y_1,\cdots, Y_{m-1};\, v_k, w_k\Bigr)\,
\Big |\,  X_i \in \Hom(V_{i+1}, V_i)\ ,\nonumber\\*[1ex] 
 && \hspace{-2mm}   \ Y_i \in \Hom(V_{i},V_{i+1}) \, ,  \, v_k \in \Hom(\mathbb C, V_k) \, ,\, w_k \in \Hom(V_k, \mathbb C) , \nonumber\\*[1ex] 
 &&  \hspace{-2mm} X_i Y_i - Y_{i-1} X_{i-1} + \tau_i \mathrm{Id}_{n_i} =0 \ , \, i\neq k , \nonumber\\*[1ex]  
 &&   \hspace{-2mm} X_k Y_k - Y_{k-1} X_{k-1} + \tau_{n_k} \mathrm{Id}_{n_k} =v_k w_k 
  \Bigr\} \Big/\!\!\!\Big/ \prod_{i} \GL(V_i) \, , \la{identiquiv}
\end{eqnarray}
where $n_i:=\dim_{\mathbb C}(V_i)$ and the action of $\prod_{i} \GL(V_i)$ is given as follows. 
For $g=(g_0,\cdots,g_{m-1}) \in \prod_{i} \GL(V_i)$,
\begin{eqnarray*}
&& g.(X_0, X_1, \cdots, X_{m-1}; \,Y_0,Y_1,\cdots, Y_{m-1};\, v_k, w_k) \\[0.15cm]
&=& (g_0 X_0 g_1^{-1}, \cdots, g_{m-1}X_{m-1}g_0^{-1}; g_1Y_0g_{0}^{-1},\cdots,
 g_{0} Y_{m-1}g_{m-1}^{-1};\, g_kv_k, w_kg_k^{-1}).
\end{eqnarray*}
 It is a well-known fact (see e.g. \cite{CB}) that

\begin{equation}
\la{dimnv1}
 \dim_{\mathbb C}\,  \mathfrak M^{\tau}_{\z_m}( V, \mathcal U_k)=  \begin{cases} 
      2(n_k -(n_0-n_1)^{2}),  &  \mbox{ for } m=2, \\[0.2cm]
     2\bigg(n_k -\bigg(\sum_{i=0}^{m-1} n^2_i -\sum_{ i < j} n_i n_j\bigg)\bigg), &  \mbox{ for } m>2 . \\
   \end{cases} 
   \end{equation}

\begin{notation} The following notation for quiver varieties will also be used later in the paper. If 
$V \cong \oplus_{i=0}^{m-1} V_i  \otimes \,\mathcal U_i$ and $n_i=\dim(V_i)$, then we set
$$ \mathfrak M^{\tau}_{\z_m}(n_0,\cdots,n_{m-1}; \,k):=\mathfrak M^{\tau}_{\z_m}( V, \mathcal U_k)\, .$$
\end{notation}

Now we state  a  simple but important observation for cyclic quiver varieties:

\begin{lemma}
\la{simpimplem}
Let $\tau = (\tau_0,\cdots,\tau_{m-1}) $ and  $\tau'=(\tau_1,\cdots,\tau_{m-1},\tau_0)\in \c\z_m \cong \c^m $. Then the map
$\mathfrak M_{\z_m}^{\tau}(n_0,\cdots,n_{m-1}; \,k) \to \mathfrak M_{\z_m}^{\tau}(n_1,\cdots, n_{m-1},n_0;\,k-1)$
defined by
$$  X_i \mapsto X_{i+1}, Y_i \mapsto Y_{i+1} \mbox{  for }   i+1 \, (\mathtt{mod} \, m) $$
 is a canonical  isomorphism of algebraic varieties. 
\end{lemma}

\begin{proof}
It follows immediately from \eqref{identiquiv}.
\end{proof}

We finish this subsection by considering a special case
of cyclic quiver varieties, $\Gamma=\{1\}$. In this case $Q_\Gamma$
is just a quiver with one vertex and one loop. We denote 
$\mathcal{C}_n:= \mathfrak M^{\tau=1}_{\Gamma=\{1\}}(\c^n,\c)$, the $n$-th 
\textit{Calogero-Moser} variety. More 
explicitly,
\begin{eqnarray}
\label{defCMv}
\mathcal{C}_n\, &=& \, \{ (X, Y, v, w) \, | \, X, Y \in \End(\c^n), \, v \in \Hom(\c^n, \c),\\[0.15cm]
&& \, w \in \Hom(\c, \c^n) , \, XY -YX + \mathrm{Id}_n \, =\,  vw \}\,  /\!\!\!\ / \GL_n(\c). \nonumber
\end{eqnarray}
The space $\mathcal{C}_n$ is an irreducible, smooth, affine algebraic variety of dimension $2n$.
It is known to be diffeomorphic to $\mathrm{Hilb}_n(\c^2)$, the Hilbert scheme of $n$ points on $\c^2$
(see \cite{N} and \cite{W}).
\subsection{$G_{\Gamma}$ action on quiver varieties}
First, let us recall the action of $G:=\Aut(A_1)$
on $\mathcal{C}_n, (n \ge 0)$ as has been defined in \cite{BW1,BW2}.
A classical result by Dixmier \cite{D} states that  $G$ is
generated by two families of $A_1$-automorphisms
\begin{equation}
\la{autA1act}
 \Phi_{a,n}(x,y):=(x+a\, y^n, y) \, \mbox{  and } \,  \Psi_{b,m}(x,y):=(x, y+b\,x^m) \, ,
 \end{equation}
where $a,b \in \c$ and $n,m \in \z_{\ge 0}$. Then for $\sigma \in G$
and $(X, Y, v, w) \in \mathcal{C}_n$, let
 \begin{equation}
\label{action1}
\sigma. (X, Y, v, w)\, := \, (\sigma^{-1}(X), \sigma^{-1}(Y), v, w)\, ,
\end{equation}
where $\sigma^{-1}$ is the inverse of $\sigma$ in $G$. For example, if
$\sigma= \Phi_{a,n}$ then
$$ \Phi_{a,n}. (X, Y, v, w)\, = \,(X-a\,Y^{n}, Y, v, w)\, .$$

To define the action of $G_\Gamma$ on $\mathfrak M^{\tau}_\Gamma$
for $\Gamma \cong \z_m$,
we give a presentation of  this quiver variety  similar to that
of $\mathcal{C}_n$ as in \eqref{defCMv}.
For a point 
$$(X_0, X_1, \cdots, X_{m-1}; \,Y_0,Y_1,\cdots, Y_{m-1}, v_k, w_k)
 \in \mathfrak M^{\tau}_{\z_m}( n_0,\cdots,n_{m-1};\, k) \, ,$$
set 
\begin{equation}
\la{XYbasis}
X:=
\begin{pmatrix} 
0 & X_0 & 0& \hdots  &0 \\
 0& 0 & X_1  & \hdots & 0\\
0& 0&0&\ddots&\vdots \\
\vdots& \vdots&\ddots&\ddots &X_{m-2}\\
X_{m-1} &0  & \hdots &0& 0 
\end{pmatrix} 
\, , \,
Y:=
\begin{pmatrix} 
0 & 0 &0& \hdots &  Y_{m-1} \\
Y_0 & 0 &0& \hdots &  0\\
0 & Y_1 &0& \ddots &\vdots \\
\vdots& \vdots&\ddots& \ddots &0\\
0 &0 & \hdots & Y_{m-2} & 0
\end{pmatrix} 
 \, .
\end{equation}
These are $m\times m$ block matrices of size $n\times n$, where $n:=n_0+\cdots+n_{m-1}$,
such that $X$ (resp. $Y$) is a block matrix whose only nonzero entries on the $1$st and $(-m+1)$-st
(resp. $(-1)$-st and $(m-1)$-st) diagonals.
Next, let 
\begin{equation}
\la{defvw}
 w:=(0,\cdots,0,w_k, 0,\cdots,0) \, , \quad  v:= (0, \cdots,0,v_k, 0,\cdots,0)^{\mathrm{t}}\, ,
\end{equation} 
where $\mathrm{t}$ is taking the transpose of the matrix.
 So $w$ is an $n$-dimensional row vector and $v$ is an $n$-dimensional column vector.
Finally, let 
\begin{equation}
\la{defT}
 \mathcal{T} :=\, \mathrm{Diag}\, [\, \tau_0 \mathrm{Id}_{n_0},\, \cdots\, ,
 \tau_{m-1} \mathrm{Id}_{n_{m-1}}\, ]  
 \end{equation}
be an $n\times n$  block diagonal matrix. Then it is easy to see that relations in \eqref{identiquiv} 
are equivalent to
\begin{equation}
\la{CMtype}
  X\, Y -Y\, X + \mathcal{T }\, = \,  v w\,.
  \end{equation}
Thus we can  refer to points of $\mathfrak{M}^{\tau}_{\z_m}(n_0,\cdots,n_{m-1};\, k)$ as quadruples $(X,Y,v,w)$
satisfying condition \eqref{CMtype}.
This identification allows us to define the action of $G_{\z_m} := \Aut(O_\tau(\z_m))$ on each  
$\mathfrak{M}^{\tau}_{\z_m}$ as follows. By Theorem~\ref{semdir}, the group $G_{\z_m}$ 
is generated by automorphisms defined
by \eqref{gens1}-\eqref{gens3}. So it suffices to define an action on these generators.
Let $\sigma \in \mathrm{Aut}_\Gamma(\SS_\tau)$ be as one of \eqref{gens1}-\eqref{gens3}.
 Then for a $(X,Y,v,w) \in \mathfrak{M}^{\tau}_{\z_m}(n_0,\cdots,n_{m-1}; \,k)$, define exactly the 
 same action as for $\mathcal{C}_n$:
 \begin{equation}
\label{action}
\sigma. (X, Y, v, w)\, := \, (\sigma^{-1}(X), \sigma^{-1}(Y), v, w).
\end{equation}
This is a well-defined action, since the RHS of  \eqref{action}
satisfies \eqref{CMtype} and $\sigma^{-1}(X)$ (resp. $ \sigma^{-1}(Y)$) has the same non-zero diagonals
as $X$ (resp. $Y$).
For instance, if $\sigma=\psi_{k,\lambda}$ then
$$ \psi_{k,\lambda}. (X, Y, v, w)\,= \, (X + k \lambda (\tau_0+\tau_1+\cdots+\tau_{m-1})
Y^{km-1}, Y, v, w)\, .$$
One can easily see that  the action $\psi_{k,\lambda}$ on $(X_0,\cdots,X_{m-1}; Y_0,\cdots,Y_{m-1}; v_k,w_k)$
corresponding via identification \eqref{XYbasis}  to $(X, Y, v, w)$ is given by
 \begin{eqnarray*}
&& X_0 \, \mapsto \, X_0 +  k \lambda (\tau_0+\cdots+\tau_{m-1})\,Y_{m-1}\,Y_{m-2}\cdots Y_1, \\[0.1cm]
&& X_i \, \mapsto \,  X_i +  k \lambda (\tau_0+\cdots+\tau_{m-1})\,
 Y_{i-1}\, Y_{i-2}\cdots Y_0\, Y_{m-1}\, Y_{m-2}\cdots Y_{i+1}\,, \, i=1,\cdots,m-1, \\[0.1cm]
 && Y_i\,  \mapsto \, Y_i \quad  i=0,1,\cdots,m-1  \,  , \, v_k \, \mapsto v_k\, , \,  w_k \mapsto w_k.
\end{eqnarray*}
The next statement follows directly from the definition of quiver varieties:

\begin{corollary} 
The map in Lemma~\ref{simpimplem} is $G_{\z_m}$-equivariant.
\end{corollary}

\subsection{$G_{\Gamma}$-equivariant bijective correspondence.}

Let $\mathcal{R}_\tau$ be the set of isomorphism classes $\SS_\tau$-submodules 
of $e\, \SS_\tau$ and let $\mathcal{R}'_\tau$ be the set of isomorphism classes
of $O_\tau$-ideals. Then the functor $F$ (see \eqref{equiv}) gives a natural bijection
between  $\mathcal{R}'_\tau$ and $\mathcal{R}_\tau$. 
Let $K_0(\Gamma)$, $K_0(\SS_\tau)$ and $K_0(O_\tau)$
be the Grothendieck groups of the algebras $\c\Gamma$, $\SS_\tau$ and $O_\tau$ 
respectively. Using Quillen's theorem and Morita equivalence between $\SS_\tau$
and $O_\tau$, we can identify all three groups. 
Next, we recall that $\hat \Gamma$ is a  set of irreducible $\Gamma$-modules.  
Then there is a map $\gamma :  \mathcal{R}_\tau \to K_0(\Gamma) \times \hat{\Gamma}$
which sends a submodule $M$ to a pair $(V,W)$ so that $V$ does not
contain $\c\Gamma$. One can show (see \cite[Theorem 7]{E}) that  
$\gamma(M_1)=\gamma(M_2)$ if and only if $[M_1]=[M_2]$ in $K_0(\SS_\tau)$.
Hence, the $K$-theoretic description  produces a decomposition 
of $\mathcal{R}_\tau$ (and of $\mathcal{R}_\tau'$):
\begin{equation}
\la{eq1}
\mathcal R_\tau = \bigsqcup_{V,W} \mathcal R_{\tau} (V,W) \,  , 
\end{equation}
where $V$ runs over all finite-dimensional $\Gamma$-modules and $W\cong \mathcal U_k$ for 
some $k=0,\cdots,m-1$. 

The action of $G_{\z_m}$ on $\mathcal R_\tau$ is pointwise, 
i.e., if $\sigma \in G_{\z_m}$ and $M \subset O_\tau$,
then $\sigma.M:=\{\sigma(m) \, | \, m\in M\}$.
 Moreover, this action respects the decomposition \eqref{eq1}.
The following result is proved in \cite{E}:
\begin{theorem}
\la{thgeqivbij}
Let $V$ be a  finite-dimensional $\Gamma$-module.
 Then for any $0 \le k \le m-1$,  there is a natural $G_{\z_m}$-equivariant bijection 
$\Omega :  \mathfrak{M}^{\tau}_{\z_m} (V, \mathcal U_k) \to \mathcal R_\tau(V, \mathcal U_k) $
sending a point  $(X,Y, v, w)$ in $\mathfrak{M}^{\tau}_{\z_m} (V, \mathcal U_k)$
 to the class of the fractional ideal of $\SS_\tau$:
\begin{equation}
\la{idealpre}
M =  e_k \, \det (Y- y\,\mathrm{Id})\, \SS_\tau + e_k\, \kappa \, \det (X-x\, \mathrm{Id}) \,\SS_\tau ,
\end{equation}
where $\mathrm{Id}$ is the idenity matrix on $V$ and $\kappa$ is the following element
\begin{equation}
\la{kappafor}
 \kappa \, = \, 1 -  w (Y- y\,\mathrm{Id})^{-1}(X-x\, \mathrm{Id})^{-1} v
 \end{equation}
in $Q(\SS_\tau)$,  the classical ring of quotients of $\SS_\tau$.
\end{theorem}

\begin{remark}
$\bf{1.}$ $\kappa$ is, indeed, an element of  $Q(\SS_\tau)$, since $ (Y- y\,\mathrm{Id})^{-1}(X-x\, \mathrm{Id})^{-1} $
is an $n\times n$ matrix with entries from $Q(\SS_\tau)$ and multiplying from the left by $w$
($1\times n$ matrix) and from the right by $v$($n\times 1$ matrix) produces an element in  $Q(\SS_\tau)$.

$\bf{2.}$\, The bijective correspondence part (without $G$-equivariance and a presentation
 for an ideal) was  proved earlier by Baranovsky, Ginzburg and Kuznetsov \cite{BGK}.
 
 $\bf{3.}$\, The case when $\Gamma=\{1\}$, that is, when $O_\tau \cong A_1$ and the corresponding
 quiver variety is $\mathcal{C}_n$ was proved by Berest and Wilson \cite{BW1,BW2}, and they also
 showed the transitivity of the $G$-action.
\end{remark}

Since our main results are concerning  quiver varieties for $\Gamma \cong \z_2$, 
we will restate the above theorem more explicitly in this case.
First, let us identify all triples $(n_0, n_1; \epsilon)$ for which the corresponding quiver variety
 $ \mathfrak M^{\tau}_{\z_2}( n_0,n_1;  \epsilon )$ is non-empty. If we assume that $n_0 \le n_1$
 and take $n_0=n-k$  and $n_1=n$,  then by \eqref{dimnv1}, we have
\begin{equation*}
\la{dimnv2}
 \dim_{\mathbb C}\,  \mathfrak M^{\tau}_{\z_2}( n-k,n;  \epsilon )=  \begin{cases} 
      2\,(n-k-k^2) , &  \mbox{ for } \epsilon= \mathrm{0}, \\[0.2cm]
     2\, (n-k^2), &  \mbox{ for }  \epsilon=\mathrm{1} . \\ 
   \end{cases} \,
   \end{equation*}
So we need to assume $n\ge k^2+k$ for $\epsilon=\mathrm{0} $ and $n \ge k^2$ for $\epsilon=\mathrm{1}$.
Now assuming $n_1\ge n_0$, we obtain similar constraints on $n$ and $k$. Thus
$ \mathfrak M^{\tau}_{\z_2}(n_0,n_1;  \epsilon) \neq \emptyset$
 if and only if $(n_0,n_1) \in L_{\epsilon} \subset (\z_{\ge 0})^2$, where
 \begin{eqnarray}
 \la{Lep}
&& L_{0}:=\{ (n-k, n)\, | \, k \ge 0, n \ge k^2+k \}  \cup 
 \{ (n, n-k) \, | \,k \ge 0,  n \ge k^2\}, \\[0.1cm]
 && L_1:=\{ (n-k, n)\, | \, k \ge 0, n \ge k^2 \}  \cup \{ (n, n-k) \, | \,k \ge 0,  n \ge k^2+k\}  .\nonumber
\end{eqnarray}
Thus, by Theorem~\ref{thgeqivbij},  for $\Gamma\cong \z_2$ there is a decomposition
 \begin{equation}
 \la{iddecz2}
\mathcal{R}_\tau =\bigsqcup_{\epsilon=0,1}\,  \bigsqcup_{(n_0,n_1)\in L_{\epsilon}} \,
\mathcal{R}_\tau(n_0,n_1; \epsilon) ,
\end{equation}
and a $G_{\z_m}$-equivariant bijection defined for any $\epsilon=0,1$ and $(n_0,n_1)\in L_{\epsilon}$:
\begin{equation}
\la{geqbijz2}
\Omega\,: \, \mathfrak M^{\tau}_{\z_2}(n_0,n_1;  \epsilon) \simeq \mathcal{R}_\tau(n_0,n_1; \epsilon)\, , \quad (X,Y,v,w) \mapsto M,
\end{equation}
where $M$ is defined in \eqref{idealpre}-\eqref{kappafor}.
 
 \section{Generating set for $\O (\mathfrak{M}^{\tau}_{\z_2})$}
\la{sec4}

Our goal for this section is to produce a set of generators for 
$\mathcal{O}(\mathfrak{M}^\tau_{\z_2}
(V, \mathcal U_k))$, which will be used
to show the transitivity of the $G$-action.

Let $(X,Y,v,w)$ be a point in $\mathfrak{M}^\tau_{\z_m}(V, \mathcal U_k)$. Then
by Theorem~\ref{thgeqivbij}, the ideal corresponding to this point is uniquely 
determined by $\kappa$. Expand $\kappa$ into the formal power series:
\begin{equation}
\label{kappaexp}
 \, \kappa \ = \ 1+ \sum_{l,q \ge 0} (w\, Y^l X^q\, v) \, y^{-l-1} x^{-q-1} \, ,
\end{equation}
where $w\, Y^l X^q\, v \in  \c$.
This defines an embedding
$$  \mathfrak{M}^\tau_{\z_m}(V, \mathcal U_k) \, \to \, \c^{\infty}\, , \, (X, Y, v, w) \, \mapsto \, ( w Y^l X^q v)_{l,q\ge 0} \, .$$
Let $p:=\dim(V)$. Then by the Cayley-Hamilton identity any power of $X$ or $Y$ can
be expressed in terms of powers of $X$ or $Y$ less than or equal to $p$. So we have an
embedding 
$$  \mathfrak{M}^\tau_{\z_m}(V, \mathcal U_k) \, \to \, \c^{N} \, , $$
 where $N=p^2$.
Dually, we have $\c[x_1,\cdots,x_N] \onto \mathcal{O}( \mathfrak{M}_\tau(V, \mathcal U_k))$.
\begin{lemma}
\label{weakgenfun}
\begin{enumerate}
\item[(a)]
 $\mathcal{O}(\mathfrak{M}^\tau_{\z_m}(V, \mathcal U_k))$ is generated by $( w Y^q X^r v )_{q,r \le p}$.
\item[(b)] $w Y^qX^r v =0 $ unless $q = r ( \mathtt{mod}\, m)$.
\end{enumerate}
\end{lemma}

\begin{proof}
(a) It follows from the above arguments. 

(b) This is a special case of the next Proposition.
\end{proof}

\begin{proposition}
\label{wYTXv}
If $k_1,\cdots,k_t, l_1,\cdots,l_t$ are non-negative integers and $\epsilon=0,1$, then
\begin{equation}
\label{identityT}
 w \, Y^{k_1} X^{l_1} \, \cdots \, Y^{k_j} \mathcal{T}^{\epsilon} X^{l_j}\, \cdots
 \, Y^{k_t} X^{l_t} \, v  \, = \, c  w \, Y^{k_1} X^{l_1}\, \cdots\, Y^{k_t} X^{l_t} \, v \, , 
 \end{equation}
for some constant $c$. Moreover, both sides of  \eqref{identityT} are 
zero unless $K = L ( \mathtt{mod}\, m)$, where 
$K=k_1+\cdots+k_t$ and $L=l_1+\cdots+l_t$.
\end{proposition}

\begin{proof}
The first part of the statement for $\epsilon=0$ is trivial. Proof of the second part
for $\epsilon=0$ and $\epsilon=1$ are the same, so we may assume $\epsilon=1$. 
Recall (see \eqref{XYbasis}) that $X$ (resp. $Y$) is an $m\times m$ 
block matrix with only non-zero entries on the $(m-1)$-st and on the $(-1)$-st
(resp. $1$-st and $-(m-1)$-st) diagonals. Those non-zero entries
are the matrices $X_0,\cdots,X_{m-1}$ and $Y_0,\cdots,Y_{m-1}$ respectively. Set 
for $0\le i, j  \le m-1$,
$$D_{[i,j]}:=  \begin{cases} 
      Y_i\, Y_{i-1}\, \cdots\,Y_j, &  \mbox{ if } i\ge j ,\\[0.2cm]
    Y_i\, Y_{i-1}\cdots Y_0\,Y_{m-1}\cdots Y_j,
    &  \mbox{ if } i<j,  
   \end{cases} \,  \quad\,
   C_{[i,j]}:=  \begin{cases} 
      X_i\, X_{i+1} \cdots X_m\,X_0\cdots X_j,&  \mbox{ if } i> j, \\[0.2cm]
    X_i\, X_{i+1}\,\cdots\,X_{j},
    &  \mbox{ if } i\le j . 
   \end{cases} $$
Let $k_i=k'_im+q_i$ and $l_i=l_i'm+r_i$ for $i=1,\cdots,t$, where $k_i',l_i' \ge 0$
and $0 \le r_,q_i \le m-1$.
 Then  $Y^{k_i}$ is an $m \times m$ block matrix whose only nonzero entries are 
 on the $(m-q_i)$-th  and the $(-q_i)$-th diagonals. 
The entries on the $(m-q_i)$-th diagonal are
$$(D_{[j-1,j]})^{k_i'}\,D_{[j-1,m-q_i+j]} \,,  \mbox{ for } \,  0\le j \le q_i-1\, ,$$
where $D_{[-1,i]}:=D_{[m-1,i]}$, while the entries on the $(-q_i)$-th diagonal are
$$(D_{[j-1,j]})^{k_i'}\,D_{[j-1, j-q_i]} \, , \mbox{ for } \, q_i \le j \le m-1\, .$$
Similarly, $X^{l_i}$ is an $m \times m$ block matrix whose two nonzero diagonals are 
 the $(r_i)$-th and $(r_i-m)$-th ones.  The entries on the $(r_i)$-th one are
$$ C_{[j,j+r_i-1]}\,(C_{[j+r_i, j+r_i-1]})^{l_i'} \,  \mbox{ for } \,  0\le j \le m-r_{i}-1\, , $$
where $C_{[j,-1]}:=C_{[j,m-1]}$ and the entries on the $(r_i-m)$-th diagonal are
$$ C_{[j,j-m+r_i-1]}(C_{[j-m+r_i,j-m+r_i-1]})^{l_i'} \,  \mbox{ for } \, m-r_i \le j \le m-1\,.$$
Without loss of generality, we may assume $r_i \le q_i$. 
Then non-zero diagonals of $Y^{k_i}X^{l_i}$ are the $(m-q_i+r_i)$-th and the  
$(r_i-q_i)$-th ones and their entries are
$$ (D_{[j-1,j]})^{k_i'}\,D_{[j-1, m-q_i+j]} \, C_{[m-q_i+j, m-q_i+j+r_i-1]} \,
(C_{[m-q_i+j, m-q_i+j-1]})^{l_i'},  $$
for $  0\le j \le  q_i-r_i-1$ and 
$$(D_{[j-1,j]})^{k_i'}\,D_{[j-1, j-q_i]}\, C_{[j-q_i, j-q_i-r_i-1]} \,(C_{[j-q_i-r_i, j-q_i-r_i-1]})^{l_i'}, $$
for $q_i-r_i\le j \le m-1$. Thus both matrices
\begin{equation}
\label{twomatr}
Y^{k_1} X^{l_1}\,  \cdots  \, Y^{k_j} \mathcal{T}^{\epsilon} X^{l_j} \, \cdots\, 
 Y^{k_t} X^{l_t}\, \mbox{  and  } \,  Y^{k_1} X^{l_1} \cdots Y^{k_t} X^{l_t} 
 \end{equation} 
 are $m\times m$ block 
 matrices whose only non-zero entries are on the $m-\sum_i(q_i-r_i)$-th 
 and on the $-\sum_i(q_i-r_i)$-th diagonals. Hence these matrices are block diagonal matrices
 if and only if  $\sum_i q_i= \sum_i r_i$. The later is equivalent to $K = L (  \mathtt{mod}\, m)$.
 Now if $A=(A_{ij})$ is an $m\times m$ block matrix of size $n\times n$ then
 $ wAv=w_k(A_{kk})v_k$ (see \eqref{defvw}). This proves
 the second part  of the statement, i.e., both sides of \eqref{identityT} are zero unless $K = L ( \mathtt{mod}\, m)$.
 
 We now can assume the later condition and hence matrices in  \eqref{twomatr} are block diagonal.
Moreover, the  diagonal entries of the first matrix are multiples of the diagonal entries of the
second one, where the multiples $c_1,\cdots,c_{m}$ are certain permutations
of $\tau_0,\tau_1,\cdots,\tau_{m-1}$. Once again using $ wAv=w_k(A_{kk})v_k$, we get
that the LHS of  \eqref{identityT} is a multiple of the RHS.
\end{proof}
\begin{lemma}
\la{form1} 
For $(X,Y,v,w) \in \mathfrak{M}^\tau_{\z_m}(V, \mathcal U_k)$ we have
$$(YX) Y^l =  Y^l (YX) - \sum^{l-1}_{i=1} Y^i \mathcal{T} Y^{l-i-1} +   \sum^{l-1}_{i=1} Y^i vw Y^{l-i-1}\, .$$
\end{lemma}
\begin{proof}
Since $XY=YX-\mathcal{T}+vw$, we have
$$(YX) Y\, =\,  Y(YX-\mathcal{T}+vw)\, =\, Y(YX) - Y\mathcal{T} + Yvw \, .$$
Now the statement can be easily proved by induction.
\end{proof}
\begin{proposition}
\la{form2}
Let $(X,Y,v,w) \in \mathfrak{M}^\tau_{\z_m}(V, \mathcal U_k)$. Then
$$w \, Y^{k_1} X^{l_1} \, \cdots \, Y^{k_t} X^{l_t} \, v  \, =  \, w \, (Y^{K} X^{L} )\, 
v + f (w Y^i X^j v)_{i< K\, , \, j<L}\,,$$
where $K=k_1+ \cdots +k_t$, $L=l_1+ \cdots + l_t$
 and  $f$ is some polynomial function.
\end{proposition}

\begin{proof}
We prove this by induction. Suppose the statement holds for $K'< K$ and $L'< L$. Then 
\begin{eqnarray*}
&& w \, Y^{k_1} X^{l_1 -1} \, (XY) \, Y^{k_2-1} X^{l_2} \, \cdots \, Y^{k_t} X^{l_t} \, v  \\[0.1cm]
&=& w \, Y^{k_1} X^{l_1 -1} \, (YX-T+ vw) \, Y^{k_2-1} X^{l_2} \, \cdots \, Y^{k_t} X^{l_t} \, v\\[0.1cm]
&=& w \, Y^{k_1} X^{l_1 -1} \, (YX) \, Y^{k_2-1} X^{l_2} \, \cdots \, Y^{k_t} X^{l_t} \, v
- w \, Y^{k_1} X^{l_1 -1} \,\mathcal{T} \, Y^{k_2-1} X^{l_2} \, \cdots \, Y^{k_t} X^{l_t} \, v\\[0.1cm]
&&
 + \,w \, Y^{k_1} X^{l_1 -1} \, v w \, Y^{k_2-1} X^{l_2} \, \cdots \, Y^{k_t} X^{l_t} \, v .
\end{eqnarray*}
If we use Lemma~\ref{form1} to the first term,  Proposition~\ref{wYTXv} to the
second one and the induction assumption to the third term in the last expression, then
\begin{eqnarray*}
&&  w \, Y^{k_1} X^{l_1 -1} \, \bigg[\,Y^{k_2-1} (YX) - \sum^{k_2-1}_{i=1} Y^i T Y^{k_2-i-1} 
+ \sum^{k_2-1}_{i=1} Y^i vw Y^{k_2-i-1}\, \bigg]\, X^{l_2} \, \cdots \, Y^{k_t} X^{l_t} \, v\\[0.15cm]
&& - c w \, Y^{k_1} X^{l_1 -1} Y^{k_2-1} X^{l_2} \, \cdots \, Y^{k_t} X^{l_t} \, v   
+  g (w Y^i X^j v)_{i<K , \, j< L}\,  \\[0.15cm]
&=&  w \, Y^{k_1} X^{l_1-1} Y^{k_2} X^{l_2+1} \, \cdots \, Y^{k_t} X^{l_t} \, v 
- \sum^{k_2-1}_{i=1}  w \, Y^{k_1} X^{l_1-1} Y^{i} T Y^{k_2-i-1} X^{l_2} \, \cdots \, Y^{k_t} X^{l_t} \, v\\[0.15cm]
&& + \sum^{k_2-1}_{i=1}  w \, Y^{k_1} X^{l_1-1} Y^{i}\, v\times
w Y^{k_2-i-1} X^{l_2} \, \cdots Y^{k_t} X^{l_t} v - c w \, Y^{k_1} X^{l_1 -1} Y^{k_2-1} X^{l_2} \, \cdots \, Y^{k_t} X^{l_t} \, v
\\[0.15cm] 
&  & +g (w Y^i X^j v) \\[0.15cm]
&=& w \, Y^{k_1} X^{l_1-1} Y^{k_2} X^{l_2+1} \, \cdots \, Y^{k_t} X^{l_t} \, v
  + \tilde{g}(w Y^i X^j v)_{i<K,\, j<L}\, ,
\end{eqnarray*}
where in the last equality we again used Proposition~\ref{wYTXv} and the induction assumption.
Now repeating  this procedure we can move $Y^{k_2}$ further to the left and obtain
$$
w \, Y^{k_1+k_2} X^{l_1+l_2} Y^{k_3} \, \cdots \, Y^{k_t} X^{l_t} \, v + h (w Y^i X^j v)
$$
for some polynomial $h$. Similarly moving other powers of $Y$ to the left we get
$$
\, w \, (Y^{k_1+ \cdots + k_t} X^{l_1+ \cdots + l_t} )\, v + f (w Y^i X^j v)_{i< K, \, j<L},
$$
for some polynomial $f$.
\end{proof}

Recall that
 for $m=1$ we have $\mathfrak{M}^\tau_{\z_m=\{1\}}(\c^n,\c) \cong \mathcal{C}_n$ (see \eqref{defCMv}), 
 the $n$-th Calogero-Moser space.  Then one has:
\begin{theorem}
\la{w1}
The algebra  $\O (\mathcal{C}_n)$ is generated by the set
$W_1:=\{w (Y+cX)^n v\}_{c \in \c, n \in \N} $. 
\end{theorem}
\begin{proof}
Similar to the case $m=2$ below.
\end{proof}

Next we prove our main result in this section.

\begin{theorem}
\la{w2}
The set $W_2:=\{w (Y+cX)^{2n}v\}_{c \in \c, n \in \N} $
 generates $\O (\mathfrak{M}^\tau_{\z_2}(V, \mathcal U_k))$.
\end{theorem}

\begin{proof}
We prove this statement by expressing generators $w Y^i X^j v$ of $\O (\mathfrak{M}^{\tau})$
(see Lemma~\ref{weakgenfun}($a$))
 as polynomials of elements of $W_2$.
The proof is by induction on the total degree $i+j$ of $w Y^i X^j v$. 
First recall that, by Lemma~\ref{weakgenfun}($b$),
 $w Y^i X^j v$  are zero for odd values of 
$i+j$, so we will only consider $i+j=2n$.  For $n=1$, expanding $w (Y+ c X)^{2n} v$, one gets 
$$ w Y^2 v + c w (YX+ XY) v + c^2 w X^2 v \,, $$
and since 
$$
w (YX + XY) v = 2 w YX v - w T v + (wv)^2 = 2 w YX v + d,
$$
for some $d \in \c$, we have
\begin{equation}
\la{mon1}
w (Y+ c X)^{2}\, =\, w Y^2 v + 2c\,w YX v+c^2\, w X^2 v+c\,d.
\end{equation}
By choosing distinct constants $c_1, c_2$ and $c_3$ for $c$ in \eqref{mon1}, we can show that 
$w Y^2 v, w YX v, w X^2 v$ can be expressed by  elements of $W_2$. Now assume that all elements $w Y^i X^j v$ for
$i+j < 2n$ can be generated by elements of $W_2$. Then the expansion of  $w (Y+ c X)^{2n} v$ by powers of $c$ gives
$$
w Y^{2n} v + c w (Y^{2n-1} X + \cdots ) v + c^2  w(Y^{2n-2} X^2+ \cdots ) v + \cdots + c^{2n} w X^{2n} v.
$$
By Lemma \ref{form2}, the coefficient of $c^j$ can be written as
$$
 w(Y^{i} X^j+ \cdots ) v = s w Y^i X^j v +  f (w Y^p X^q v)_{p+q<2n}.
  $$
Then, by the induction assumption, $f (w Y^p X^q v)$ can be generated by $W_2$ and hence each $w Y^i X^j v$ 
for $i+j=2n$ can be expressed in terms of $W_2$.
\end{proof}

\section{Transitivity of the $G_\Gamma$-action}
\la{sec5}

%
Let $X$ be a smooth algebraic variety. A subgroup $G$ of $\Aut(X)$ is called {\it algebraically generated}
 if it is generated as an abstract group by a family $\mathfrak{G}$ of connected algebraic subgroups of $\Aut(X)$. 
 Suppose $\mathfrak{G}$ is closed under conjugation by elements of $G$.

\begin{definition}
We say that a point $x \in X$ is $G$-{\it flexible} if the tangent space $T_x X$ is spanned by the tangent 
vectors to the orbits $H \cdot x$ of subgroups 
$H  \in \mathfrak{G}$. The variety $X$ is called $G$-{\it flexible} if every point $x \in X$ is $G$-flexible. 
\end{definition}
Let us give some comments on this definition.
First, since $X$ is smooth, $G$-flexibility of $X$  can be defined in terms of cotangent spaces instead of tangent spaces.
Second, one can easily show that $X$ is $G$-flexible if one point of $X$ is $G$-flexible
and  $G$ acts transitively on $X$.

We have the following characterization of flexible points
(see \cite[Corollary 1.11]{AFKKZ}):

\begin{proposition}
\la{flex}
A point $x \in X$ is $G$-flexible if and only if the orbit $G \cdot x$ is open in $X$. 
An open $G$-orbit (if it exists) is unique and consists of all $G$-flexible points 
in $X$.
\end{proposition}

We need the following result proved in \cite[Theorem 1.3]{CB}:

\begin{theorem}
$\mathfrak{M}^{\tau}_{\Gamma}(V,U)$ is a reduced and irreducible scheme. In particular, it is
a smooth, connected affine algebraic variety.
\end{theorem}

Recall that $\mathfrak{M}^{\tau}_{\Gamma}(V,U)$   has a natural symplectic structure with the symplectic form
 $\omega = \Tr(dX \wedge dY)$ (see e.g. \cite{CB}). This form gives an isomorphism between 1-forms and vector fields:
$$
\Tr(f \,dX+ g \, dY) \, \mapsto g \frac{\partial }{\partial X} - f  \frac{\partial }{\partial Y} =(g, -f).
$$
Since algebraic vector fields are in one-to-one correspondence
with derivations,  
for any  Hamiltonian $H \in \O (\mathfrak{M}_{\tau})$, the one-form $dH$ defines
a derivation of $\O (\mathfrak{M}_{\tau})$, given in terms of 
the Poisson bracket $\{ H , \, - \, \}$. One can easily verify:

\begin{lemma}
The action of $G_{\Gamma}$ on $\mathfrak{M}_{\tau}$ is symplectic.
\end{lemma}
Therefore any one parameter subgroup of $G_{\Gamma}$ defines Hamiltonian flow on $\mathfrak{M}_{\tau}$
and  hence corresponds to some Hamiltonian in $\O (\mathfrak{M}_{\tau})$. 

\begin{example} We show that 
$$
\phi_t \, = \, (X+aY^2 , Y) \, (X, Y+t X^2 ) \, (X- aY^2, Y)
$$
is a Hamiltonian flow with Hamiltonian $ \Tr((X+aY^2)^3)$. In fact, we have
\begin{eqnarray*}
\phi_t &=& (X+aY^2 - a(Y+t(X+aY^2)^2)^2, \, Y+ t(X+aY^2)^2 ),\\
\frac{d\phi_t }{dt} |_{t=0} &=& ( -a (Y(X+aY^2)^2+(X+aY^2)^2 Y), \,(X+aY^2)^2).
\end{eqnarray*}
On the other hand,
\begin{eqnarray*}
d \, \Tr((X+aY^2)^3) & =& \Tr((X+aY^2)^2 \, (dX + a Y \,dY + a \, dY \,Y))\\
&=&
\Tr((X+aY^2)^2 \, dX) + \Tr (a [(X+aY^2)^2 Y+ Y(X+aY^2)^2] \, dY).
\end{eqnarray*}
Using the symplectic form $\Tr(dX \wedge dY)$ which gives the isomorphism
between 1-forms and vector fields:
$$
\Tr(f \,dX+ g \, dY) \, \mapsto g \frac{\partial }{\partial X} - f  \frac{\partial }{\partial Y} =(g, -f),
$$
 one obtains the corresponding vector field
$$
( -a (Y(X+aY^2)^2+(X+aY^2)^2 Y), \,(X+aY^2)^2).
$$
\end{example}

In fact there is a well-known and more general result.
\begin{lemma}
Let $\phi_t$ be a Hamiltonian flow with Hamiltonian function $H$ and $\psi$ is any 
symplectic automorphism, then $\psi \, \phi_t \, \psi^{-1}$
is a Hamiltonian flow with Hamiltonian $\psi^* H = H \circ \psi $.
\end{lemma}
\begin{proof}
Let $X_H$ be a Hamiltonian vector field for $H$, namely 
$$
X_H \, \rfloor \, \omega = dH.
$$
A flow $\phi_t$ being the Hamiltonian flow of $X_H$ implies
$$
\frac{d \phi_t}{ dt}\Bigr |_{t=0} (x)= X_H (x).
$$
Equivalently for any function $f$ we have
$$
\lim_{t \rightarrow 0} \frac{\phi^*_t (f) -f }{t} = X_H (f).
$$
For the flow $\psi \, \phi_t \, \psi^{-1}$ we have
$$
\lim_{t \rightarrow 0} \frac{(\psi \, \phi_t \, \psi^{-1})^* (f) -f }{t}.
$$
Slightly abusing notation we put $\psi^*f$ for $f$ then we have
$$
\lim_{t \rightarrow 0} \frac{(\psi \, \phi_t )^* (f) -\psi^*(f) }{t} = X_H (\psi^*(f)) = \psi_* (X_H) (f).
$$
Thus $\psi_* (X_H)$ is the vector field for the flow $\psi \, \phi_t \, \psi^{-1}$. Therefore
\begin{equation*}
\psi_* (X_H) \rfloor \, \omega = \psi^*(dH) = d ( \psi^* H ).\qedhere
\end{equation*}
\end{proof}

From the above lemma, it follows immediately that:

\begin{corollary}
\la{1-flow}
For $m \in \mathbb{N}$  and $n_1, n_2 \in \z_{\ge 0}$, we have
\begin{enumerate}
\item[($a$)]
  $w \Big( X+a \, Y^{mn_1-1}\Big)^{mn_2} v$ is Hamiltonian with the  flow $$\phi_t = (x+a y^{mn_1-1}, y) \,
 (x,y+tx^{mn_2-1}) \, (x- a y^{mn_1-1}, y).$$
\item[($b$)]
 $w \Big( Y+ b \, X^{mn_1-1}\Big)^{mn_2} v$ is Hamiltonian with the flow
 $$\phi_t = (x, y+b x^{mn_1-1}) \, 
(x+t y^{mn_2-1}, y) \, (x, y- b x^{mn_1-1}).$$
\end{enumerate}
\end{corollary}

\begin{lemma}
\la{span}
Let $X$ be a smooth affine algebraic variety. If $\{f_i\}^n_{i=1}$ generates $\O(X)$ as an algebra,
then $\{d f_i\}^n_{i=1}$ generates the cotangent space at any point of $X$.
\end{lemma}

Let $m=1$. Then  $\mathfrak{M}^{\tau}_{\Gamma=\{1\}}(\c^n,\c)= \CC_n$ is the $n$-th Calogero-Moser space
 and $G_{\Gamma=\{1\}} $ is isomorphic to the subgroup of $\Aut(\c[x,y])$ preserving
the canonical symplectic form $dx \wedge dy$. The later group can be identified with $\Aut(A_1)$,
the automorphism group of the first Weyl algebra \cite{ML1}.
We recall that $G_{\Gamma=\{1\}}$ is generated by two families of automorphisms
$$ (x+a\, y^n, y) \, \mbox{  and } \,  (x, y+b\,x^m) \, ,$$
where $a,b \in \c$ and $n,m \in \z_{\ge 0}$.
Then we recover the following result of Berest and Wilson \cite[Theorem 1.3]{BW1}:

\begin{theorem}
\la{trans-CM}
 $G_{\Gamma=\{1\}}$ acts on $\mathcal{C}_n$ transitively  for each $n \in \z_{\ge 0}$.
\end{theorem}

\begin{proof}
By Theorem \ref{w1}, $W_1$ generates $\O (\CC_n)$ as an algebra. 
On the other hand,  by Corollary \ref{1-flow}, the family $W_1$ is Hamiltonian for 
the family of one-parameter subgroups $\phi_t = (x+a y, y) \, (x,y+tx^{k-1}) \, (x- a y , y)$. 
Therefore, by Lemma \ref{span}, differentials of elements in $W_1$ span
the cotangent space at any point. Hence $\CC_n$ is a $G$-flexible. Since $\CC_n$ 
is connected, by Proposition \ref{flex}, $\Aut (\c^2)$ acts transitively on $\CC_n$.
\end{proof}

Similarly for $\Gamma \cong\z_2$ we have:
\begin{theorem}
$G_{\Gamma}$ acts transitively on $\mathfrak{M}^{\tau}_{\z_2}(n_0,n_1; \epsilon)$ for 
each $(n_0,n_1) \in L_{\epsilon}$.
\end{theorem}

\begin{proof}
Recall that $G_{\z_2} = \langle \Theta_\lambda,  (x+\lambda y^{2k-1},y) , \, (x, y+\mu x^{2k-1}) \rangle_{k \geq 1}$. 
By Theorem \ref{w2}, elements of $W_2$ generate $\O (\mathfrak{M}^{\tau}_{\z_2})$.
 On the other hand, by Corollary \ref{1-flow},
elements of $W_2$ are Hamiltonian with Hamiltonian flows 
 $\phi_t = (x+a y^{2k-1}, y) \, (x,y+tx) \, (x- a y^{2k-1}, y)$. 
 Hence, by Lemma \ref{span}, differentials of the functions from $W_2$ 
 span the cotangent space at any point. Hence $\mathfrak{M}^{\tau}_{\z_2}$
 is a $G_{\z_2}$-flexible variety. Since $\mathfrak{M}^{\tau}_{\z_2}$ is connected,
by Proposition \ref{flex},  $G_{\z_2}$ acts transitively on $\mathfrak{M}^{\tau}_{\z_2}$.
\end{proof}

\section{$\c^*$-fixed points of $\mathfrak{M}^{\tau}_{\z_2}$}
\label{secfixpoints}

In this section we present for each quiver variety $ \mathfrak M^{\tau}_{\z_2}(n_0,n_1;  \epsilon)$
 a distinct point for which the computation of the corresponding $\kappa$ is relatively simple.
 More precisely, we will find a point of $ \mathfrak M^{\tau}_{\z_2}(n_0,n_1;  \epsilon)$
 represented by a quadruple $(X,Y, v, w)$, or equivalently
 by $(X_0, X_1,Y_0,Y_1,v_\epsilon,w_\epsilon)$,  such that it is fixed by the $\c^*$-action, where
 $\c^*$ is a subgroup of  $G_{\z_2}=\Aut(O_\tau)$. The later implies that the corresponding matrices 
 $X$ and $Y$ are nilpotent and therefore  the expansion \eqref{kappaexp} of $\kappa$ 
 is a non-commutative Laurent polynomial in $x^{-1}$ and $y^{-1}$. 
  
We now briefly outline how to construct these points. 
Recall that $\mathfrak M^{\tau}_{\z_2}(n_0,n_1;  \epsilon)$ is nonempty  
 if and only if $(n_0,n_1) \in L_{\epsilon} \subset (\z_{\ge 0})^2$, where
 $L_{\epsilon}$ is  defined in \eqref{Lep}.  
We may assume that $n_0 \le n_1$ since, by Lemma~\ref{simpimplem},
 $ \mathfrak{M}^{\tau}_{\z_2}(n_0, n_1; \mathrm{0}) \cong \mathfrak{M}^{\tau'}_{\z_2}(n_1, n_0; \mathrm{1})$.
 First we construct points for special values of $n_0$ and $n_1$, namely for 
 $ \mathfrak{M}^{\tau}_{\z_2}(k^2, k^2+k; \mathrm{0})$ and 
 $\mathfrak{M}^{\tau}_{\z_2}(k^2-k, k^2; \mathrm{1})$, where $k\ge 1$. 
 Note that, by \eqref{dimnv1}, the later varieties are zero-dimensional and 
 by connectedness each of them consists of a single point. Since $G_{\z_2}$ acts on each of these singleton  
 varieties, they are fixed by  $G_{\z_2}$ and, in particular, by its subgroup $\c^*$.
Let $(X, Y, v, w)$ be the point of $\mathfrak{M}^{\tau}_{\z_2}(k^2-k, k^2; \mathrm{1})$ 
 such that (see \eqref{XYbasis})
\begin{equation}
X=
\begin{pmatrix} 
0 & X_0  \\
X_1& 0 
\end{pmatrix} 
\, , \,
Y=
\begin{pmatrix} 
0 & Y_1\\
Y_0 & 0
\end{pmatrix} 
 \, . \nonumber
\end{equation}
Then we will describe a procedure how to add $n-k^2$  rows and $n-k^2$ columns to matrices 
$X_0,X_1,Y_0$ and $Y_1$  to get a nilpotent point of  $\mathfrak M^{\tau}_{\z_2}( n-k,n; \mathrm{1})$. 
Similarly, one can add $n-k-k^2$  rows and $n-k-k^2$ 
 columns to $X_0,X_1,Y_0$ and $Y_1$ to obtain a point of $\mathfrak M^{\tau}_{\z_2}( n-k,n; \mathrm{0})$.

\subsection{Prelimanaries}
\label{Prelim}
In this section we give decomposition of sets 
$\{1,2,\cdots,k^2\}$ and $\{1,2,\cdots,k^2-k\}$ into disjoint
union of subsets, which will be used as index sets  
for defining column vectors of matrices $X_i$ and $Y_i$ ($i=0,1$). Let
\begin{eqnarray}
S_1:&=& \bigg{\{}  \, i(i+1)-j+1 \quad \bigg{|}\quad  1 \le j \le \bigg[\frac{k}{2}\bigg] \, , \,  j \le i \le k-j\, \bigg{\}}, \nonumber\\[0.2cm]
S_2:&=&  \bigg{\{}  \,  (i+1)^2-j \quad\bigg{|} \quad  0 \le j \le \bigg[\frac{k-1}{2}\bigg] \, , \,  j \le i \le k-j-1\,\bigg{\}}, \nonumber\\[0.2cm]
S_3:&=& \bigg{\{}  \, i(i+1)-j+1 \quad \bigg{|} \quad  \bigg[\frac{k}{2}\bigg]+1 \le i \le k-1 \, , \,  k-i+1 \le j \le i \, \bigg{\}}, \nonumber\\[0.2cm]
S_4:&=&  \bigg{\{}  \, (i+1)^2-j \quad\bigg{|} \quad \bigg[\frac{k+1}{2}\bigg] \le i \le k-1 \, , \,  k-i \le j \le i \, \bigg{\}}.  \nonumber
\end{eqnarray}
We have:
\begin{lemma} 
\label{setS}
\begin{enumerate}
\item[(i)] For any two distinct pairs $(i,j)$
the defining relation in $S_l\,  (l=1,\cdots,4) $  yields distinct values. 
\item[(ii)] $S_i \cap S_j = \emptyset$ for $ i\neq j$. 
\item[(iii)] $ \sum_{l=1}^4 |S_l|=k^2 \,$, 
where $|S_l |$ is the cardinality of $S_l$. 
\end{enumerate}
In particular, 
$$\bigsqcup_{l=1}^{4} S_l = \{ 1,\cdots,k^2\} \, . $$
\end{lemma}

Now we introduce another sets of integers:
\begin{eqnarray}
T_1:&=& \bigg{\{}  \, i(i+1)-j+1 \quad \big{|}\quad  1 \le j \le \bigg[\frac{k-1}{2}\bigg] \, , \,  j \le i \le k-j-1 \bigg{\}}, \nonumber\\[0.2cm]
T_2:&=&  \bigg{\{}  \,  (i+1)^2-j \quad\bigg{|} \quad  0 \le j \le \bigg[\frac{k-1}{2}\bigg] \, , \,  j \le i \le k-j-2\bigg{\}}, \nonumber\\[0.2cm]
T_3:&=& \bigg{\{}  \, i(i+1)-j+1 \quad \bigg{|} \quad  \bigg[\frac{k+1}{2}\bigg] \le i \le k-1 \, , \,  k-i \le j \le i \, \bigg{\}} ,\nonumber\\[0.2cm]
T_4:&=&  \bigg{\{}  \, (i+1)^2-j \quad\bigg{|} \quad  \bigg[\frac{k}{2} \bigg] \le i \le k-2 \, , \,  k-i-1 \le j \le i \, \bigg{\}}, \nonumber
\end{eqnarray}
then one can prove a statement similar to that of Lemma~\ref{setS}:
\begin{lemma} 
\label{setT}
\begin{enumerate}
\item[(i)] For any two distinct pairs $(i,j)$
the defining relation in $T_l\,  (l=1,\cdots,4) $  yields distinct values. 
\item[(ii)] $T_i \cap T_j = \emptyset$ for $ i\neq j$. 
\item[(iii)] $ \sum_{l=1}^4 |T_l|=k(k-1) \,$. 
\end{enumerate}
In particular, 
$$\bigsqcup_{l=1}^{4} T_l = \{ 1,\cdots,k(k-1)\} \, . $$
\end{lemma}

Next we establish some important relations between sets $\{S_l\}$ and $\{T_l\}$.
Let  $S_4=S_4' \sqcup S_4''$, where 
$$ S_4' = \big{\{}  k(k-1) + i \, \big{|} \, 1 \le i \le k-1 \, \big{\}} $$
and $S_4''=S_4\setminus S_4'$.
Hence using Lemmas~\ref{setS} and \ref{setT}, one obtains
$$ S_1\, \sqcup \, S_2\setminus \{k^2\} \, \sqcup\, S_3 \, \sqcup \, S_4'' = \bigsqcup_{l=1}^{4} T_l \, .$$
The following proposition is easy to prove:
\begin{proposition}
\label{SintT}
\begin{eqnarray}
T_3 \cap S_1 &= & S_1 \setminus T_1 = \bigg \{ (k-j)(k-j+1) -j+1\, \bigg | \, 1 \le j \le \bigg[\frac{k}{2} \bigg] \bigg\}\,,  \nonumber  \\[0.3cm]
T_4 \cap S_2 &= & S_2 \setminus ( T_2 \sqcup \{k^2\}) = \bigg\{ (k-j)^2-j  \, \bigg | \, 1 \le j \le  \bigg[\frac{k-1}{2} \bigg] \bigg\} \, .  \nonumber
\end{eqnarray}
In particular, we have 
$$  T_3=S_3 \sqcup (T_3 \cap S_1) \, , \, T_4=S_4'' \sqcup ( T_4 \cap S_2)\, .$$
\end{proposition}
\subsection{$\c^*$-fixed point of $\mathfrak M^{\tau}_{\z_2}( k^2-k,k^2; \mathrm{1})$}
Recall that this point is represented by
$(X_0, X_1,Y_0,Y_1; v_1,w_1)$, where $X_0, Y_1 \in \mathrm{Mat}_{n_0 \times n_1}(\c)$,  
$X_1, Y_0  \in \mathrm{Mat}_{n_1\times n_0}(\c)$, $v_1 \in \c^{n_1}$ and $w_1 \in (\c^{n_1})^*$
such that
\begin{eqnarray}
\label{GenCM1}
&& X_0Y_0 - Y_1X_1 + \tau_0 \,\mathrm{Id}_{n_0} = 0 ,\\[0.1cm]
\label{GenCM2}
&& X_1 Y_1 - Y_0X_0+ \tau_1\, \mathrm{Id}_{n_1} = v_1w_1 .
\end{eqnarray}

\begin{notation}
Let us introduce two more notations which will be in this and later sections: $a:=\tau_1, b:=\tau_0+\tau_1$.
\end{notation}

\subsubsection{Matrix $Y_1$}
Let $\e_1,\cdots,\e_{k(k-1)}$ be the standard basis for $\c^{k(k-1)}$. Then define
vectors $\v_1,\cdots,\v_{k^2}$ as follows.

 For $S_1$ indices:
\begin{eqnarray}
\label{type11}
 \v_{i(i+1)-j+1} &:=& (2j-1)\,b \, \e_{i^2+1-j} \nonumber \\[0.3cm]
 &&+ \big(a+2(k-i+j-1)\,b \big)   
 \bigg( \sum_{l=1}^{\big[\frac{i-j+1}{2}\big]} \e_{(i-l+1)^2-i-j+1}
+ \sum_{l=1}^{ \big[\frac{i-j}{2}\big]} \e_{(i-l)(i-l+1)-i-j+1} \bigg) \, ,
\end{eqnarray}
where $[a]$ is the integer part of $a$.

For $ S_2$:
\begin{eqnarray}
\label{type12}
 \v_{(i+1)^2-j} &:= & 2j \, b \, \e_{i(i+1)-j+1} \nonumber \\[0.3cm]
 &&
 + \big(a+2(k-i+j-1)\,b \big) 
\bigg( \sum_{l=1}^{\big[\frac{i-j+1}{2}\big]} \e_{(i-l+1)(i-l+2)-i-j} 
+ \sum_{l=1}^{ \big[\frac{i-j}{2}\big]} \e_{ (i-l+1)^2-i-j}\bigg) \, .
\end{eqnarray}
For $S_3$:
\begin{eqnarray}
\label{type13}
  \v_{i(i+1)-j+1} &:= & - \big( a+ (2k-2i-1)\, b \big) \, 
\e_{i^2-j+1}  \nonumber \\[0.3cm]
&&
-  \big( a+ 2(k-i+j-1) \, b \big)  
\bigg( \sum_{l=0}^{k-i-1} 
\e_{(i+l+1)^2-i-j+1} +  \sum_{l=0}^{k-i-2}
 \e_{(i+l+1)(i+l+2)-i-j+1} \bigg)\, .
\end{eqnarray}
For $S_4$:
\begin{eqnarray}
\label{type14}
  \v_{(i+1)^2-j}& := &
 - \big( a+ 2(k-i-1)\, b \big) \,  \e_{i(i+1)-j+1} \nonumber \\[0.3cm]
 &&
 -  \big( a+ 2(k-i+j-1) \, b \big)  
  \bigg( \sum_{l=0}^{k-i-2} 
 \e_{(i+l+1)(i+l+2)-i-j}
  + \sum_{l=0}^{k-i-2} 
 \e_{ (i+l+2)^2-i-j} \bigg)\,.
\end{eqnarray}
Define  $Y_1$ as an $k(k-1)\times k^2$ matrix with columns consisting of $\v_i$'s
\begin{equation*}
Y_1:= [\v_1\, \v_2\, \cdots \,\v_{k^2}].
\end{equation*}

\subsubsection{Matrix $Y_0$}
Let $\f_1,\cdots,\f_{k^2}$ be the standard basis for  $\c^{k^2}$. Then we introduce the following vectors.
For $T_1$:
\begin{eqnarray}
\label{type21}
 \u_{i(i+1)-j+1} &:=&(a+(2j-1) \, b ) \, \f_{i(i+1)-j+1}  \nonumber \\[0.3cm]
 &&
 + \big(a+2(k-i+j-1)\,b \big)  
\bigg( \sum_{l=1}^{\big[\frac{i-j+1}{2}\big]} \f_{(i-l+1)(i-l+2)-i-j} 
+ \sum_{l=1}^{ \big[\frac{i-j}{2}\big]} \f_{(i-l+1)^2-i-j} \bigg) \, .
\end{eqnarray}
For $T_2$:
\begin{eqnarray}
\label{type22}
\u_{(i+1)^2-j} & = &  ( a + 2j \, b) \, \f_{(i+1)^2-j}   \nonumber \\[0.3cm]
&&
+  \big( a+ 2(k-i-j-1)\,b\big)
 \bigg( \,  \sum_{l=1}^{\big[\frac{i-j+1}{2}\big] } \f_{(i-l+2)^2-i-j-1} +
\sum_{l=1}^{\big[\frac{i-j}{2}\big] }\f_{(i-l+1)(i-l+2)-i-j-1} \, \bigg)  \,.
\end{eqnarray}
For $T_3$:
\begin{eqnarray}
\label{type23}
  \u_{i(i+1)-j+1}&:= &
 - (2k-2i-1)\, b \,  \f_{i(i+1)-j+1}\nonumber \\[0.3cm]
  &&
 -  \big( a+ 2(k-i+j-1) \, b \big)  
   \bigg( \sum_{l=0}^{k-i-1}  
 \f_{(i+l+1)(i+l+2)-i-j} + \sum_{l=0}^{k-i-2}  
 \f_{ (i+l+2)^2-i-j} \bigg)\,.
\end{eqnarray}
For $T_4$:
\begin{eqnarray}
\label{type24}
  \u_{(i+1)^2-j} &:=& -  2(k-i-1)\, b \, 
\f_{(i+1)^2-j} \nonumber \\[0.3cm]
&&
-   \big( a+ 2(k-i+j-1) \, b \big)    
  \bigg( \sum_{l=0}^{k-i-2}
\f_{(i+l+2)^2-i-j-1}+  \sum_{l=0}^{k-i-2}
 \f_{(i+l+2)(i+l+3)-i-j-1}\bigg)\,.
\end{eqnarray}
Thus, we can define $Y_0$ as a $k^2 \times k(k-1)$ matrix with columns consisting of $\u_i$'s
\begin{equation*}
Y_0 \, := \, [\, \u_1\, \u_2\, \cdots\, \u_{k(k-1)} \,].
\end{equation*}

\subsubsection{Matrices $X_0$ and $X_1$} We define
$$ X_0 := [\, \e_1 \, \cdots\, \e_{k(k-1)} \mathbf{0} \, \cdots\,\mathbf{0}\, ] ,$$
where the number of $\mathbf{0}$ column vectors is exactly $k$. For 
$$  X_1:= [ \, \mathbf{w}_1\, \cdots\, \mathbf{w}_{k(k-1)}\, ] \, ,  $$
where $\mathbf{w}_i$'s are defined as follows. For $1 \le i \le k-1$ 
and $1 \le j \le 2i$,
$$\mathbf{w}_{i(i-1)+j}:= \f_{i^2+j} \, .$$
\subsubsection{The vector $v_1$ and the covector $w_1$}
\label{subsub524}
 To compute $v_1$ and $w_1$,
first we need to see effects of multiplications by matrices  $X_0$ and $X_1$.
Indeed, one can easily check that the multiplication by $X_0$ from the left to 
the $k^2 \times k(k-1)$ matrix $A$ produces a $k(k-1)\times k(k-1)$ matrix obtained from $A$ by removing the last 
$k$ rows. If we multiply by $X_0$ from the right to $A$ then this gives 
the $k^2\times k^2$ matrix obtained  by adding $k$ zero columns to the right end of $A$.

Next, if we multiply by $X_1$ from the left to a $k(k-1)\times k^2$ matrix $B$, then  one obtains
the $k^2 \times k^2$ matrix whose $i^2$-th $(i=1,\cdots,k)$ rows are zero rows and if we remove those
rows, we get exactly $B$. While multiplication by $X_1$ from the right to $B$ produces
the $k(k-1)\times k(k-1)$ matrix obtained from $B$ by removing columns on the $i^2$-th positions
where  $i=1,\cdots,k$.

Thus, one has 
\begin{eqnarray}
\label{XY1}
X_0 Y_0 & = & [\, \tilde{\u}_1 \, \tilde{\u}_2\, \cdots \, \tilde{\u}_{k(k-1)}\, ]  \, ,  \quad
 Y_0 X_0 = [ \u_1 \, \u_2 \, \cdots\, \u_{k(k-1)} \, \mathbf{0}\, \cdots\, \mathbf{0}\, ] ,  \\ [0.2cm]
 \label{XY2}
X_1Y_1 &=& [ \, \tilde{\v}_1\, \tilde{\v}_2 \, \cdots\, \tilde{\v}_{k^2} \, ] \, ,
\quad Y_1X_1= [ \, \v'_1\,   \v'_2\, \cdots\, \v'_{k(k-1)}\,] ,
\end{eqnarray}
where $\tilde{\u}, \tilde{\v}$ and $\v'$ are defined as follows.
First, for $i \in T_1 \sqcup T_2$,  the vector  $\tilde{\u}_i$,
can be obtained from the corresponding $\u_i$ by replacing
the standard vectors $\f \in \c^{k^2}$ in  \eqref{type21} or \eqref{type22} by
the standard vectors $\e \in \c^{k(k-1)}$. 
 Next, the vector $\tilde{\u}_i$, for $ i \in T_3$, can be presented as  in \eqref{type23}, 
where $\f$ is replaced by $\e$ and $k-i-1$, the upper limit in the first summation, is replaced  by $k-i-2$.
Then, the vector  $\tilde{\u}_i$  for  $i \in T_4$, can be presented as  in \eqref{type24}, where $\f$ is replaced by $\e$
and $k-i-2$, the upper limit in the second summations, is replaced  by $k-i-3$.
 
 Next, we define $\tilde{\v}_i$.  
 For $S_1$ indices:
 \begin{eqnarray}
\label{type11t}
 \tilde{\v}_{i(i+1)-j+1}& :=& (2j-1)\,b \, \f_{i(i+1)-j+1} \nonumber \\[0.3cm]
&& + \big(a+2(k-i+j-1)\,b \big) 
 \bigg( \sum_{l=1}^{\big[\frac{i-j+1}{2}\big]} \f_{(i-l+1)(i-l+2)-i-j}
+ \sum_{l=1}^{ \big[\frac{i-j}{2}\big]} \f_{(i-l+1)^2-i-j} \bigg) \, .
\end{eqnarray}
For $S_2$ indices: 
\begin{eqnarray}
\label{type12t}
 \tilde{\v}_{(i+1)^2-j} &:=& 2j \, b \, \f_{(i+1)^2-j}  \nonumber \\[0.3cm]
 && + \big(a+2(k-i+j-1)\,b \big) 
 \bigg( \sum_{l=1}^{\big[\frac{i-j+1}{2}\big]} \f_{(i-l+2)^2-i-j-1} 
+ \sum_{l=1}^{ \big[\frac{i-j}{2}\big]} \f_{(i-l+1)(i-l+2)-i-j-1} \bigg) \, .
\end{eqnarray}
For $S_3$ indices:
\begin{eqnarray}
\label{type13t}
  \tilde{\v}_{i(i+1)-j+1}&:=& - \big( a+ (2k-2i-1)\, b \big) \, 
\f_{i(i+1)-j+1} \nonumber \\[0.3cm]
&&
-  \big( a+ 2(k-i+j-1) \, b \big) 
  \bigg( \sum_{l=0}^{k-i-1} 
\f_{(i+l+1)(i+l+2)-i-j} +  \sum_{l=0}^{k-i-2}
 \f_{ (i+l+2)^2-i-j} \bigg) \,.
\end{eqnarray}
For $S_4$ indices: 
\begin{eqnarray}
\label{type14t}
 \tilde{\v}_{(i+1)^2-j} &:= &
 -  \big( a+ 2(k-i-1)\, b \big)\,  \f_{(i+1)^2-j}  \nonumber \\[0.3cm] 
 && -  \big( a+ 2(k-i+j-1) \, b \big)
  \bigg( \sum_{l=0}^{k-i-2} 
 \f_{ (i+l+2)^2-i-j-1} + \sum_{l=0}^{k-i-2} 
 \f_{(i+l+2)(i+l+3)-i-j-1} \bigg) \,.
\end{eqnarray}
Finally, define $\v_i'$ as follows 
\begin{eqnarray}
\label{vp1}
 T_1\sqcup T_3\, :  \quad  \v_{i(i+1)-j+1}'=\v_{(i+1)^2-j} \, , \\[0.2cm]
 \label{vp2}
  T_2 \sqcup T_4 \, : \quad \v_{(i+1)^2-j}'= \v_{(i+1)(i+2)-j} \, .
\end{eqnarray}
Thus, by using this and the above description of $\tilde{\u}_i$'s,  we obtain
\begin{equation}
\label{x0y0}
 X_0Y_0-Y_1X_1 \, = [\, \tilde{\u}_1- \v_1' \, \cdots\, \tilde{\u}_{k(k-1)} - \v_{k(k-1)}'\,]= (a-b)\,  \mathrm{Id}_{k(k-1)} ,
 \end{equation}
where $\mathrm{Id}_{k(k-1)}$ is the $k(k-1) \times k(k-1)$ identity matrix.

Next by \eqref{XY1} and \eqref{XY2}, one can write  
\begin{equation}
\label{x1y1x0y0}
 X_1Y_1-Y_0X_0 + a \, \mathrm{Id}_{k^2}  
\end{equation}
 as
\begin{eqnarray}
  && [ \, \tilde{\v}_1- \u_1 - a\, \f_1  \quad  \cdots\quad   \tilde{\v}_{k(k-1)} - \u_{k(k-1)}+a\, \f_{k(k-1)}  \,  \nonumber \\[0.3cm]
  \label{compcomm}
  &&\quad \quad   \tilde{\v}_{k(k-1)+1} + a\ \f_{k(k-1)+1} \quad \cdots \quad \tilde{\v}_{k^2} +a\, \f_{k^2}\,  ]\, .  \nonumber
 \end{eqnarray}
To compute this expression, we use \eqref{type21}-\eqref{type24}, \eqref{type11t}-\eqref{type14t} 
and the following decomposition (see Proposition~\ref{SintT}
and discussion before)
$$ \{1,2,\cdots,k^2\} =  T_1 \sqcup T_2 \sqcup S_3 \sqcup S_4
 \sqcup (T_3 \cap S_1)  \sqcup (T_4 \cap S_2) \sqcup \{k^2\} \, .$$
Explicitly, for $T_1$ indices, one computes 
\begin{equation}
\label{diffT1}
\tilde{\v}_{i(i+1)-j+1} - \u_{i(i+1)-j+1} +a\,  \f_{i(i+1)-j+1} = 0 \, .
\end{equation}
The same vanishing result is true for $T_2, S_3$ and $S_4$ indices.

Now for $T_3 \cap S_1$, we have
\begin{eqnarray}
\label{difft3s1}
&&\tilde{\v}_{(k-j)(k-j+1)-j+1} - \u_{(k-j)(k-j+1)-j+1} \nonumber\\[0.3cm]
 &=&(4j-2)\, b\, \f_{(k-j)(k-j+1)-j+1} \nonumber\\[0.3cm]
&&+ \big(a+(4j-2)\,b\big) \bigg( \sum_{l=1}^{\frac{k+1}{2}-j} \f_{(k-j-l+1)(k-j-l+2)-k} 
 +\sum_{l=1}^{\frac{k-1}{2}-j} \f_{(k-j-l+1)^2-k} \bigg) \nonumber \\[0.3cm]
 \label{t3caps1ex}
&&+ \big( a+ (4j-2)\, b \big) \bigg( \sum_{l=0}^{j-1} \f_{(k-j+l+1)(k-j+l+2)-k} + \sum_{l=0}^{j-2} \f_{(k-j+l+2)^2-k} \bigg).
\end{eqnarray}
By change of variables in the above summations, one gets
  \begin{eqnarray}
&&\tilde{\v}_{(k-j)(k-j+1)-j+1} - \u_{(k-j)(k-j+1)-j+1}  + a\, \f_{(k-j)(k-j+1)-j+1} \nonumber\\[0.3cm]
\label{difft3s1p}
&= &\big(a+(4j-2)\,b\big) \bigg( \sum_{l=1}^{\frac{k+1}{2}} \f_{(k-l+1)(k-l+2)-k} 
 +\sum_{l=1}^{\frac{k-1}{2}} \f_{(k-l+1)^2-k} \bigg)  .
 \end{eqnarray}
Similarly, for $T_4 \cap S_2$, we get
\begin{eqnarray}
&&
\tilde{\v}_{(k-j)^2-j} - \u_{(k-j)^2-j} + a\, \f_{(k-j)^2-j} \nonumber\\[0.3cm]
&=& \big(a+4j\,b\big) \bigg( \sum_{l=1}^{\frac{k+1}{2}} \f_{(k-l+1)(k-l+2)-k} 
 +\sum_{l=1}^{\frac{k-1}{2}} \f_{(k-l+1)^2-k} \bigg)  \, 
 \end{eqnarray}
and finally for $k^2$, one has
\begin{equation}
\label{diffk2}
\tilde{\v}_{k^2} + a \, \f_{k^2} = a\, \bigg( \sum_{l=1}^{\frac{k+1}{2}} \f_{(k-l+1)(k-l+2)-k} +
\sum_{l=1}^{\frac{k-1}{2}} \f_{(k-l+1)^2-k} \bigg).
\end{equation}
Thus, by \eqref{diffT1}-\eqref{diffk2} all the columns \eqref{x1y1x0y0} except those from
the set $(T_3 \cap S_1) \sqcup (T_4 \cap S_2) \sqcup \{k^2\}$ are zero. Moreover, we can
state
\begin{proposition}
\label{v1w1prop}
\eqref{x1y1x0y0} is a rank one matrix and the $(i,j)$-th entry is non-zero if and only if
$i, j \in (T_3 \cap S_1) \sqcup (T_4 \cap S_2) \sqcup \{k^2\}$.
\end{proposition}
\begin{proof}
First,  \eqref{difft3s1p}-\eqref{diffk2} are linearly dependent  
vectors and hence \eqref{x1y1x0y0} is a rank one matrix.

The second assertion is also easy to show. The vectors in \eqref{difft3s1p}-\eqref{diffk2}
are indexed by the set $(T_3 \cap S_1) \sqcup (T_4 \cap S_2) \sqcup \{k^2\}$ and they  are colinear to
\begin{equation}
\label{colin1}
   \sum_{l=1}^{\frac{k+1}{2}} \f_{(k-l+1)(k-l+2)-k} 
 +\sum_{l=1}^{\frac{k-1}{2}} \f_{(k-l+1)^2-k} \, .
 \end{equation}
So one only needs to show
 \begin{eqnarray}
(T_3 \cap S_1) \sqcup (T_4 \cap S_2) \sqcup \{k^2\} & = & \bigg\{ (k-l+1)(k-l+2)-k , \, 1 \le l \le  \frac{k+1}{2} \bigg\}  \nonumber\\
\label{setidentity2}
&& \bigsqcup \bigg\{ (k-l+1)^2-k , \, 1 \le l \le  \frac{k-1}{2} \bigg\}.
\end{eqnarray}
For  $ 2 \le l \le \frac{k+1}{2}$,
$$ (k-l+1)(k-l+2)-k = (k-l+1)^2-l+1, $$
which are exactly the elements of  $T_4 \cap S_2$, while for $l=1$, $ (k-l+1)(k-l+2)-k = k^2$.
On the other hand, for $1 \le l \le \frac{k-1}{2}$,
$$ (k-l+1)^2-k=(k-l)(k-l+1)-l+1,$$
which are exactly the elements of $T_3 \cap S_1$.
\end{proof}

\begin{remark}
Note that the identity \eqref{setidentity2} can be rewritten as
 \begin{equation}
(T_3 \cap S_1) \sqcup (T_4 \cap S_2) \sqcup \{k^2\}  =  \bigg\{ l(l+1)-k , \,  \frac{k+1}{2} \le l \le k \bigg\}
 \bigsqcup \bigg\{ l^2-k , \, \frac{k+3}{2} \le l \le k  \bigg\},  \nonumber
 \end{equation}
and \eqref{colin1} is equivalent to
\begin{equation}
\label{colin2}
   \sum_{l=\frac{k+1}{2}}^{k} \f_{l(l+1)-k} 
 +\sum_{l=\frac{k+3}{2} }^{k} \f_{l^2-k} \, .
 \end{equation}
\end{remark}

Thus, from Proposition~\ref{v1w1prop} and the Remark after it, we conclude:
\begin{corollary}
\label{zerodimpoint}
Let $v_1$ and $w_1$  be defined by
\begin{eqnarray*}v_1 & :=& \sum_{l=\frac{k+1}{2}}^{k} \f_{l(l+1)-k} 
 +\sum_{l=\frac{k+3}{2} }^{k} \f_{l^2-k} \, ,\\
w_1 &:=& \sum_{l=\frac{k+1}{2}}^{k} \big(a+4(k-l) \,b \big)\, \f_{l(l+1)-k}^{\mathrm{t}} 
 +\sum_{l=\frac{k+3}{2} }^{k} \big(a+(4k-4l+2) \,b \big) \, \f_{l^2-k}^{\mathrm{t}} ,
 \end{eqnarray*}
where $\mathrm{t}$  is taking the transpose of $\f$.
Then $(X_0, X_1; Y_0, Y_1; v_1,w_1)$ is the point of 
$\mathfrak{M}^{\tau}_{\z_2} (k^2-k, k^2; \mathrm{1} )$.
\end{corollary}

\subsection{$\c^*$-fixed point of $\mathfrak{M}_{\z_2}^{\tau}(n-k, n; \mathrm{1})$}

Using the above description of the nilpotent point of $\mathfrak{M}_{\z_2}^{\tau}(k^2-k, k^2; \mathrm{1})$, 
 we construct a $\c^*$-fixed point  of $ \mathfrak{M}_{\z_2}^{\tau}(n-k, n; \mathrm{1})$.

\subsubsection{Matrices $X_0$ and $X_1$} They  are obtained from their 
$\mathfrak{M}_{\z_2}^{\tau}(k^2-k, k^2; \mathrm{1})$ counterparts
by forming two by two block diagonal matrices where on the top left corner we place 
the idenity matix of size $n-k^2$.
More explicitly, let $\e_1, \e_2,\cdots,\e_{n-k}$ and $\f_1, \f_2,\cdots, \f_n$
  be the standard basis of $\c^{n-k}$ and $\c^n$ respectively. Then 
$$X_0\, :=\, [\, \e_1 \, \e_2 \, \cdots\, \e_{n-k} \, \mathbf{0}\,\cdots\, \mathbf{0}\,]\, , $$ 
where the number of $\mathbf{0}$ columns is exactly $k$, while $X_1$ is 
$$ X_1:= [ \, \mathbf{w}_1\, \mathbf{w}_2\,  \cdots\, \mathbf{w}_{n-k}\, ] \, ,  $$
where $\mathbf{w}_i$'s are defined as follows.
For $ 1 \le s \le n-k^2$, $\mathbf{w}_s:= \f_{s+1}$,
and for $1 \le i \le k-1$ and $1 \le j \le 2i$,
$$ \mathbf{w}_{i(i-1)+j+n-k^2}:= \f_{i^2+j+n-k^2} \, .$$

\subsubsection{Matrix $Y_0$} 
We introduce vectors $\h_i \in \c^{n} (1 \le i \le n-k)$ so that
$$Y_0\, := \, [ \, \h_1\, \h_2 \, \cdots \, \h_{n-k}\,]\, .$$
First, we recall that $Y_0:=[\,\u_1\, \u_2 \cdots \u_{k(k-1)}] \in \mathfrak{M}_{\z_2}^{\tau}(k^2-k, k^2; \mathrm{1})$, where
$\u_i$'s are defined as in \eqref{type21}-\eqref{type24}. 
Let $\h_i' (i=1,\cdots,k^2-k)$ be vectors in $\c^{n}$ defined as follows. If $1 \le i \le k-1$, we set
$\h_i'$ to be defined by the same formulas as $\u_i$.
If $k \le i \le k^2$ then $\h_i'$ is defined as $\u_i$, except indices  of $\f$ shifted down by $n-k^2$, 
that is, a term $\f_i$ should be replaced by  $\f_{i+n-k^2}$.
Next, let $T_{3d}$ and $T_{4d}$ be  subsets of $T_3$ and $T_4$ respectively, consisting elements for which $j=i$:
$$ T_{3d}\, =\, \bigg\{\, i^2+1\,  \bigg | \, \bigg[\frac{k+1}{2}\bigg] \le i \le k-1 \bigg\}\, , \, 
T_{4d}\, =\, \bigg\{\, (i+1)^2 - i  \bigg | \, \bigg[\frac{k}{2}\bigg] \le i \le k-2 \bigg\}\, .$$
Then define
\begin{equation}
 \h_s :=  \begin{cases} 
     \h_s',  &  \mbox{ for } 1 \le s \le k-1, \\[0.2cm]
      \big( a+ (s-1)\, b \big) \e_s, &  \mbox{ for }  k \le s \le n-k^2+k-1, \\[0.2cm]
       \h_{s'-n+k^2},  &  \mbox{ for } s \in [n-k^2+k, n-k] \setminus \big( n-k^2+ T_{3d} \cup T_{4d} \, \big). \\
   \end{cases}  \nonumber
   \end{equation}
It remains to define $\h_s$ for $s\in n-k^2+ T_{3d} \cup T_{4d} $. For $s\in n-k^2+ T_{3d} $,
\begin{eqnarray*}
   \h_s &:= &
 - (2k-2i-1)\, b \, \f_{i^2+1+n-k^2}\nonumber \\[0.3cm]
 &&  
 -  \big( a+ 2(k-1+n-k^2) \, b \big) 
  \bigg( \sum_{l=0}^{k-i-1}  
\f_{(i+l+1)(i+l+2)-2i+n-k^2} + \sum_{l=0}^{k-i-2}  
\f_{ (i+l+2)^2-2i+n-k^2} \bigg),
\end{eqnarray*}
and for $s \in n-k^2+ T_{4d}$,
\begin{eqnarray*}
  \h_s &:=& -  2(k-i-1)\, b \, 
\f_{(i+1)^2-i+n-k^2}  \nonumber \\[0.2cm] 
&&
-   \big( a+ 2(k-1+n-k^2) \, b \big) 
  \bigg( \sum_{l=0}^{k-i-2}
\f_{(i+l+2)^2-2i-1+n-k^2}+  \sum_{l=0}^{k-i-2}
 \f_{(i+l+2)(i+l+3)-2i-1+n-k^2}\bigg).
\end{eqnarray*}

\subsubsection{Matrix $Y_1$} We introduce vectors $\g_1,\g_2,\cdots, \g_n \in \c^{n-k}$ so that
$$Y_1\, :=\, [\, \g_1\, \g_2 \, \cdots \, \g_n \ ].$$
Recall that $Y_1= [ \v_1\, \v_2 \ \cdots\v_{k^2} ] \in \mathfrak{M}_{\z_2}^{\tau}(k^2-k, k^2; \mathrm{1})$,
where $\v_i$'s are defined in \eqref{type11}-\eqref{type14}.
As above, first,  we introduce vectors $\g_i' \in \c^{n-k} \,(i=1,\cdots,k^2)$.
If $1 \le i \le k-1$, we set $\g_i'$ to be defined by the same formula as $\v_i$.
If $k \le i \le k^2-k$ then $\g_i'$ is defined as $\u_i$ except
indices of $\e$ are shifted down by $n-k^2$.
Then we define
\begin{equation*}
 \g_s :=  \begin{cases} 
     \g_s' , &  \mbox{ for } 1 \le s \le k-1, \\[0.2cm]
     (s-1)\, b\, \e_{s-1} ,&  \mbox{ for }  k \le s \le n-k^2+k-1, \\[0.2cm]
     \g_{s-n+k^2}'  ,&  \mbox{ for }  s=n-k^2+k, \\[0.2cm]
     \g_{s-n+k^2+1}' + (n-k^2)\, b\, \e_{n-k^2+k-1},  &  \mbox{ for }  s=n-k^2+k+1,\\[0.2cm]
       \g_{s-n+k^2}'  ,& \begin{array}{l}\!\!\! \mbox{ for } s \in [n-k^2+k+2, n-k]   \\
                                \!\!\! \mbox { but }  s\notin  \big( n-k^2+ S_{3d} \cup S_{4d} \, \big),
                              \end{array} \\
   \end{cases} 
   \end{equation*}
where
$$ S_{3d}\, =\, \bigg\{\, i^2+1\,  \bigg | \, \bigg[\frac{k}{2}\bigg]+1 \le i \le k-1 \bigg\}\, , \quad
S_{4d}\, =\, \bigg\{\, (i+1)^2 - i \, \bigg | \, \bigg[\frac{k+1}{2}\bigg] \le i \le k-1 \bigg\}\, .$$
For $s\in n-k^2+ S_{3d}$,
\begin{eqnarray*}
  \g_s &:= & - \big( a+ (2k-2i-1)\, b \big) \, 
\e_{i^2-i+1+n-k^2}  \nonumber \\[0.2cm]
&&
-  \big( a+ 2(k-1+n-k^2) \, b \big) 
\bigg( \sum_{l=0}^{k-i-1} 
\e_{(i+l+1)^2-2i+1+n-k^2} +  \sum_{l=0}^{k-i-2}
 \e_{(i+l+1)(i+l+2)-2i+1+n-k^2} \bigg).
\end{eqnarray*}
For $s\in n-k^2+S_{4d}$,
\begin{eqnarray*}
  \g_s & := &
 - \big( a+ 2(k-i-1)\, b \big) \,  \e_{i^2+1+n-k } \nonumber \\[0.2cm]
 &&
 -  \big( a+ 2(k-1+n-k^2) \, b \big) 
  \bigg( \sum_{l=0}^{k-i-2} 
 \e_{(i+l+1)(i+l+2)-2i+n-k^2}
  + \sum_{l=0}^{k-i-2} 
 \e_{ (i+l+2)^2-2i+n-k^2} \bigg).
\end{eqnarray*}

\subsubsection{The vector $v_1$ and the covector $w_1$}

We can carry out the same type of computations like in the subsection~\ref{subsub524}
to find $v_1$ and $w_1$ and prove the following proposition:

\begin{proposition} 
\label{nilpointhd}
Let $X_0, X_1,Y_0,Y_1$ be defined as above, and let 
\begin{eqnarray*}
v_1  & := & \sum_{l=\big[\frac{k}{2}\big]+1}^{k} \f_{l(l+1)-k+n-k^2} 
 +\sum_{l=\big[\frac{k+1}{2}\big]+1}^{k} \f_{l^2-k+n-k^2} \, ,\\[0.2cm]
w_1 & := & \sum_{l=\big[\frac{k}{2}\big]+1}^{k} \big(a+4(k-l) \,b \big)\, \f_{l(l+1)-k+n-k^2}^{\mathrm{t}} \\[0.1cm]
 && + \sum_{l=\big[\frac{k+1}{2}\big]+1}^{k} \big(a+(4k-4l+2) \,b \big) \, \f_{l^2-k+n-k^2}^{\mathrm{t}} +
 (n-k^2)\, b\, \f_{\big[\frac{k^2}{4}\big]+1+n-k^2}^{\mathrm{t}}.
\end{eqnarray*}
Then $(X_0, X_1; Y_0, Y_1; v_1,w_1)$ is a $\c^{*}$-fixed point of 
$\mathfrak{M}^{\tau}_{\z_2} (n-k, n; \mathrm{1} )$.
\end{proposition}

\subsection{$\c^*$-fixed points of $\mathfrak{M}^{\tau}_{\z_2}(n-k, n; \mathrm{0})$}

$\c^*$-fixed points of quiver varieties $\mathfrak{M}^{\tau}_{\z_2}(n-k, n; \mathrm{0})$ 
can be defined in a similar way as has been done for   $\mathfrak{M}^{\tau}_{\z_2}(n-k, n; \mathrm{1})$.
We leave the details to interested reader.

\section{Computations of $\kappa$ for $\c^*$-fixed points }
\la{sec7}

Recall that for $(X,Y, v, w) \in \mathfrak{M}^{\tau}_{\Gamma}$ the corresponding $\kappa$ can be computed as follows
$$  \kappa=1-\sum_{l,q \ge 0} (w Y^l X^q v)\, y^{-l-1} x^{-q-1} \, .$$
If $(X,Y, v, w)$ is a $\c^*$-fixed point then $w Y^l X^q v=0$ unless $l=q$, since
the $\c^*$-action is given by $(X,Y,v,w) \mapsto (\lambda X, \lambda^{-1}Y,v,w)$, where $\lambda \in \c^{\ast}$.
Thus, we have
 $$  \kappa=1-\sum_{l \ge 0} (w Y^l X^l v)\, y^{-l-1} x^{-l-1} \, .$$
To compute coefficients $ w Y^l X^l v$ we need:

\begin{proposition}
\label{coeffkappa}
Let $(X,Y,v,w)$ be a  point of  $\mathfrak{M}^{\tau}_{\z_2}(V, \mathcal{U}_k)$.
Then for  $l \in \z_{\ge 0}$,
$$ wY^l X^l v = w_1 \, \prod_{q=0}^{\big[\frac{l-1}{2}\big]}\big(Y_0X_0+q(\tau_0+\tau_1)\, \mathrm{Id}\big)
 \cdot  \prod_{q=1}^{\big[\frac{l}{2}\big]} \big(Y_0X_0+q \tau_0+ (q-1)\tau_1) \, \mathrm{Id}\big)\, v_1.$$
\end{proposition}

\begin{proof} 
We may assume that $l$ is odd and leave the proof of the even case to the interested reader. 
Using \eqref{GenCM1} and \eqref{GenCM2} we have
\begin{eqnarray}
 wY^l X^l v & = &w_1 \, (Y_0 Y_1)^{\frac{l-1}{2}}Y_0 X_0 (X_{1}X_0)^{\frac{l-1}{2}}\, v_1 \nonumber\\[0.1cm]
 & = & w_1 \, (Y_0 Y_1)^{\frac{l-1}{2}} \, (X_1Y_1 + \tau_1\mathrm{Id}_{n_1} -v_1w_1)  (X_{1}X_0)^{\frac{l-1}{2}}\, v_1 \nonumber\\[0.1cm]
 &=&  w_1 \, (Y_0 Y_1)^{\frac{l-1}{2}} \, (X_1Y_1 + \tau_1\mathrm{Id}_{n_1})  (X_{1}X_0)^{\frac{l-1}{2}}\, v_1 - 
 wY^{\frac{l-1}{2}}v \cdot w X^{\frac{l-1}{2}} v \nonumber\\[0.1cm]
 &=& w_1 \, (Y_0 Y_1)^{\frac{l-3}{2}} Y_0\, (Y_1X_1)(Y_1X_1+ \tau_1\mathrm{Id}_{n_0})\,X_0  (X_{1}X_0)^{\frac{l-3}{2}}\, v_1 \nonumber\\[0.1cm]
 &=& w_1 \, (Y_0 Y_1)^{\frac{l-3}{2}} Y_0\, (X_0Y_0 + \tau_0 \mathrm{Id}_{n_0}) (X_0Y_0 + (\tau_0+\tau_1)\mathrm{Id}_{n_0})
 \,X_0  (X_{1}X_0)^{\frac{l-3}{2}}\, v_1 \nonumber\\[0.1cm]
 &=& w_1 \, (Y_0 Y_1)^{\frac{l-3}{2}} \, (Y_0X_0)(Y_0X_0+ \tau_0 \mathrm{Id}_{n_1}) (Y_0X_0+ (\tau_0+\tau_1)\mathrm{Id}_{n_1})
 \, (X_{1}X_0)^{\frac{l-3}{2}}\, v_1. \nonumber 
 \end{eqnarray}
Now if $l=3$ then we are done, otherwise we repeatedly use \eqref{GenCM1} and \eqref{GenCM2}
to obtain our formula.
\end{proof}

\subsection{$\kappa$ for the $\c^*$-fixed point of $\mathfrak{M}^{\tau}_{\z_2}(k^2-k, k^2; \mathrm{1})$}

In view of Proposition~\ref{coeffkappa}, we need to see the action of $Y_0X_0$
on $v_1$. Recall that in this case 
$$v_1 \, = \sum_{l=\big[\frac{k}{2}\big]+1}^{k} \f_{l(l+1)-k} 
 +\sum_{l=\big[\frac{k+1}{2}\big]+1 }^{k} \f_{l^2-k} \, .$$
Using \eqref{XY1}, for $\big[\frac{k}{2}\big]+1 \le l \le k -1$, we have
\begin{eqnarray}
\label{y0x0f1}
  Y_0X_0\, \f_{l(l+1)-k} =\u_{l(l+1)-k} & = & -2(k-l)\, b\, \f_{l(l+1)-k}   \nonumber\\[0.15cm]
 && - \big( a +4(k-l)\, b \big)  \bigg(\, \sum_{s=l+1}^{k} \f_{s^2-k} + \sum_{s=l+1}^{k} \f_{s(s+1)-k} \, \bigg), 
  \end{eqnarray}
and for $l=k$, $l(l+1)-k=k^2$, we have
\begin{equation}
\label{y0x0f2}
Y_0X_0\, \f_{k^2}=0 \, . 
\end{equation}
For $\big[\frac{k+1}{2}\big]+1 \le l \le k$, we have
\begin{eqnarray}
\label{y0x0f3}
 Y_0X_0\, \f_{l^2-k} = \u_{l^2-k} &=&-(2k-2l+1)\, b\, \f_{l^2-k} \nonumber\\[0.15cm]
&&- \big(a+2(2k-2l+1)\, b\big) \bigg(\, \sum_{s=l+1}^{k} \f_{s^2-k} + \sum_{s=l}^{k} \f_{s(s+1)-k} \, \bigg). 
\end{eqnarray}

Let $\{ \tilde{\f}_i\}_{i=1}^k$ be a set of vectors defined as follows:
\begin{eqnarray}
 \tilde{\f}_{2l-\big(\big[\frac{k+1}{2} \big]-\big[\frac{k}{2}\big]\big)}& := & \f_{\big(l+\big[\frac{k}{2}\big]\big)\big(l+\big[\frac{k}{2}\big]+1 \big)-k},
\quad \mbox{   for   } \, 1 \le l \le \bigg[\frac{k+1}{2}\bigg] \, , \nonumber \\[0.1cm]  
\tilde{\f}_{2l+\big(\big[\frac{k+1}{2} \big]-\big[\frac{k}{2}\big]\big)} 
&=&  \f_{\big(l+\big[\frac{k+1}{2}\big]\big)^2-k},
\quad \mbox{   for   } \, 1 \le l \le  \bigg[\frac{k}{2}\bigg]  \, . \nonumber
\end{eqnarray}
Then relations \eqref{y0x0f1}-\eqref{y0x0f3} are equivalent to
\begin{eqnarray}
\label{y0x0tf1}
  &&Y_0X_0\, \tilde{\f}_{i} = -(k-i)\, b\, \tilde{\f}_{i}-\big( a +2(k-i)\, b \big) \, \sum_{s=i+1}^k \tilde{\f}_s\, , \,  1 \le i \le k \, .
 \end{eqnarray}
Thus $Y_0X_0$ acts as a linear transformation on the subspace of $\c^{k^2}$ spanned by 
$\{\tilde{\f}_i\}$, and we denote by $A$  the corresponding matrix of this transformation with respect to
to the basis $\{\tilde{\f}_i\}$ (note they are linearly independent, since $\f_i$'s are).
It is clear from \eqref{y0x0tf1} that  $A$ is a lower triangular matrix.  
Now it is elementary linear algebra to show that $A$ can be diagonalized by choosing the basis
 $\{\g_i\}_{i=1}^k$ defined by
\begin{equation}
\label{chanbas1}
 \g_i\, :=\, \sum_{j=i}^k c_{ij} \tilde{\f}_j \, ,
 \end{equation}
where  $c_{ij}=0$ for $i >j$, $c_{ii}=1$ and for $i<j$,
%
%
%
\begin{equation}
\label{tij}
 c_{ij}:= \frac{ ( z + 2(k-i))_{j-i}  }{(j-i)!},
 \end{equation}
%
%
here $z:=a/b$ and $(z)_n=\prod_{r=0}^{n-1}(z-r)$ is the \textit{ lower Pochhammer symbol}. 
Explicitly, we have $D=CAC^{-1}$, where $C=(c_{ij})$  and
$D=\mathrm{Diag}[-(k-1)b,-(k-2)b,\cdots, 0\,]$.

Next, we express $v_1$ and $w_1$ in the basis $\{\g_i\}$.
Recall that
\begin{equation}
\label{v1w1nbas1}
 v_1= \sum_{i=1}^k \tilde{\f}_i \, , \quad w_1 = \sum_{i=1}^{k} \big( a+ 2(k-i)\, b\big) \tilde{\f}_i^{\mathrm{t}} \, .
 \end{equation}
It is easy to verify that $C^{-1}=(d_{ij})$, where $d_{ij}=0$ for $i>j$, $d_{ii}=1$ and for $i<j$,
\begin{equation}
\label{sij}
d_{ij}= (-1)^{j-i} \, \frac{z+2(k-i)\,}{(j-i)!}\, ( z+ 2k-i-j-1)_{j-i-1}  \, .
\end{equation}
Since $\tilde{\f}_i= \sum_{j=i}^k d_{ij} \g_j$ we have $  v_1=\sum_{i=1}^k \sum_{j=i}^k d_{ij} \g_j$.
Now we can compute $wY^lX^lv$. To make notations simpler, we compute $wY^{2l}X^{2l}v$,
since computations for odd ones will be essentially  the same.
By Proposition~\ref{coeffkappa} we get
\begin{eqnarray}
wY^{2l}X^{2l}v&=& w_1 \prod_{q=0}^{l-1}\big(Y_0X_0+q(\tau_0+\tau_1)\, \mathrm{Id}\big)\cdot  
\prod_{q=1}^{l} \big(Y_0X_0+(q \tau_0+ (q-1)\tau_1) \, \mathrm{Id}\big)\, v_1 \nonumber \\[0.1cm] 
&=&\sum_{i=1}^k \sum_{j=i}^k d_{ij} \prod_{q=0}^{l-1} 
\big( -(k-j)\, b + q(\tau_0+\tau_1) \big) \prod_{q=1}^{l} \big(-(k-j)\, b +q \tau_0+ (q-1)\tau_1)\, w_1 \cdot \g_j \nonumber \\[0.1cm]  
&=& \sum_{i=1}^k  \sum_{j=i}^{k} d_{ij} \prod_{q=0}^{l-1} 
( q-k+j) \, b \,  \prod_{q=1}^{l} \big(-a+(q-k+j)\,b \big) \, \sum_{m=j}^k c_{jm} \, w_1\cdot \tilde{\f}_m  \nonumber\\[0.1cm]  
&=& \sum_{m=1}^k \sum_{i=1}^m \sum_{j=i}^{m} (a+2(k-m)\,b) \, d_{ij} c_{jm}  \prod_{q=0}^{l-1} 
( q-k+j) \, b \,  \prod_{q=1}^{l} \big(-a+(q-k+j)\,b \big) \nonumber\\[0.1cm] 
&=&  b^{2l+1} \, \sum_{m=1}^k \sum_{i=1}^m\sum_{j=i}^{m} \, \, d_{ij} c_{jm}\, (z+2(k-m))\, (k-j)_{l} \, (z+k-j-1)_{l} ,
\label{wyx2lv}
\end{eqnarray}
where $w_1 \cdot \f_m = a+2(k-m)\,b$ is the usual dot product of vectors
and we recall that $z=a/b$.

Next we need to find $d_{ij} \, c_{jm}$. For $i=j=m$ we have  $d_{ij} \, c_{jm}=1$.
For $i<m$, using \eqref{tij} and \eqref{sij},  we get
\begin{eqnarray*}
 d_{ij} c_{jm} &=&(-1)^{j-i} \, \frac{z+2(k-i)\,}{(j-i)!}\, ( z+ 2k-i-j-1)_{j-i-1}\,
 \frac{ \big( z + 2(k-j)\big)_{m-j}  }{(m-j)! }\nonumber\\[0.1cm]
& = &  (-1)^{j-i}\, \frac{z+2(k-i)\,}{(j-i)!\, (m-j)!} \,
( z+ 2k-i-j-1)_{m-i-1}.
\end{eqnarray*}
Thus from \eqref{wyx2lv} one obtains $ wY^{2l}X^{2l}v= b^{2l+1} F_{k,2l}(z)$, where 
\begin{eqnarray}
 F_{k,2l}(z) &:=& \, \sum_{m=2}^k \, \big( z +2(k-m)\big)\sum_{i=1}^{m-1}\,  \big(z+2(k-i)\big) \sum_{j=i}^{m} 
 \,  (-1)^{j-i}\, \frac{( z+ 2k-i-j-1)_{m-i-1}}{(j-i)!\, (m-j)!} \,
 \nonumber \\[0.1cm]
 & &\times   (k-j)_{l} \, (z+k-j-1)_{l}  +  \sum_{m=1}^k \, \big( z +2(k-m)\big) (k-m)_l(z+k-m-1)_l.\nonumber 
 \end{eqnarray}
We claim the following lemma, whose proof will be given in the Appendix:

\begin{lemma} 
\label{lemmaFk2l}
For all $k \ge 2$ and $0 \le l \le k-1$,
$$ F_{k,2l}(z)= \begin{pmatrix} 
k+l \\
2l+1
\end{pmatrix}  (z+k+l-1)_{2l+1} ,$$
and, similarly, for all $k \ge 2$ and $0 \le l \le k-2$
$$  F_{k,2l+1}(z)=- \begin{pmatrix} 
k+l \\
2l+2
\end{pmatrix} (z+k+l)_{2l+2}\,.$$
For all other values of $k$ and $l$ the corresponding $F_{k,l}(z)=0$.
\end{lemma}

It follows from Theorem~\ref{thgeqivbij} and Lemma~\ref{lemmaFk2l} that:

\begin{corollary}
\label{corkappa1}
For $(X_0, X_1; Y_0, Y_1; v_1, w_1) \in \mathfrak{M}_{\z_2}^\tau(k^2-k,k^2; \mathrm{1})$, we have
\begin{equation}
\label{kappakk1}
\kappa=1+\sum_{l=1}^{2k-1} A_l \, y^{-l}x^{-l} \, , 
\end{equation}
where
\begin{eqnarray*}
 A_{2l-1} & = &- \begin{pmatrix} 
k+l-1 \\
2l-1
\end{pmatrix}   \prod_{r=1}^{2l-1} \big( (k+l-r-1)\tau_0 +(k+l-r)\tau_1 \big),  \nonumber \\[0.15cm]
 A_{2l} &=& \begin{pmatrix} 
k+l-1 \\
2l
\end{pmatrix} \prod_{r=0}^{2l-1} \big( (k+l-r-1)\tau_0 +(k+l-r)\tau_1 \big) .
\end{eqnarray*}
%
\end{corollary}

\subsection{$\kappa$ for $\c^*$-fixed points of $\mathfrak{M}^{\tau}_{\z_2}(n-k, n; \, \mathrm{1})$
and $\mathfrak{M}^{\tau}_{\z_2}(n,n-k;\, \mathrm{0})$}
The computation of $\kappa$ for the point $(X_0,Y_0,Y_0,Y_1; v_1,w_1)$ defined in Proposition~\ref{nilpointhd}
is essentially the same as in the previous section with some minor modifications. 
First, recall that $C$ is the change of basis matrix from $\{\tilde \f_i\}$ to $\{\g_i\}$ 
for the point of $\mathfrak{M}^{\tau}_{\z_2}(k^2-k, k^2; \mathrm{1})$(see \eqref{chanbas1}-\eqref{tij}).
Let $\tilde{C}=(\tilde c_{ij})$ be that one for the above point of $\mathfrak{M}^{\tau}_{\z_2}(n-k, n; \mathrm{1})$.
 Then $\tilde{c}_{ij}=c_{ij}$ if $i \neq 1$ and
$$ \tilde c_{1j}= c_{1j}+ \frac{(n-k^2) (z+ 2k-3)_{j-2}}{(j-1)!}\, , $$
where $c_{ij}$ are defined in \eqref{tij}.
The vectors $v_1$ and $w_1$  in the basis $\{\tilde{\f}_i\}$ will be presented as (compare to \eqref{v1w1nbas1})
\begin{equation}
\label{v1w1nbas2}
 v_1= \sum_{i=1}^k \tilde{\f}_i \, , \quad w_1 = (n-k^2)\,b\, \tilde{\f}_1^{\mathrm{t}}+
 \sum_{i=1}^{k} \big( a+ 2(k-i)\, b\big) \tilde{\f}_i^{\mathrm{t}} \, .
 \end{equation}
Next,  it is easy to verify that $\tilde{C}^{-1}=(\tilde{d}_{ij})$ is given by $\tilde{d}_{ij}=d_{ij}$ if 
$i \neq 1$ and 
$$ \tilde{d}_{1j}:=d_{1j}+(-1)^{j-i} \frac{n-k^2}{(j-1)!} \, (z+2k-j-2)_{j-2} \, ,$$
where $d_{ij}$ is defined in \eqref{sij}. If we let $\tilde{F}_{k,l}(z):=b^{l+1} wY^lX^lv$, then
arguing as in \eqref{wyx2lv} we get 
\begin{eqnarray}
 \tilde{F}_{k,2l} (z) &=& F_{k,2l}(z) + (n-k^2) \bigg( (k-1)_l\, (z+k-2)_l + \sum_{m=2}^k \, \big( z +2(k-m)\big) \nonumber \\[0.2cm]
 &&\times   \sum_{j=1}^{m} 
 \,  (-1)^{j-1}\, \frac{( z+ 2k-j-2)_{m-2}}{(j-1)!\, (m-j)!}  (k-j)_{l} \, (z+k-j-1)_{l} \bigg) \nonumber \\[0.2cm]
 &=&  F_{k,2l}(z) + (n-k^2) \hat{H}_{k-1,2l}(z),   \nonumber 
 \end{eqnarray}
where $ F_{k,2l} (z)$ and $ \hat{H}_{k-1,2l}(z)$ are  computed in Lemma~\ref{lemmaFk2l}.
Similarly, we can show
$$    \tilde{F}_{k,2l+1}(z)=  F_{k,2l+1}(z)+\hat{H}_{k-1,2l+1}(z)\, .$$  
Thus we obtain:

\begin{proposition}
For $(X_0, X_1; Y_0, Y_1; v_1, w_1) \in \mathfrak{M}_{\z_2}^\tau(n-k,n; \, \mathrm{1})$ we have
\begin{eqnarray}
 \kappa & = &1- \sum_{l=0}^{k-1} \,  \begin{pmatrix} 
k+l \\
2l+1
\end{pmatrix}   \prod_{r=2}^{2l+1} \big( (k+l-r)\tau_0 +(k+l-r+1)\tau_1 \big) \nonumber \\[0.2cm]
 &&\times  \bigg( (k+ l-1)\tau_0 +(k+l)\tau_1+\frac{(n-k^2)(2l+1)}{k+l}(\tau_0+\tau_1)\bigg) 
 \, y^{-2l-1}x^{-2l-1} \nonumber \\[0.2cm]
 &+& \sum_{l=0}^{k-2} \, \begin{pmatrix} 
k+l \\
2l+2
\end{pmatrix}  \prod_{r=1}^{2l+1} 
 \big( (k+l-r)\tau_0 +(k+l-r+1)\tau_1 \big)\nonumber \\[0.2cm]
 &&\times  \bigg( (k+ l)\tau_0+(k+l+1)\tau_1 +\frac{(n-k^2)(2l+2)}{k+l}(\tau_0+\tau_1)\bigg) 
 \,y^{-2l-2} x^{-2l-2}. 
  \label{kappank}
\end{eqnarray}
\end{proposition}
We now recall that (see Lemma~\ref{simpimplem}) $\mathfrak{M}^{\tau}_{\z_2}(n,n-k;\, \mathrm{0})$ can be identified
with $\mathfrak{M}^{\tau'}_{\z_2}(n-k,n;\, \mathrm{1})$, where $\tau'=(\tau_1,\tau_0)$. Since this identification
is $\Aut(O_\tau)$-equivariant, fixed points are mapped to fixed points under this identification.
Hence, we have:

\begin{corollary}
\label{n,n-k,0}
For the fixed point of $\mathfrak{M}^{\tau}_{\z_2}(n,n-k;\, \mathrm{0})$ corresponding to the fixed
point of  $\mathfrak{M}^{\tau'}_{\z_2}(n-k,n;\,  \mathrm{1})$ under the bijection defined in Proposition~\ref{nilpointhd},
$\kappa$ is given by  \eqref{kappank} where $\tau_0$ and $\tau_1$ are reversed.
\end{corollary}

\subsection{$\kappa$ for $\c^*$-fixed points of $\mathfrak{M}^{\tau}_{\z_2}(n-k, n; \, \mathrm{0})$
and $\mathfrak{M}^{\tau}_{\z_2}(n,n-k;\,  \mathrm{1})$}

\begin{proposition}
For $(X_0, X_1; Y_0, Y_1; v_1, w_1) \in \mathfrak{M}_{\z_2}^\tau(n-k,n; \, \mathrm{0})$ we have
\begin{equation}
  \label{kappank1}
\kappa \, = \, 1+ \sum_{l=1}^{2k} B_l \, y^{-l}x^{-l} \, , 
\end{equation}
where
\begin{eqnarray}
B_{2l-1}&=&-  \begin{pmatrix} 
k+l-1 \\
2l-1
\end{pmatrix} 
 \prod_{r=2}^{2l-1} \big( (k+l-r)
 \tau_0 +(k+l-r+1)\tau_1 \big) \nonumber \\[0.2cm]
 &&\times  \bigg( (k+l-1)\tau_0 +(k+l)\tau_1+\frac{(n-k^2-k)(2l-1)}{k+l-1}(\tau_0+\tau_1)\bigg) ,
 \nonumber \\[0.2cm]
 B_{2l}&=& \begin{pmatrix} 
k+l \\
2l
\end{pmatrix}   \prod_{r=1}^{2l-1} 
 \big( (k+l-r-1)\tau_0 +(k+l-r)\tau_1 \big)\nonumber \\[0.2cm]
 &&\times  \bigg( (k+ l-1)\tau_0+(k+l)\tau_1 +\frac{(n-k^2-k)(2l)}{k+l}
 (\tau_0+\tau_1)\bigg).
\end{eqnarray}
\end{proposition}

Similarly to Corollary~\ref{n,n-k,0}, we can conclude:

\begin{corollary}
\label{n,n-k,1}
For the fixed point of $\mathfrak{M}^{\tau}_{\z_2}(n,n-k;\, \mathrm{1})$ corresponding to the fixed
point of  $\mathfrak{M}^{\tau'}_{\z_2}(n-k,n;\, \mathrm{0})$ under the bijection
 defined in Proposition~\ref{nilpointhd},
$\kappa$ is given by  \eqref{kappank1} where $\tau_0$ and $\tau_1$ are reversed.
\end{corollary}

\section{$\c^*$-fixed ideals and their endomorphism rings}
\la{sec8}

Recall that for $\Gamma \cong \z_2$, by transitivity of the $G_\Gamma$-action,
we have the following decomposition of $\mathcal{R}^{\tau}_{\Gamma}$ into $G_\Gamma$-orbits
$$  \mathcal{R}^\tau_{\Gamma} = \bigsqcup_{ (m,n, \epsilon) \in L_{\epsilon} \times \z_2} 
 \mathcal{R}^\tau_{\Gamma}(m,n; \epsilon) \ .$$
 Let $P^{(\epsilon)}_{(m,n)} \in  \mathcal{R}^\tau_{\Gamma}(m,n; \epsilon)$
 be the images, under the bijective map $\Omega$, of 
 $\c^*$-fixed points, defined in Section~\ref{secfixpoints}.
 In this section we give an explicit description of $P^{(\epsilon)}_{(m,n)} $
 and its endomorphism ring.

\subsection{$\c^*$-fixed ideals} Recall that, by Proposition~\ref{prop1}, there is an algebra
isomorphism $\phi: O_\tau(\z_2) \to A(v)$, where $v(h)=(\tau_0+\tau_1))^2 (h+1)(h+\tau_0/(\tau_0+\tau_1))$.
Using $\phi$ we realize ideals of $O_\tau$ in $A(v)$.
\begin{lemma}
\la{iden}
Let $f_n(h) := \phi(ey^{n} x^{n})$, where $\phi$ is defined in Proposition~\ref{prop1}.  Then 
$$
f_n (h) = \begin{cases} 
      v(h-1) \cdots v(h-k), & n=2k,  \\
            (\tau_0+\tau_1) (h-k+1) v(h-1) \cdots v(h-k+1), & n=2k-1.
   \end{cases}
   $$
\end{lemma}

\begin{proof}
We prove by induction. For $n=1$ it follows from the definition of $\phi$.
 Suppose that the identity holds for all degrees less than $n$. Then
$$
ey^{n} x^n = e y (y^{n-1} x) x^{n-1}  .
$$
Repeatedly using  $ yx = xy - \tau$, we obtain
$$
ey^{n} x^n = e y (xy^{n-1} - \sum^{n-2}_{i=0} y^i \tau y^{n-2-i}) x^{n-1} = (eyx)  
(e y^{n-1} x^{n-1}) - \sum^{n-2}_{i=0} e y^{i+1} \tau y^{n-2-i} x^{n-1}.
$$
Suppose $n=2k$, then $\sum^{2k-2}_{i=0} e y^{i+1} \tau y^{2k-2-i} = 
(k (\tau_0+\tau_1) -\tau_0 ) e y^{2k-1}$. Hence
$$
(eyx)  (e y^{2k-1} x^{2k-1}) - (k (\tau_0+\tau_0) -\tau_0 ) e y^{2k-1}x^{2k-1} = 
(eyx + \tau_0 - k (\tau_0+\tau_1) )  e y^{2k-1}x^{2k-1} \, .
$$
Since $\phi(eyx) = (\tau_0 + \tau_1) h$, the image of the later expression is  
$$
(\tau_0+\tau_1) \bigg( h + \frac{\tau_0}{\tau_0+\tau_1} - k \bigg) \, \phi( e y^{2k-1}x^{2k-1}).
 $$
By induction assumption on $ e y^{2k-1}x^{2k-1}$, we obtain
$$
(\tau_0+\tau_1)^2 \bigg( h + \frac{\tau_0}{\tau_0+\tau_1} - k \bigg) (h-k+1) v(h-1) \cdots v(h-k+1)  = v(h-k) v(h-1)
 \cdots v(h-k+1) .
$$
Analogously one can prove the identity for the case $n=2k-1$.
\end{proof}

Let 
\begin{equation}
l(h):=\frac{v(h)}{h+1}=h+ \frac{\tau_0}{\tau_0+\tau_1},
\end{equation}
and define 
\begin{eqnarray*}
s_{n,2k-\epsilon} (h) &:=& \prod_{i=1}^{2k-1-\epsilon}\, l(h-i) \cdot l \big(h-n+(k-\epsilon)(k-1)\big) \, , \\[0.1cm]
s_{n,2k-\epsilon}' (h)&:=&\prod_{i=1}^{2k-1-\epsilon}\, (h-i) \cdot l \big(h-n+(k-\epsilon)(k-1)\big)\, .
\end{eqnarray*}
\begin{proposition}
\label{genideals}
By identification $\phi: O_\tau(\z_2) \to A(v)$, we have:
\begin{enumerate}
\item[(a)]  For $(n-k, n, \epsilon) \in L_{\epsilon} \times \z_2$,
\begin{eqnarray*}
P^{(\epsilon)}_{(n-k, n)} &  \cong & a^{n-(k-\epsilon)(k-1)}\, A(v) + s_{n,2k-\epsilon}(h) \, A(v) \,.
\end{eqnarray*}
 \item[(b)] 
 For $(n,n-k, \epsilon) \in L_{\epsilon}\times \z_2$,
 $$
 P^{(\mathrm{\epsilon})}_{(n,n-k)} \cong a^{n-(k+\epsilon-1)(k-1)+1} \,A(v) + s_{n, 2k+\epsilon-1}'(h)\, A(v) \, .
 $$
 \end{enumerate}
\end{proposition}
\begin{proof}
We only prove part (a),  since the proof of (b) is identical. 
Without loss of generality, we may assume $n=k^2+k$
for $\epsilon=0$, and $n=k^2$ for $\epsilon=1$.
For simplicity, we also assume that $k$ is even. Then
by Theorem~\ref{thgeqivbij}, we have
\begin{eqnarray*}
&& P^{(0)}_ {(k^2,k^2+k)}\, = \, e y^{2 k^2+k } \, O_\tau + e \kappa_0 x^{2 k^2+k }\, O_\tau \, ,  \\[0.1cm]
&& P^{(1)}_{(k^2-k,k^2)}\, =\, e_1 y^{2 k^2-k} \, O_\tau + e_1 \kappa_1 x^{2 k^2-k}\, O_\tau \, ,
\end{eqnarray*}
since the corresponding matrices $X$ and $Y$ are nilpotent. Here $\kappa_0$ is 
given by \eqref{kappank1} for $n=k^2+k$ and $\kappa_1$ is given by \eqref{kappakk1}.
If we multiply $P^{(0)}_{(k^2,k^2+k)}$ by $y^{2k}$ and $P^{(1)}_{(k^2-k,k^2)}$ by $y^{2k-1}$
then $P^{(0)}_{(k^2,k^2+k)} \cong \langle ey^{k(2k+3)}, ey^{2k}\kappa_0 x^{2k^2+k}\rangle$
and $P^{(1)}_{(k^2-k,k^2)} \cong  \langle ey^{(k+1)(2k-1)}, ey^{2k-1} \kappa_1 x^{2k^2-k}\rangle$,
 where $\langle-,-\rangle$ means generated as right $O_\tau$-modules.

First we show 
$$ \phi(ey^{2k} \kappa_0 x^{2k})=s_{k^2+k,2k}(h)\, 
\mbox{ and } \phi( ey^{2k-1} \kappa_1 x^{2k-1})=s_{k^2, 2k-1}(h) \, .$$
We prove by induction. Both identities obviously hold for $k=1$. 
Suppose that they hold for all degrees less than $2k$. 
Consider the finite difference operator $\Delta \, : \, \c[h] \rightarrow \c[h]$ given by $\Delta (f) = f(h+1) - f(h)$. 
Then it is easy to see that
\begin{equation}
\label{deltasl}
\Delta (s_l (h)) = l (\tau_0 + \tau_1) s_{l-1} (h)\, .
\end{equation}
By \eqref{kappank1} for $n=k^2+k$, we have
$$
e y^{2k} \kappa_0 x^{2k} = e y^{2k} x^{2k} + B_1 e y^{2k-1} x^{2k-1} +  B_2  e y^{2k-2} x^{2k-2} + \, \cdots \,  +B_{2k} .
$$
Applying $\Delta$ to the image of $\phi$, we get
$$
\Delta (\phi(e y^{2k} \kappa_0 x^{2k}))  = \Delta(f_{2k} (h)) + B_1 \Delta(f_{2k-1} (h)) + B_2  \Delta(f_{2k-2} (h)) +
 \, \cdots \, + B_{2k-1}  \Delta(f_{1} (h)).
$$
Then, using Lemma~\ref{iden},  we obtain
\begin{eqnarray*}
\Delta(f_{2i} (h))& = &(\tau_0 + \tau_1)^2  v(h-1) \, \cdots \, v(h-i+1) 
\bigg( 2i h - i(i-1)+ \frac{i \tau_0}{\tau_0+ \tau_1} \bigg)\\[0.1cm]
&=&2i (\tau_0 + \tau_1) f_{2i-1} +  (\tau_0 + \tau_1)^2  
\bigg( i(i-1)+ \frac{i \tau_0}{\tau_0+ \tau_1} \bigg) f_{2i-2},
\end{eqnarray*}
and similarly
$$ 
\Delta(f_{2i-1} (h)) = (2i-1) (\tau_0 + \tau_1) f_{2i-2} +  (\tau_0 + \tau_1)^2  
\bigg( i(i-1)- \frac{(i-1) \tau_0}{\tau_0+ \tau_1} \bigg) f_{2i-3}\, .
$$
It is easy to see that  $f_{2i-2}$ appears in $\Delta (e y^{2k} \kappa x^{2k})$ only in $\Delta(f_{2i} (h))$ and $\Delta(f_{2i-1} (h))$. 
Therefore, since $\{f_i\}$ are linearly independent,
the coefficient of $f_{2i-2}$ can be computed as
$$
B_{2k-2i} (\tau_0 + \tau_1)^2 \bigg( i(i-1)+ \frac{i \tau_0}{\tau_0+ \tau_1}\bigg) + B_{2k-2i+1} (2i-1)  (\tau_0 + \tau_1).
$$
Simplifying the later expression, we get  that it is equal to $2k  (\tau_0 + \tau_1) A_{2k-2i+1}$,
where  $A_i$ is defined as in Corollary~\ref{corkappa1}.
Similarly, we can show that the coefficient of $f_{2i-1}$ is $2k  (\tau_0 + \tau_1) A_{2k-2i}$.
Therefore we have
 $$
\Delta ( \phi(e y^{2k} \kappa_0 x^{2k}))  =  2k (\tau_0 + \tau_1) \Big (f_{2k-1}+ A_1 f_{2k-2}+ \, \cdots \, A_{2k-2} f_1 + A_{2k-1}),
$$
and hence $\Delta ( \phi(e y^{2k} \kappa_0 x^{2k}))  = \phi(e y^{2k-1} \kappa_1 x^{2k-1})$. 
By \eqref{deltasl} and by the induction assumption $e y^{2k-1} \kappa_1 x^{2k-1} = 2k (\tau_0 + \tau_1) P_{2k-1}$, we have 
$$
\Delta (\phi(e y^{2k} \kappa_0 x^{2k})) \, = \, \Delta (s_{2k} (h)).
$$
Now comparing constant terms, we get  $s_{2k}(0)=B_{2k}$ and  hence 
$\phi(e y^{2k} \kappa_0 x^{2k})=s_{k^2+k, 2k}(h)$. Using similar 
arguments, we also obtain $ \phi( ey^{2k-1} \kappa_1 x^{2k-1})=s_{k^2, 2k-1}(h)$.

Now, applying $\phi$ to the generators of $P^{(0)}_{(k^2,k^2+k)}$ 
and $ P^{(\mathrm{1})}_{(k^2-k,k^2)}$, we obtain 
$$P^{(0)}_{(k^2,k^2+k)} \cong  a^{ \frac{k(2k+3)}{2}}\,A(v) + s_{k^2+k, 2k}(h) b^{ \frac{k(2k-1)}{2} }\,A(v)$$
and 
$$ P^{(\mathrm{1})}_{(k^2-k,k^2)} \cong a^{\frac{(k+1)(2k-1)-1}{2}} 
\,A(v) + s_{k^2, 2k-1} (h) b^{ \frac{k(2k-3)}{2}}\,A(v)\, .$$
Since 
$$a^{ \frac{k(2k+3)}{2}}\, b^{\frac{k(2k-1)}{2}}=a^{2k}\cdot 
\prod_{i=1}^{\frac{k(2k-1)}{2}} v(h-i)
$$
and 
$$ s_{k^2+k, 2k}(h) \, b^{ \frac{k(2k-1)}{2}} \, a^{ \frac{k(2k+3)}{2}}
=a^{2k}\, s_{k^2+k,2k}(h+2k) \cdot \prod_{i=0}^{\frac{k(2k-1)}{2}-1}
v(h+i+2k)\, ,
$$
the greatest common divisor (GCD for short) of the left hand sides is $a^{2k}$ and therefore $a^{2k} \in P^{(0)}_{(k^2,k^2+k)}$. 
Then $s_{k^2+k, 2k}(h)$ divides $a^{2k}b^{2k}$ and hence $s_{k^2+k, 2k}(h) \in P^{(0)}_{(k^2,k^2+k)}$.
Thus 
$$P^{(0)}_{(k^2,k^2+k)} \cong  a^{2k}\, A(v) + s_{k^2+k, 2k}(h)\, A(v)\, .$$
Similarly, we can show
\begin{equation*}
 P^{(\mathrm{1})}_{(k^2-k,k^2)} \cong a^{2k-1}\, A(v)+s_{k^2,2k-1}(h)\,A(v)\,.
 \qedhere\end{equation*}
\end{proof}

 \subsection{Endomorphism rings of $P^{(\epsilon)}_{(m,n)}$ }
 First we introduce a grading on the algebra $A(v)$ as follows.
 For $t \in \z$, put $D(t)=\{ c \in A(v)\, | \, h\cdot c- c \cdot h =t c\}$.
 Then, by \eqref{relAv},  $A(v)= \oplus_{-\infty}^{\infty} D(t)$, where
$$
D(t)=\left\{
\begin{array}{ll} a^{t}\, \c[h], & t \ge 0,\\
 b^{-t}\, \c[h],& t <0\, .
\end{array}
\right.
$$
$A(v)$ is a graded ring, since $D(t_1)D(t_2) \subseteq D(t_1+t_2)$. 

Set 
\begin{equation}
\la{defendrings}
D^{(\epsilon)}_{m,n} \, := \, \End_{A(v)}\big(P^{(\epsilon)}_{m,n} \big).
\end{equation}
Then we have:
\begin{proposition}  
\la{prop16}
\begin{enumerate}
\item[(a)]
For $(n-k,n) \in L_{\epsilon}$, $D^{(\epsilon)}_{n-k,n}=\oplus_{-\infty}^{\infty} E(t)$, where $E(t)$ is
\begin{equation}
 \begin{cases} 
     a^{t}\,\c[h],  &  \mbox{ for } t \geq n-k(k-\epsilon+1)+1  ,\\[0.15cm]
      a^{t}\cdot l \big(h-n+(k-\epsilon)(k-1)+t \big)
  \, \c[h], &  \mbox{ for }   0 < t  \le   n-k(k-\epsilon+1) ,\\[0.15cm]
       \c[h],  &  \mbox{ for } t=0, \\[0.15cm]
       b^{-t}\cdot \frac{s_{n,2k} (h+t)} {s_{n,2k}(h)}\cdot 
       l \big(h-n+(k-\epsilon)(k-1) \big) \, \c[h],  & \mbox{ for } 
       -n+k(k-\epsilon+1) \le t <0 ,\\[0.2cm]
       b^{-t}\cdot \frac{s_{n,2k} (h+t)} {s_{n,2k}(h)} \, \c[h]  , & \mbox{ for }  t < -n+k(k-\epsilon+1).
   \end{cases}   \nonumber
   \end{equation}
\item[(b)] For $(n,n-k)\in L_{\epsilon}$, $D_{n,n-k}^{(\epsilon)}=\oplus_{-\infty}^{\infty}F(t)$,
where $F(t)$ is
\begin{equation}
 \begin{cases} 
     a^{t}\,\c[h],  &  \mbox{ for } t \geq n-k(k+\epsilon)+2 , \\[0.15cm]
      a^{t}\cdot l \big(h-n+(k+\epsilon-1)(k-1)+t \big)
  \, \c[h], &  \mbox{ for }   0 < t  \le   n-k(k+\epsilon)+1, \\[0.15cm]
       \c[h],  &  \mbox{ for } t=0, \\[0.15cm]
       b^{-t}\cdot \frac{s_{n,2k} (h+t)} {s_{n,2k}(h)}\cdot 
       l \big(h-n+(k-\epsilon)(k-1) \big) \, \c[h],  & \mbox{ for } 
       -n+k(k+\epsilon)-1 \le t <0, \\[0.2cm]
       b^{-t}\cdot \frac{s_{n,2k} (h+t)} {s_{n,2k}(h)} \, \c[h]  , & \mbox{ for }  t <-n+k(k+\epsilon)-1 .
   \end{cases}   \nonumber
   \end{equation}
\end{enumerate}
\end{proposition}  
 \begin{proof}
 We only prove part (a),  since the proof of (b) is similar. 
 By Proposition~\ref{genideals}, $P^{(0)}_{(n-k,n)}$
is generated by homogeneous 
elements and hence can be presented as $P^{(0)}_{(n-k,n)}=\oplus_{-\infty}^{\infty}(P^{(0)}_{(n-k,n)} \cap D(t))$.
Then $P^{(0)}_{(n-k,n)} \, \cap D(t)$ is
\begin{equation*}
\begin{cases}  a^{t}\,\c[h],  &    t \ge n-k^2+k, \\[0.1cm]
 a^{t}\cdot l(h-n+k^2-k+t)\, \c[h],  &    2k-1 \le t < n-k^2+k, \\[0.1cm]
 a^{t} \cdot \prod_{i=1}^{2k-t-1} \, l(h-i)\,l(h-n+k^2-k+t) \, \c[h],& 
1\le t \le 2k-2, \\[0.1cm] 
 b^{-t} \cdot s_{n,2k}(h+t) \, \c[h],&    t \le 0 .
 \end{cases} 
\end{equation*}
 Now $c \in D^{(0)}_{n-k,n}$ if and only if $c\cdot a^{n-k^2+k}$
  and $c \cdot s_{n,2k}(h) \in P^{(0)}_{n-k,n}$.
  Thus $D^{(0)}_{n-k,n}=  P^{(0)}_{n-k,n}a^{-n+k^2-k}\, \cap\,  P^{(0)}_{n-k,n}s_{n,2k}^{-1}(h)$,
  and since $a^{n-k^2+k}$ and $ s_{n,2k}(h)$ are homogeneous elements,
  $D^{(\epsilon)}_{m,n}=\oplus_{-\infty}^{\infty} E(t)$, where
  $$ E(t)\, =\, \big(P^{(0)}_{n-k,n}  \, \cap \, D(t+n-k^2+k) \big)\cdot  a^{-n+k^2-k}
 \, \cap \,  \big(P^{(0)}_{n-k,n}  \, \cap \, D(t) \big)\cdot s_{n,2k}^{-1}(h)\,.
  $$
 It follows from \eqref{relAv} that $a^{-1}=b v(h-1)^{-1}$ and 
 $a^{-n+k^2-k}=b^{n-k^2+k} v(h-1)^{-1}\cdots v(h-n+k^2-k)^{-1}$. Then $P^{(0)}_{n-k,n}\cdot a^{-n+k^2-k}$ can be computed as
 \begin{equation*}
 \begin{cases}    a^t \c[h],                                                                                   & t \geq 0 \, , \\[0.15cm]
                           b^{-t} \, \prod_{i=1}^{-t} v(h-i)^{-1} \cdot l(h-n+k^2-k+t)\, \c[h], &  -n+k^2+k-1  \le t <0 \, , \\[0.15cm]
                          \!\! \begin{array}{l} b^{-t} \, \prod_{i=1}^{-t} v(h-i)^{-1} \cdot l(h-n+k^2-k+t)\\  
                                \quad\times\prod_{i=n-k^2+k+1}^{2k-t-1} l(h-i)  \,  \c[h] ,
                                \end{array}    &  -n+k^2-k+1 \le  t  \le -n+k^2+k-2\, , \\[0.15cm]
                           b^{-t}   \, \prod_{i=1}^{n-k^2+k} v(h-i)^{-1} s_{n,2k}(h+t) \, \c[h],  & t \le  -n+k^2-k\, , \\
                   \end{cases}
\end{equation*}  
while $ P^{(0)}_{n-k,n} \cdot s_{n,2k}(h)^{-1}$ can be expressed as
\begin{equation*}
\begin{cases}  a^{t}\cdot s_{n,2k}(h)^{-1} \,\c[h],  &    t \ge n-k^2+k \,, \\[0.1cm]
 a^{t}\cdot s_{n,2k}(h)^{-1} l(h-n+k^2-k+t)\, \c[h],  &    2k-1 \le t < n-k^2+k \, , \\[0.1cm]
 a^{t} \cdot s_{n,2k}(h)^{-1}  \prod_{i=1}^{2k-t-1} \, l(h-i)\,l(h-n+k^2-k+t) \, \c[h],& 
1\le t \le 2k-2 \, ,\\[0.1cm] 
 b^{-t} \cdot s_{n,2k}(h)^{-1} s_{n,2k}(h+t) \, \c[h],&    t \le 0 .
 \end{cases}
\end{equation*}
 Now we can compute $E(t)$:
 \begin{enumerate}
 \item[(i)] 
 Suppose $t\ge n-k^2-k+1$. 
 For $t \ge n-k^2+k$, it is clear that $E(t)=a^{t}\,\c[h]$;
 for $n-k^2-k+1 \le t < n-k^2+k$, we have 
 $$ h-(2k-1) \le  h-n+k^2-k+t \le h-1,$$
 which means $ l(h-n+k^2-k+t)$ divides $s_{n,2k}(h)$, and hence $E(t)=a^{t}\,\c[h]$.
 \item[(ii)] Suppose $0 < t \le n-k^2-k$.
 In this case
 $$ h-n+k^2-k < h-n+k^2-k+t \le h-2k,$$
 which implies $ l(h-n+k^2-k+t)$ is relatively prime to $s_{n,2k}(h)$. Hence,
 $ E(t)=a^{t}\cdot  l(h-n+k^2-k+t)\, \c[h]$.

 \item[(iii)] Suppose $t=0$.
 It can be easily seen that $E(0)=\c[h]$.
 \item[(iv)] Suppose $ -n+k^2+k \le t \le -1$.
 Then
 \begin{eqnarray}
 \la{repiv}
 E(t) &=& b^{-t} \, \prod_{i=1}^{-t} v(h-i)^{-1} \cdot l(h-n+k^2-k+t)\, \c[h] \\[0.1cm] 
      &  & \cap  b^{-t} \cdot s_{n,2k}(h)^{-1} s_{n,2k}(h+t) \, \c[h] \, . \nonumber
 \end{eqnarray}       
Recall that we have
 \begin{equation}
 \la{firstfrac}
  \frac{s_{n,2k}(h+t)}{s_{n,2k}(h)} \cdot l(h-n+k^2-k)
   = \frac{\prod_{i=-t+1}^{2k-t-1} l(h-i)}{\prod_{i=1}^{2k-1}l(h-i)}\, 
\cdot l(h-n+k^2-k+t)
  \end{equation}
  and
  \begin{equation}
  \la{secfrac}
  \prod_{i=1}^{-t}  v(h-i)^{-1} = \prod_{i=0}^{-t+1}\frac{1}{h-i} \cdot \prod_{i=1}^{-t}\, \frac{1}{l(h-i)} \, .
  \end{equation}
  Now we consider two sub-cases: first $t\le -2k+1$ and then $t >-2k+1$.
  In the first case, GCD of denominators 
  of \eqref{firstfrac} and \eqref{secfrac} is the denominator of \eqref{firstfrac}.
 In the second case, we can simplify the multiplier
  of LHS of \eqref{firstfrac} to
 $$  \frac{\prod_{i=2k}^{2k-t-1} l(h-i)}{\prod_{i=1}^{-t} l(h-i)}, $$
 and again the GCD of denominators is the one of  \eqref{firstfrac}.
 Thus, we obtain
  $$ E(t)=b^{-t}\cdot \frac{s_{n,2k} (h+t)} {s_{n,2k}(h)}\cdot l(h-n+k^2-k) \, \c[h] \, .$$
 \item[(v)] Suppose $t < -n+k^2+k$.
 Divide this case into  three sub-cases : $t=-n+k^2+k-1, -n+k^2-k+1 \le t \le -n+k^2+k-2$
 and $ t \le-n+k^2-k$.\\
 (a) For $t=-n+k^2+k-1$, $E(t)$ has the same presentation as \eqref{repiv} and
 \begin{eqnarray}
 \la{thirdfrac}
 \frac{s_{n,2k}(h+t)}{s_{n,2k}(h)}&=&\frac{\prod_{i=-t+1}^{2k-t-1} l(h-i)}{\prod_{i=1}^{2k-1}l(h-i)}\cdot
 \frac{l(h-n+k^2-k+t)}{l(h-n+k^2-k)} \\[0.15cm]
 &=&   \frac{\prod_{i=-t+1}^{2k-t-2}  l(h-i)}{\prod_{i=1}^{2k-1}l(h-i)}\cdot l(h-n+k^2-k+t)\, ,
  \nonumber
 \end{eqnarray}
 where the last equality follows from $2k-t-1=n-k^2+k$. So GCD
 of denominators of \eqref{secfrac} and \eqref{thirdfrac}
 is the one of \eqref{thirdfrac}. Hence, we get
 \begin{equation}
 \la{repv}
 E(t)=b^{-t}\cdot \frac{s_{n,2k} (h+t)} {s_{n,2k}(h)} \, \c[h] \,.
 \end{equation}
\\
(b)
 Next, for $-n+k^2-k+1 \le t \le -n+k^2+k-2$,
 \begin{eqnarray*}
  E(t)&=& b^{-t} \, \prod_{i=1}^{-t} v(h-i)^{-1} \cdot l(h-n+k^2-k+t) \prod_{i=n-k^2+k+1}^{2k-t-1} l(h-i)  \,  \c[h]\\[0.10cm]
&&\cap  b^{-t} \cdot s_{n,2k}(h)^{-1} s_{n,2k}(h+t) \, \c[h] \, , \nonumber
\end{eqnarray*}  
 and \eqref{repv} for this case
  follows using the same argument as the previous case, once we notice
 that $-t+1 \le n-k^2+k$.\\
 (c) Finally, the proof of   \eqref{repv}
 for $ t \le-n+k^2-k$ is exactly the same as the proof of $(\mathrm{iv})$.
   \end{enumerate}
 This completes the proof.
  \end{proof}

  One can easily verify:
  \begin{corollary}
  \la{cor11}
   For $k \ge 1$,
  $$ D^{(0)}_{k^2,k^2+k} \cong A(w_1)\, , \quad  D^{(1)}_{k^2-k,k^2} \cong A(w_2)\, ,
  $$
  where 
  $$ w_1(h):= h\cdot l(h-2k-1) \, , \quad w_2(h):=h\cdot l(h-2k)\, .
   $$
  \end{corollary}
\section{Proof of Theorems \ref{intrth2} and \ref{intrth2.5}.}
\la{sec9}
\begin{proof}[Proof of Theorem \ref{intrth2}]
In Proposition~\ref{prop16}, we gave a description of the endomorphism rings
$D^{(\epsilon)}_{m,n}$ of ideals of $O_\tau \cong A(v)$.
Recall that Theorem \ref{intrth2} claims that
 $D^{(\epsilon)}_{m,n} \ncong A(v)$ for any $(m,n; \epsilon)\in L_{\epsilon}\times \z_2$
 unless $(m,n; \epsilon)=(0,0;0)$.
Suppose there is an isomorphism $f: \,  A(v) \rightarrow  D^{(\epsilon)}_{n-k,n}$.  
By Proposition~\ref{prop16},
$h$ is strictly semisimple both in  $A(v)$ and in $D^{(\epsilon)}_{m,n}$.
It is clear that
 under isomorphism a strictly semisimple element 
is  mapped to a semisimple one.
 Hence, by \cite[Proposition 3.26]{BJ},
  there is an automorphism $\Phi$ of $A(v)$ such that 
  $\Phi ( \, f^{-1}(h)) = \gamma h + \alpha$ for some $\gamma, \, \alpha \in \mathbb{C}$. 
  Therefore $\Phi \circ f$ is
a  graded isomorphism.   In particular, the image of $a\, \c[h]$ must generate all positive 
degree elements  in $D^{(\epsilon)}_{n-k,n}$.
But this is impossible, because $a \, \c[h]$ 
is mapped to $a \, l(h-n+k^2 -2k) \cdot \c[h]$
and hence all elements in the positive degrees in $A(v)$ must be multiple 
of $l(h-n+k^2 -2k)$.  Similarly, we can show $D^{(\epsilon)}_{m,n} \ncong A(v)$
for all other cases. This finishes the proof of Theorem~\ref{intrth2}.
\end{proof}

Before proceeding to the proof of Theorem~\ref{intrth2.5},
let us recall some basic facts 
about Picard groups. Let $A$ be a $\c$-algebra.
Then  the \textit{Picard group} of $A$, denoted by  $\Pic(A)$, is the 
multiplicative group consisting
of all bimodule isomorphism classes $(X)$ of invertible bimodules
$X$ over $A$. Multiplication is defined by the formula $(X)\cdot (X'):=(X\otimes_AX')$
and the inverse of $X$ is given by $X^*:=\Hom_A(X, A)$.
By definition, $\Pic(A)$ is the group of autoequivalences of $\mathrm{Mod}(A)$,
the category of finitely generated right $A$-modules.
In particular, $\Pic(A)$ acts on $\mathrm{P}(A)$, 
 the subcategory of finitely generated projective modules over $A$:
$$ \mathrm{P}(A) \times   \Pic(A)  \to \mathrm{P}(A) \, ,
 \quad  P \times (X) \mapsto P \otimes_A X\, .$$
The orbits consist of projective modules having isomorphic endomorphism rings. 
Indeed, by the dual basis lemma $P_A$ is a projective module if and only if 
$P \otimes_A P^* \cong \End_A(P)$.
Then, for $(X) \in \Pic(A)$, one has 
\begin{eqnarray*}
\End_A(P\otimes_A X) &\cong & (P \otimes_A X) \otimes _A (P \otimes_A X)^* 
\cong P\otimes_A (X\otimes_AX^*) \otimes_A P^*  \\[0.1cm]
& \cong & P \otimes_A A \otimes_A P^* \cong P \otimes_A P^* \cong \End_A(P) .
\end{eqnarray*}
Assuming that $A$ has no units except nonzero scalars in $\c$ then
\begin{equation}
\la{somegaeq}
\omega_A :\Aut(A) \to  \Pic(A)\, , \quad \sigma \mapsto  {}_1(A)_\sigma \, , 
\end{equation}
where ${}_1(A)_\sigma$  means that the right action twisted by $\sigma$,
is a group monomorphism.  The following is  a well-known fact 
(see e.g. \cite[Theorem 37.16]{Re}):

\begin{lemma}
\label{thmfro}
Let $(X), (Y) \in \Pic(A)$. Then $X_A \cong Y_A$ if and only if  
$(Y) \in (X) \cdot \mathrm{Im}(\omega_A)$, that is,
if and only if $Y \cong {}_1X_{\sigma}$ (as bimodules) for some $\sigma \in \Aut(A)$.
\end{lemma}
Using this result Stafford showed (see \cite[Corollary E]{S}) that 
the map $\omega_{A_1}$ is a bijection.
Theorem~\ref{intrth2.5} is about the bijection of $\omega_{O_\tau}$
 and our proof is similar to that of Stafford's.

\begin{proof}[Proof of Theorem~\ref{intrth2.5}]
As we pointed out above, it is enough to show that
 $\omega_{O_\tau}$ is surjective.
By Theorem~\ref{intrth1}, for $(m,n) \in L_{\epsilon}$
the ideals  $P^{(\epsilon)}_{(m,n)}$  are representatives of $\Aut(A_1)$-orbits
on $\mathrm{P}(O_\tau)$. By Theorem~\ref{intrth2},
the corresponding endomorphism rings $D^{(\epsilon)}_{(m,n)} \ncong O_\tau$,
unless $(m,n;\epsilon)=(0,0;0)$ in which case $P^{(\epsilon)}_{(m,n)}\cong O_\tau$.
If $(X) \in \Pic(O_\tau)$, then as a right $O_\tau$-module $X$ is an ideal of $O_\tau$,
and $X \otimes_{O_\tau} X^{\ast} \cong O_\tau$.
But  $X \otimes_{O_\tau} X^{\ast}  \cong \End_{O_\tau}(X)$ and
hence, by Lemma~\ref{thmfro}, 
$X \in \mathrm{Im}(\omega_{O_\tau})$, which shows that $\omega_{O_\tau} $ is surjective.
\end{proof}
We would like to finish this section with the following discussion.
In the above mentioned statement \cite[Corollary E]{S}, it was also
shown that $\Aut(A_1)$  is an invariant distinguishing $A_1$
from non-isomorphic algebras Morita equivalent to it.
More explicitly, if $D$ is Morita equivalent to $A_1$ then
 $\omega_D$ is an isomorphism
 if and only if $A_1 \cong D$. 
%
%
However, this is not the case for $O_\tau$.
\begin{theorem}
For $k\in \z_{\ge 0}$, let $D$ be either $D^{(0)}_{k^2,k^2+k}$ or $D^{(1)}_{k^2-k,k^2}$.
Then $\omega_D$ is an isomorphism.
\end{theorem}

\begin{proof}
Let $D$
be a domain Morita equivalent to $O_\tau$. Then $D \cong \End_{O_\tau}(P)$
for some right ideal  $P \in \mathcal{R}^{\tau}_{\z_2}$. 
Denote by $G_P$ the stabilizer subgroup of $P$ under the action of $\Aut(O_\tau)$
on $ \mathcal{R}^{\tau}_{\z_2}$. Then we have the following group homomorphisms
\begin{equation}
\la{inc1} 
\Aut(D) \hookrightarrow \Aut(O_\tau) \hookleftarrow G_P.
\end{equation}
Here the right inclusion is the natural embedding of $G_P$, and
the left inclusion is given by
\begin{equation}
\la{inc2}
 \Aut(D) \hookrightarrow \Pic(D) \cong \Pic(O_\tau) \cong \Aut(O_\tau),
  \end{equation}
where the first map is $\omega_D$, the second one is induced from
Morita equivalence and the third one is the inverse of $\omega_{O_\tau}$.
Then it is easy to show (see e.g. \cite[Theorem 1]{BEE2}) that 
the images of $\Aut(D)$ and $G_P$ in $ \Aut(O_\tau) $ coincide.
Now for $k\in \z_{\ge 0}$, the sets
$\mathcal{R}^{\tau}_{\z_2}(k^2, k^2+k; \mathrm{0})$ and 
$\mathcal{R}^{\tau}_{\z_2}(k^2-k, k^2; \mathrm{1})$
are singletons consisting of the ideals $P^{(0)}_{k^2,k^2+k}$
and $P^{(1)}_{k^2-k,k^2}$ respectively.
Hence the stabilizers coincide with $\Aut(O_\tau)$ itself.
Therefore both inclusions in \eqref{inc1} are bijections,
which implies that $\omega_D$ is an isomorphism.
\end{proof}

\begin{appendix}
\section{Proof of Lemma~\ref{lemmaFk2l} }

First we prove the following statement:
\begin{proposition}
\label{prop2.2}
For $k\ge 2$ and $0 \le l \le k$,
$$ H_{k,2l}(z):= \sum_{j=0}^k (-1)^j \frac{(k-j)_l (z+k-j-l)_{k+l}}{(k-j)!\, j!} = \frac{(k+l)!}{(2l)!\,(k-l)!} (z)_{2l} \, .$$
\end{proposition}
\begin{proof}
Recall that the finite difference operator $\Delta[f](z):=f(z+1)-f(z)$  acts on Pochhammer symbols
as follows: $\Delta(x+a)_k=k(x+a)_{k-1}$.  Hence it suffices to show
$$ \Delta^{i}H_{k,2l}(0)=0\, \mbox{   for }  0 \le i <2l , \quad \Delta^{2l}H_{k,2l}(z) = \frac{(k+l)!}{(k-l)!}\, , \quad 
\Delta^{i}H_{k,2l}(z) =0 \mbox{   for } i > 2l \, .$$
Indeed, one can easily check that for any $i\ge 0$,
$$ \Delta^{i}H_{k,2l}(z) = (k+l)_i \, \sum_{j=0}^k (-1)^j  \frac{(k-j)_l (z+k-j-l)_{k+l-i}}{(k-j)!\, j!}\, .$$
Then $ \Delta^{i}H_{k,2l}(0)=0$ for $0 \le i <2l$, since $(k-j-l)_{k+l-i}=0$.  Next, by taking derivatives from the expansion 
of $(z-1)^k$:
$$ \frac{d^i}{dz^i}(z-1)^k = k!\, \sum_{j=0}^{k}(-1)^j \frac{(k-j)_i}{(k-j)!\, j!}\, z^{k-j-i}, $$
and then evaluating at $z=1$, we obtain
$$ \sum_{j=0}^{k}(-1)^j \frac{j^i}{(k-j)!\, j!} = 0 \mbox{  for }  0 \le i < k \, \mbox{ and  }  \sum_{j=0}^{k}(-1)^{j+k} \frac{j^k}{(k-j)!\, j!} = 1\,. $$
From these identities one can 
easily derive $\Delta^{2l}H_{k,2l}(z) = \frac{(k+l)!}{(k-l)!}$ and $\Delta^{i}H_{k,2l}(z) =0$ for $ i > 2l$.
\end{proof}

\begin{proof}[Proof of Lemma~\ref{lemmaFk2l}]
We proceed by induction on $k$. For $k=2$, we have
$$F_{2,2l}(z)\, =\, z(z+2) \big( (1)_l(z)_l - (0)_l(z-1)_l\big) + (z+2)(1)_l(z)_l+z(0)_l(z-1)_l \, .$$
Then $F_{2,0}(z)=2(z+1)$ and $F_{2,2}(z)=(z+2)_3$, which is exactly our statement for $k=1$.
Assuming it is true for $k$, we show it for $k+1$. To this end one notes
$$F_{k+1,2l}(z)-F_{k,2l}(z) = (z+2k) \hat{H}_{k,2l}(z),$$
where  $\hat{H}_{k,2l}(z)$ is 
$$ \sum_{m=1}^k \, \big( z +2(k-m)\big)  \sum_{j=0}^{m} 
 \,  (-1)^{j}\, \frac{( z+ 2k-j-1)_{m-1}}{j!\, (m-j)!}  (k-j)_{l} \, (z+k-j-1)_{l}  +  (k)_l(z+k-1)_l .$$
Since $z+2(k-m)=(z+2k-j-m)-(m-j)$, we can express $\hat{H}_{k,2l}(z)$ as follows
\begin{eqnarray*}
\hat{H}_{k,2l}(z) &=& \sum_{m=1}^{k} \, \sum_{j=0}^m \, (-1)^j \frac{(z+2k-j-1)_m}{(m-j)! \, j!} (k-j)!(z+k-j-1)_l \\[0.1cm]
&&- \sum_{m=1}^{k} \, \sum_{j=0}^{m-1} \, (-1)^j \frac{(z+2k-j-1)_{m-1}}{(m-j-1)! \, j!} (k-j)_l(z+k-j-1)_l +  (k)_l(z+k-1)_l\\[0.1cm]
&=& \sum_{m=1}^{k} \, \sum_{j=0}^m \, (-1)^j \frac{(z+2k-j-1)_m}{(m-j)! \, j!} (k-j)_l(z+k-j-1)_l \\[0.1cm]
&&- \sum_{m=0}^{k-1} \, \sum_{j=0}^m \, (-1)^j \frac{(z+2k-j-1)_{m}}{(m-j)! \, j!} (k-j)_l(z+k-j-1)_l +  (k)_l(z+k-1)_l\\[0.1cm]
&=& \sum_{j=0}^{k}  \, (-1)^j \frac{(z+2k-j-1)_k}{(k-j)! \, j!} (k-j)_l(z+k-j-1)_l \, \\[0.1cm]
&=& \sum_{j=0}^{k}  \, (-1)^j \frac{(z+2k-j-1)_{k+l}}{(k-j)! \, j!} (k-j)_l\, .
\end{eqnarray*}
Now we see that $\hat{H}_{k,2l}(z) = H_{k,2l}(z+k+l-1)$ and hence by Proposition~\ref{prop2.2},
$$F_{k+1,2l}(z)=F_{k,2l}(z) + (z+2k) \frac{(k+l)!}{(2l)!(k-l)!} (z+k+l-1)_{2l}\, ,$$
and we are done once we use the induction assumption for $F_{k,2l}(z)$.
\end{proof}

\end{appendix}
 
\bibliographystyle{amsalpha}

\end{document}